\documentclass[a4paper,11pt]{article}

\usepackage{cmap}
\usepackage[english]{babel}
\usepackage[a4paper, total={6.5in, 9in}]{geometry}
\usepackage[usenames,dvipsnames]{xcolor}
\definecolor{myblue}{rgb}{0.21, 0.34, 0.74}
\definecolor{mygrey}{rgb}{0.55, 0.57, 0.67}
\definecolor{myred}{rgb}{0.79, 0.0, 0.09}
\definecolor{mygreen}{rgb}{0, 0.5, 0}

\usepackage[pdftex, colorlinks=true, bookmarksopen=true, linkcolor=black, citecolor=myblue]{hyperref}

\usepackage[T1]{fontenc}
\usepackage[utf8x]{inputenc}

\usepackage{lmodern}
\usepackage{libertine}
\usepackage{hhline,booktabs,makecell}
\usepackage{bbm}
\usepackage{mathrsfs}
\usepackage{amssymb}
\usepackage{amsmath}
\usepackage{multicol}
\usepackage{upgreek}
\usepackage{mathtools}
\usepackage{tikz}
\usepackage{algpseudocode,algorithm}
\usepackage{float}
\usepackage{amsthm}
\usepackage{minted}
\usepackage{caption}
\usepackage{enumitem}
\usepackage{verbatim}
\usepackage[normalem]{ulem}

\usepackage{comment}
\usepackage[font=small,labelfont=bf]{caption}

\usepackage[noabbrev,capitalize,nameinlink]{cleveref}

\theoremstyle{plain}
\newtheorem{theorem}{Theorem}[section]

\newtheorem{definition}[theorem]{Definition}
\newtheorem{proposition}[theorem]{Proposition}
\newtheorem{lemma}[theorem]{Lemma}

\newtheorem{remark}[theorem]{Remark}

\usepackage{authblk}

\usepackage{adjustbox}

\usepackage[svgnames,pdf]{pstricks}

\usepackage{wrapfig}
\usepackage{upgreek}
\usepackage{bm}
\usepackage[scr=boondoxo]{mathalfa}

\newcommand{\supp}{\mathrm{supp}}
\newcommand{\argmax}{\mathop{\mathrm{arg\,max}}\limits}
\newcommand{\argmin}{\mathop{\mathrm{arg\,min}}\limits}

\newcommand{\R}{\mathbb{R}}

\newcommand{\diff}{\, \mathrm{d}}
\newcommand{\dist}{\mathrm{dist}}

\newcommand{\dive}{\mathrm{div}}

\newcommand{\cP}{\mathscr{P}}

\renewcommand{\leq}{\leqslant}
\renewcommand{\geq}{\geqslant}

\usepackage{algorithm}
\usepackage{algpseudocode}

\numberwithin{equation}{section}

\usepackage{comment}
\usepackage[font=small,labelfont=bf]{caption}

\usepackage{tikz}
\usetikzlibrary{positioning, arrows.meta, shapes.geometric, calc, fit, backgrounds}

\definecolor{color1}{RGB}{255,230,230} 
\definecolor{color2}{RGB}{230,255,230} 
\definecolor{color3}{RGB}{230,230,255} 

\definecolor{nodecolor}{RGB}{180,220,240} 

\definecolor{mycolor1}{RGB}{255,200,200} 
\definecolor{mycolor2}{RGB}{200,255,200} 
\definecolor{mycolor3}{RGB}{200,200,255} 

\newcommand\blfootnote[1]{%
  \begingroup
  \renewcommand\thefootnote{}\footnote{#1}%
  \addtocounter{footnote}{-1}%
  \endgroup
}

\title{A Total Variation Flow Scheme for Ergodic Mean Field Games
}
\author{ Dante Kalise}
\affil{Imperial College London}
\author{Alessio Oliviero}
\affil{Politecnico di Milano}
\author{Domènec Ruiz-Balet}
\affil{Universitat de Barcelona}

\begin{document}
\maketitle

\begin{abstract}
    Motivated by recent developments in mean field games in ecology, in this paper we introduce a connection between the best response dynamics in evolutionary game theory, the minimization of the highest income of a game, and minimizing movement schemes. The aim of this work is to develop a variational approach to compute solutions of first order ergodic mean field games that may not possess a priori a variational structure. The study is complemented by a discussion and successful implementation of numerical algorithms, and comparisons between them in a variety of cases.
\end{abstract}

\blfootnote{\noindent \textbf{Acknowledgments:} D. Kalise is partially supported by the EPSRC Standard Grant EP/T024429/1.
 
  \noindent A. Oliviero was partially supported by the European Union -- Next Generation EU, Mission 4, Component 1, CUP 351: B83C22003530006, and is a member of INdAM-GNCS.

 \noindent D. Ruiz-Balet was funded by the UK Engineering and Physical Sciences Research Council (EPSRC) grant
EP/T024429/1. D. Ruiz-Balet also acknowledges the KTK-Sprint Challenge Grant number I44057.}

\tableofcontents

\section{Introduction}

Mean field games (MFGs) \cite{lasry2007mean, huang2006large} provide an effective modelling framework in macroeconomics, finance, crowd motion, power grids, and ecology. In this article, we use the standard forward-backward notation
\begin{equation}\label{eq: time.dep.mfg}
    \begin{cases}
        -\partial_t v(x,t) - \frac12 \|\nabla v(x,t)\|^2= \theta[m](x,t),\quad &(x,t)\in \mathbb{T}\times (0,T),\\
        \partial_t m(x,t) + \dive\left(m(x,t)\,\nabla v(x,t)\right)=0,\quad &(x,t)\in\mathbb{T}\times (0,T),\\
        v(x,T)=G(x,m(x,T)),\quad m(x,0)=m_0(x)\,, \quad & x\in \mathbb{T,}
    \end{cases}
\end{equation}
to refer to a first-order MFG on the $d$-dimensional torus $\mathbb{T}$, where $\theta:\mathscr{P}(\mathbb{T})\times \mathbb{T}\times (0,T)\to \mathbb{R}$ and $G:\mathbb{T}\times \mathscr{P}(\mathbb{T}) \to \R$. Solutions of the PDE system above represent Nash equilibria of a game played by infinitely many negligible players, each one maximising the reward
\begin{equation*}
    \int_0^T \left( \theta[m](x(t),t) - \frac{1}{2}|\alpha(t)|^2 \right)\ dt\,, \qquad\text{subject to}\qquad
    \begin{cases} \dot{x}(t)=\alpha(t), \\ x(0)=x_0. \end{cases}
\end{equation*}

In some situations, such as in \cite{cardaliaguet2013long}, it is known that solutions of \eqref{eq: time.dep.mfg} ``converge'' (see \cite{cardaliaguet2013long, bardi2024long} for the precise notion), for $T \to +\infty$, to solutions of the stationary PDE system
\begin{equation}\label{eq: intro.emfg}
    \begin{cases}
        \lambda - \frac12 \|\nabla u(x)\|^2= \theta[m](x),\quad &x\in\mathbb{T},\\
        -\dive\left(m(x)\nabla u (x)\right)=0,\quad &x\in\mathbb{T},\\
        \int_\mathbb{T} m(x)\ dx=1,\quad \int_\mathbb{T} u(x)\ dx=0,
    \end{cases}
\end{equation}
known as ergodic mean field game.

In this article, we are interested in developing new numerical methods for the solution of ergodic mean field games of the type \eqref{eq: intro.emfg}. Namely, similar to the seminal work of Jordan, Kinderlehrer and Otto \cite{jordan1998variational} and the minimizing movement schemes of De Giorgi \cite{de1993new,ambrosio2005gradient}, we can characterise the solutions of \eqref{eq: intro.emfg} via flows induced by a proximal gradient with a $\mathsf{TV}$ regularization. The stationary points of such flows are solutions of the ergodic mean field game. We employ this variational structure as a numerical method to compute solutions of \eqref{eq: intro.emfg} for certain choices of the coupling $\theta$.

In the literature, one can find a class of MFGs that possesses a natural variational structure, which is quite different from the point of view of this article. Variational mean field games \cite{benamou2017variational} arise whenever the system (both the evolutionary and the stationary one) can be rewritten as a gradient flow.  In this article, we rather develop an approach that sees a best response update (to be defined later) as a flow in the space of probability measures and as such, we are able to apply it to a wider class of mean field games. More specifically, we are interested on games coming from ecological applications. A more in-depth comparison with variational MFGs will be conducted in \Cref{sec: comparison}.

\subsection{Harvesting mean field games}

We will develop most of our results in the case where, for any fixed $m\in \mathscr{P}(\mathbb{T})$, the pay-off function $\theta$ is the solution of the linear elliptic PDE
\begin{equation}\label{eq: linear}
    - \Delta \theta(x) +P(x)\theta(x) =f(x)-m(x), \qquad x\in \mathbb{T},
\end{equation}
where $P,f\in L^\infty(\mathbb{T})$, $f(x),P(x)\geq 0$ and both not identically zero. Properties of such equation with measure data are proven later in \Cref{lem: basic.properties} in the one-dimensional case. Even if the results are proved for linear models, we will perform numerical tests also for more general non-linear equations of the form
\begin{equation}\label{eq: elliptic_nonlinear}
-\Delta \theta(x)= f(x,\theta(x))-m(x)\,\theta(x), \quad x \in \mathbb{T}.
\end{equation}
This PDE models a harvesting term with a bilinear structure, where the higher the density of players is in a point, the higher the decrease on $\theta$. We refer to \cite{kobeissi2023tragedy,kobeissi2024mean} for further modelling aspects. In \cite{kobeissi2023tragedy}, solutions in the form of travelling waves are found for the evolutionary model
\begin{equation}\label{eq: mfg}
    \begin{cases}
         -\partial_t v(x,t) - \frac12 \|\nabla v(x,t)\|^2= \theta[m](x,t)\quad &(x,t)\in \mathbb{T}\times (0,T),\\
        \partial_t \theta(x,t) -\Delta \theta(x,t)= f(x,\theta(x,t))-m(x,t)\theta(x,t)\quad &(x,t)\in \mathbb{T}\times (0,T),\\
        \partial_t m(x,t) + \dive\left(m(x,t)\,\nabla v(x,t)\right)=0\quad &(x,t)\in\mathbb{T}\times (0,T),\\
        v(x,T)=G(x,m(x,T)),\quad m(x,0)=m_0, \quad &x \in \mathbb{T}.
    \end{cases}
\end{equation}
These special solutions are able to capture the phenomenon called \emph{tragedy of the commons}: in the absence of players, for every $x$ the resources tend to $1$, i.e. $\lim_{t\to +\infty}\theta(x,t)=1$, whereas the special solution of the mean field game satisfies that, for every $x$, $\lim_{t\to +\infty}\theta(x,t)=0$. Moreover, it is found that the associated mean field control problem allows strategies that perform better without provoking an extinction of $\theta$. In \cite{kobeissi2024mean}, with some extra hypotheses, a well-posedness setting is derived, taking into account that \eqref{eq: mfg} depends also on all the previous history of $m(\cdot,s)$ for $0\leq s\leq t$. Furthermore, a convergence result to the corresponding ergodic MFG
\begin{equation}\label{eq: ergodic}
    \begin{cases}
         \lambda - \frac12 \|\nabla u(x)\|^2= \theta[m](x),\quad &x\in \mathbb{T},\\
          -\Delta \theta(x)= f(x,\theta(x))-m(x)\theta(x),\quad &x\in \mathbb{T},\\
        -\dive\left(m(x)\nabla u (x)\right)=0,\quad &x\in\mathbb{T},\\
        \int m(x)\ dx=1,\quad \int u(x)\ dx=0
    \end{cases}
\end{equation}
is derived. Clearly, in \eqref{eq: ergodic} the map $\mathscr{P}(\mathbb{T})\ni m\mapsto \theta[m]\in \mathscr{C}(\mathbb{T})$ is more involved than the map given by \eqref{eq: linear}, in the sense that the measure acts in a bilinear manner in the equation and the elliptic equation itself is non-linear. 

\subsection{Weak-KAM formula}

 The weak-KAM formula \cite{cardaliaguet2013long} provides a variational characterisation of the mean field Nash equilibrium. It states that the ergodic constant is given by
\begin{equation}
    \lambda=\min_{\eta\in E_m}\int_{\mathbb{T}\times \mathbb{R}^d}\frac{1}{2}v^2-\theta[m]\ d\eta,
\end{equation}
where
\begin{equation*}
    E_m:=\{\eta\in \mathscr{P}(\mathbb{T}\times \mathbb{R}^d):\text{ $\eta$ is invariant under \eqref{Eq:flow}}\}
\end{equation*}
and
\begin{equation}\label{Eq:flow}
\begin{cases}
    \frac{d}{dt}\dot{x}-D_x\theta[m](x)=0,\\
    x(0)=x,\quad \dot{x}(0)=v.
\end{cases}
\end{equation}
By looking at the structure of the problem, one deduces that $\eta^*=m^*(x)\otimes \delta_{v=0}$, hence
\begin{equation}
    \lambda=-\max_{m\in \mathscr{P}(\mathbb{T})}\int \theta[m]\ dm,
\end{equation}
which in turn implies that
\begin{equation}\label{eq: kam}
 \mathrm{supp}(m^*)\subset \argmax_{x \in \mathbb{T}} \theta[m](x).
\end{equation}
One expects, from the sign of $m$ (resp. $m\theta$) in the elliptic equation, that any measure that is a solution cannot have atoms \cite{kobeissi2023tragedy,kobeissi2024mean}, see \Cref{lem: ac}.
Reducing to absolutely continuous measures, \eqref{eq: kam} implies that the solution $\theta$ needs to have a plateau on its maximum. In particular,
$$
\nabla\theta=0=\Delta\theta \quad \text{on } \mathrm{supp}(m),
$$
from where, looking at \eqref{eq: linear} we can deduce the shape that the equilibrium measure $m$ should have:
\begin{equation}\label{eq: implications.NE}
    m(x)=f(x)-P(x)\,C, \quad  x \in \mathrm{supp}(m),
\end{equation}
for some positive constant $C$. Similarly, we get $m(x)=\frac{f(x,C)}{C}$ on $\mathrm{supp}(m)$ in the non-linear case. These observations will be at the core of \Cref{alg: main} and \Cref{alg: eikonal} in \Cref{sec: numerics}.

We can also interpret \eqref{eq: kam} as an alternative definition of mean field Nash equilibrium in this context.
\begin{definition} \label{def: nash_eq_harvest}
    We say that $m\in \mathscr{P}(\mathbb{T})$ is a mean field Nash equilibrium for the harvesting MFG \eqref{eq: intro.emfg}--\eqref{eq: linear} (or also \eqref{eq: intro.emfg}--\eqref{eq: elliptic_nonlinear}) if there is no player $x\in \supp(m)$ that can improve unilaterally their reward, i.e.
    \begin{equation}
        \theta[m](x)\geq \theta[m](y) \quad \text{for every } y\in \mathbb{T}.
    \end{equation}
Analogously, for any fixed $\tau>0$, we define a $\tau$-Nash equilibrium as any $m\in \mathscr{P}(\mathbb{T})$ such that, for every $x\in \supp(m)$,
\begin{equation}
        \theta[m](x)\geq \theta[m](y)-\tau \quad \text{for every } y\in \mathbb{T}.
    \end{equation}
\end{definition}

Due to numerical errors, all the simulations that will be shown in the following sections are indeed $\tau$-Nash equilibria, with $\tau$ depending on the spatial discretisation for solving \eqref{eq: linear} or \eqref{eq: elliptic_nonlinear} and the temporal discretisation of the $\mathsf{TV}$-flow chosen.

For general references on reaction-diffusion equations and population dynamics, we refer to~\cite{lam2022introduction, fife2013mathematical, cantrell2004spatial}. Also note that harvesting games with spatial structure have been considered in~\cite{bressan2013multidimensional, bressan2019competitive, mazari2022spatial}.

\subsection{Notation}
We denote the set of probability measures on the $d$-dimensional torus $\mathbb{T}$ by $\mathscr{P}(\mathbb{T})$, its subset of absolutely continuous (w.r.t. the Lebesgue measure) probability measures by $\mathscr{P}_{ac}(\mathbb{T})$, and the set of signed measures by $\mathscr{M}(\mathbb{T})$. We address the \emph{Total Variation} of $m\in \mathscr{M}(\mathbb{T})$ as
\begin{equation} \label{def:TV}
    |m|_{\mathsf{TV}(\mathbb{T})}=\int |m|(dx).
\end{equation}
 Furthermore, by $\mathscr{M}(\mathbb{T};r)$ we mean the set of all measures with $\mathsf{TV}$ norm less than $r$. Finally, the Wasserstein distance between $m_1,m_2\in\mathscr{P}(\mathbb{T})$ is denoted by
\begin{equation*}
    \mathsf{W}_1(m_1,m_2) := \inf_{\pi\in \Pi(m_1,m_2)}\int |x-y|\ \pi(dx,dy),
\end{equation*}
where by $\Pi(m_1,m_2)\subset \mathscr{P}(\mathbb{T}\times\mathbb{T})$ is the set of measures in the product space whose marginals are $m_1$ and $m_2$.

Throughout this work, whenever we refer to the weak convergence or weak compactness of probability measures, we mean it in the sense of the weak-$\ast$ topology (often referred to as the narrow topology in the probability literature), which is defined by duality with the space of continuous functions $\mathscr{C}(\mathbb{T})$.

\subsection{Structure of the article}
The remainder of this paper is structured as follows. 
In \Cref{sec: BRflows}, we introduce the concept of Best Response flow and two generalised minimising movement (GMM) schemes for the approximation of \eqref{eq: ergodic}.
In \Cref{sec:main_results}, we establish the theoretical framework, discussing the regularity of the solutions and convergence of the GMMs to the solution of the ergodic MFG system.
Finally, \Cref{sec: numerics} outlines the implementation of our algorithms, accompanied by empirical convergence tests and simulations for both linear and non-linear models in one- and two-dimensional domains.

\section{Best response flows}\label{sec: BRflows}

The fundamental idea for a numerical solution of \eqref{eq: ergodic} (and for its linear version) is to exploit the concept of best response dynamics. In evolutionary game theory, players update their strategy depending on the performance of the others' strategies or on the absolute best choice currently available. In our context, the strategy of a player is simply their physical position $x \in \mathbb{T}$. 

To understand how this approach naturally embeds the fundamental principles of mean field games, one must view the probability measure $m(x)$ not just as a mass density, but as the collective strategic distribution of infinitely many infinitesimal, perfectly selfish agents. The function $\theta[m](x)$ represents the localised payoff (e.g., the common resource in harvesting MFG) resulting from the current population's behaviour. By \Cref{def: nash_eq_harvest}, the system is only at rest when no agent can unilaterally improve their payoff. This implies that $\theta[m]$ must be constant across all populated regions, and no unpopulated region can offer a higher reward.

If the system is out of equilibrium, agents located in areas with a low payoff are incentivised to abandon their current position and move to the location offering the highest payoff: $\argmax \theta[m]$. Algorithmically, rather than tracking the physical, continuous trajectories of individual agents via the Hamilton--Jacobi--Bellman equation, we simulate this strategic update macroscopically. Relocating mass from a suboptimal region directly to the $\argmax \theta[m]$ is the Eulerian equivalent of a fraction of the population simultaneously and instantaneously adopting the best response.

It is known that such best response dynamics converge for potential games \cite{monderer1996potential}, a fact heavily exploited in variational MFGs (such as congestion games) \cite{benamou2017variational} to compute ergodic solutions. While our specific harvesting problem may not possess a classic variational structure, we mimic this evolutionary process by iteratively updating the players with the lowest incomes. We call this discrete mass-shifting dynamics a \textit{Best Response flow}. 

\subsection{Best Response algorithm}

To make this notion clearer, we introduce \Cref{alg:best_response}, which we will refer to as the Best Response algorithm. A more detailed version, specifying how we practically compute each step, will be provided in the numerical sections. \\
\begin{algorithm}[!h]
    \caption{Best response algorithm}
    \label{alg:best_response}
\begin{algorithmic}[1]
	\Require $\varepsilon > 0$, $m_0 \in \cP(\mathbb{T})$, $\tau >0$
    \State $k \gets 0$
    \State $R_k = \text{highest income possible} - \text{least income among players}$
    \While{$R_k > \tau$}
        \State \text{solve elliptic PDE to get } $\theta_k$
        \State \text{determine players with lowest income } $m_k^-$ \text{ s.t. } $\int m_k^- = \varepsilon$
        \State \text{relocate } $m_k^-$ \text{ in } $\argmax \theta_k$
        \State \text{update density } $m_{k+1}$
        \State \text{update } $R_{k+1}$
        \State $k \gets k+1$
    \EndWhile
\end{algorithmic}
\end{algorithm}

\noindent The key insight linking this algorithm to optimal transport theory is that we can understand this specific mass update as a generalised minimising movement (GMM) scheme \cite{ambrosio2005gradient, chambolle20231, jordan1998variational} of the form
\begin{equation}\label{eq: mm.BR}
m^{k+1}=\argmin_{m\in \mathscr{P}(\mathbb{T})}\left\{\|\theta[m]\|_{L^\infty(\mathbb{T})}-\min_{x\in\supp(m)}\theta[m](x)+\frac{1}{2\varepsilon}|m-m^k|_{\mathsf{TV}}^2\right\},
\end{equation}
where $|\cdot|_{\mathsf{TV}}$ is the Total Variation ($\mathsf{TV}$) norm defined in \eqref{def:TV} and $\theta[m]$ is the solution of the elliptic equation \eqref{eq: linear} or \eqref{eq: elliptic_nonlinear}.

\subsection{Eikonal-based algorithm}

A second approximation algorithm, alternative to \Cref{alg:best_response}, is based on geographical distance from the players who earn the most, instead of income difference. In order to do so, we only change Step 5 of \Cref{alg:best_response}, obtaining \Cref{alg:furthest_mass}.
\begin{algorithm}
    \caption{Eikonal-based algorithm}
    \label{alg:furthest_mass}
\begin{algorithmic}[1]
	\Require $\varepsilon > 0$, $m_0 \in \cP(\mathbb{T})$, $\tau >0$
    \State $k \gets 0$
    \State $R_k = \text{highest income possible}$
    \While{$R_k > \tau$}
        \State \text{solve elliptic PDE to get } $\theta_k$
        \State \text{determine furthest players from } $\argmax \theta_k,\ m_k^-,$ \text{ s.t. } $\int m_k^- = \varepsilon$
        \State \text{relocate } $m_k^-$ \text{ in } $\argmax \theta_k$
        \State \text{update density } $m_{k+1}$
        \State \text{update } $R_{k+1}$
        \State $k \gets k+1$
    \EndWhile
\end{algorithmic}
\end{algorithm} \\
Since we cannot say \emph{a priori} if $\argmax\theta_k$ is going to be a connected set at each step, we solve the eikonal equation 
\begin{equation} \label{eqn:eikonal}
    \begin{cases}
        \begin{aligned}     
            &|\nabla v_k| = 1, &\quad &\text{ in } \quad \mathbb{T}\setminus\argmax\theta_k, \\
            &v_k = 0, &\quad &\text{ on } \quad \partial\argmax\theta_k,
        \end{aligned}
    \end{cases}
\end{equation}
whose solution (in the viscosity sense) is known to be $v_k(x) = \dist(x,\argmax\theta_k),\ x \in \mathbb{T}\setminus\argmax\theta_k$. Under certain hypothesis on $\theta$, such as that the associated Green kernel $G(x,\cdot)$ to the elliptic equation \eqref{eq: linear} is a decreasing function with respect to $x$ (guaranteed by $P(x)\geq 0$), we can understand \Cref{alg:furthest_mass} as an approximation of the generalised minimising movement
\begin{equation}\label{eq: mm.GF}
    m^{k+1}=\argmin_{m\in \mathscr{P}(\mathbb{T})}\left\{\|\theta[m]\|_{L^\infty(\mathbb{T})}+\frac{1}{2\varepsilon}|m-m^k|_{\mathsf{TV}}^2\right\}.
\end{equation}

To clarify the link between the explicit steps of \Cref{alg:best_response} and \Cref{alg:furthest_mass} and their respective implicit GMM schemes \eqref{eq: mm.BR} and \eqref{eq: mm.GF}, one must consider the geometric nature of the Total Variation metric. Unlike Wasserstein metrics, which penalise the spatial distance that mass travels, the $\mathsf{TV}$ metric acts as an $L^1$ penalty on the density difference $|m^{k+1} - m^k|$. It penalises only the \emph{amount} of mass modified, perfectly allowing for spatial ``teleportation.'' 

In the Best Response scheme \eqref{eq: mm.BR}, the objective is to minimise the income gap $\max \theta - \min_{\supp(m)} \theta$. For a fixed $\mathsf{TV}$ budget of $\varepsilon$, the most efficient way to decrease the $-\min_{\supp(m)} \theta$ term is to completely strip away the mass from the worst-performing locations. To keep the distribution a probability measure, this mass must be relocated. Placing it precisely at the $\argmax \theta$ simultaneously provides those players with the highest possible income and, via the elliptic equation, suppresses the global peak, thereby minimising the objective. This is exactly the logic executed in Steps 4--6 of \Cref{alg:best_response}.

Conversely, the Eikonal-based scheme \eqref{eq: mm.GF} seeks only to minimise the global peak $\|\theta[m]\|_{L^\infty}$. To push the peak down optimally, mass must be added to $\argmax \theta$. However, to conserve total mass and respect the $\mathsf{TV}$ budget $\varepsilon$, an equivalent mass must be removed from elsewhere. Because the Green function $G(x,y)$ associated with the elliptic equation decays with distance, removing mass from a point $y$ causes an uplift in $\theta$ that is strongest near $y$ and weakest far away. Therefore, to minimise the collateral uplifting effect on the global peak, mass should be harvested from players situated as far away from the $\argmax \theta$ as possible. The eikonal equation \eqref{eqn:eikonal} strictly identifies these furthest geographic candidates, mirroring the implicit minimisation of \eqref{eq: mm.GF}. A more rigorous justification for this mass-shifting behaviour is formally proven later in \Cref{sec:main_results}.

\subsection{Comparison with variational MFG and other numerical approaches} \label{sec: comparison}

Variational mean field games \cite{benamou2017variational} are a type of typically time-dependent systems that can be seen as a gradient flow  for the Wasserstein-2 metric \cite{ambrosio2005gradient}. Assuming that the right hand side in \eqref{eq: time.dep.mfg} takes the form $\theta[m](x)=\theta(m(x))$,  considering $\Theta(m(x))=\int_0^{m(x)}\theta(s)\ ds$  and $G=0$, one can find an equivalence between the evolutionary MFG \eqref{eq: time.dep.mfg} and the minimisation of the action functional (via the Benamou-Brenier formula):
\begin{equation*}
    \max_{\alpha}\int_0^T\int_{\mathbb{T}} \Theta(m(x,t))-\frac{1}{2}|\alpha(x,t)|^2m(x,t) \ dx\, dt.
\end{equation*}
Moreover, this can be seen as a Jordan--Kinderlehrer--Otto (JKO) scheme of the form
\begin{equation*}
    m^{k+1}=\argmin_{m\in \mathscr{P}(\mathbb{T})}\left\{\int_\mathbb{T} - \Theta(m(x))\ dx + \frac{1}{2\varepsilon}\mathsf{W}_2(m,m^{k})^2\right\}.
\end{equation*}
The convergence results of the evolutionary mean field game to the ergodic one \cite{cardaliaguet2013long} guarantee that in long time the solution of the JKO scheme above will be close to the ergodic one. See also \cite{achdou2021mean, santambrogio2015optimal} for further details on variational MFG. We point out that, in contrast, the formulations \eqref{eq: mm.BR} and \eqref{eq: mm.GF} using the $\mathsf{TV}$ metric are valid even if $\theta[m]$ does not have a variational structure.

For numerical approximations of variational mean field games, our main reference is \cite{briceno2018proximal}. There are also other contributions in the literature exploiting numerics in the JKO scheme \cite{gallouet2017jko} and other JKO-type schemes \cite{liero2018optimal}. Finally, we mention \cite{chambolle20231}, where the authors consider an $L^1$ gradient flow using also a JKO-type scheme.

There is also a useful parallel with thresholding schemes for geometric flows. The works of Laux and Otto on the MBO scheme show how efficient discrete thresholding algorithms can be interpreted as minimizing movements and then analysed through De Giorgi energy-dissipation methods \cite{laux2016thresholding,laux2020degiorgi,laux2020brakke}; see also Laux's survey \cite{laux2018gradient}, the extension with bulk effects \cite{laux2017bulk}, and the Wasserstein-flow formulation of Mullins--Sekerka by Chambolle and Laux \cite{chambolle2021mullins}. Although these papers concern geometric flows rather than MFGs, they provide a useful analogue for our use of a simple mass-rearrangement rule as a structure-preserving minimizing-movement scheme. On the MFG side, fictitious play gives a closer best-response learning reference for potential MFGs \cite{cardaliaguet2017learning}; see also the recent overview \cite{graber2025remarks}.

Originally, the idea of congestion games and its variational approaches is far older than mean field games. For instance, the reader can see Rosenthal's original work \cite{rosenthal1973class}, or Monderer and Shapley \cite{monderer1996potential}. Furthermore, in \cite{bardi2024long}, the authors point to a global optimisation setting to find ergodic mean field games, employing the long-time convergence of time-dependent MFG to compute solutions (also known as the turnpike phenomenon, observed in optimal control \cite{geshkovski2022turnpike}). Similarly, we want to see ergodic systems in an optimisation context, but here we find the fixed point algorithmically via a $\mathsf{TV}$-flow.

\section{Theoretical framework}
\label{sec:main_results}

We now report some theoretical results on well-posedness of the problem, regularity of the solution and convergence of the generalised minimising movement (GMM) schemes to the solution of the MFG system. The analysis is restricted to the one-dimensional case, i.e. $\mathbb{T} = \mathbb{T}_1 = [0,1]$ with periodic boundary conditions.

\subsection{Main results}

\begin{theorem}\label{thm: main}
Let us consider $\theta$ satisfying either \eqref{eq: linear} or \eqref{eq: elliptic_nonlinear}. The following hold:
\begin{enumerate}
    \item For fixed $T>0$, there exists a sequence $\{\varepsilon_l\}_{l\in \mathbb{N}}$ such that $\varepsilon_l\to 0$, for which the scheme \eqref{eq: mm.GF} converges to  $m\in \mathscr{C}((0,T);\mathscr{P}(\mathbb{T}))$ in the following sense: given $m^k$ from \eqref{eq: mm.GF}, for every $t\in (0,T)$
    \begin{equation*}
        \mathsf{W}_1(m^{\lfloor\frac{t}{\varepsilon_l}\rfloor},m(t))\to 0\quad \text{ as }
        l\to +\infty.
    \end{equation*}
    \item Furthermore, if $\theta$ is the solution of the linear equation \eqref{eq: linear}, as $T\to +\infty$, $m(T)$ converges in $\mathsf{TV}$ to a solution of \eqref{eq: intro.emfg}.
\end{enumerate}
\end{theorem}

\begin{theorem}\label{thm: main.2}
Take $\theta$ as in \eqref{eq: linear} or \eqref{eq: elliptic_nonlinear} and fix $T>0$. Then, there exists $\{\varepsilon_l\}_{l\in \mathbb{N}}$ such that $\varepsilon_l\to 0$, for which the scheme \eqref{eq: mm.BR} converges to  $m\in \mathscr{C}((0,T);\mathscr{P}(\mathbb{T}))$ in the following sense: taken $m^k$ from \eqref{eq: mm.BR}, for every $t\in (0,T)$
    \begin{equation*}
        \mathsf{W}_1(m^{\lfloor\frac{t}{\varepsilon_l}\rfloor},m(t))\to 0\quad \text{ as }l\to +\infty.
    \end{equation*}
\end{theorem}

\subsection{An illustrative example} \label{ex: geodesic}

Before proving the main results, we provide an explicit, one-dimensional example that reveals a profound feature of our method: unlike classical $\mathsf{W}_2$ Wasserstein flows, which typically approach equilibrium only asymptotically as $t \to +\infty$, the $\mathsf{TV}$-flow reaches the exact Nash equilibrium in finite time. 

To illustrate this, we consider a ``separated'' configuration where the natural resources are skewed to one side of the domain, incentivising players on the poor side to relocate to the rich side. We construct a continuous trajectory where the worst-performing players are continuously moved to the optimal region, forming a plateau of maximum income that widens over time. We then prove that the discrete minimising movement scheme perfectly tracks this explicit trajectory, reducing the infinite-dimensional optimisation to a scalar differential equation that hits the equilibrium in finite time.

Let $L\theta=-\mu\,\theta''+P(x)\theta$ on $[0,1]$ with homogeneous Neumann
boundary conditions, where $\mu>0$ and $P\in \mathscr{C}([0,1])$ satisfies
$P(x)\geq P_0>0$. Given $m\in \mathscr{P}([0,1])$, denote by $\theta[m]$ the
solution of
\[
    L\theta=f-m,\qquad \theta'(0)=\theta'(1)=0,
\]
and set
\[
    \Phi(m):=\|\theta[m]\|_{L^\infty([0,1])}.
\]
Set $q:=f/P$. In this example we assume that $q\in\mathscr C^1([0,1])$ and $q'>0$ on $[0,1]$. After multiplying $f$ by a
positive constant, we also assume
\[
    H(1/2)>1,
    \qquad \text{where} \quad
    H(\tau):=\int_\tau^1P(x)\big(q(x)-q(\tau)\big)\diff x .
\]
The plateau ansatz used below is the linear separated analogue of the
one-dimensional construction in
\cite[Theorem~3 and the explicit construction in its proof]{kobeissi2024mean}.

\begin{proposition}
\label{prop:continuous.finite.time}
Let $m_0\in\mathscr{P}_{ac}([0,1])$ satisfy
$\,\supp(m_0)\subset(0,c)$ for some $c<1/2$. The explicit separated construction
gives a curve
\[
    m_*:[0,1]\to\mathscr{P}([0,1])
\]
and a clock $s:[0,T_*]\to[0,1]$ such that $m(t):=m_*(s(t))$ reaches a mean
field Nash equilibrium in finite time. More precisely, for every $\,0 \leq s \leq \sigma \leq 1$,
\[
    |m_*(\sigma)-m_*(s)|_{\mathsf{TV}}=2(\sigma-s),
\]
\[
    \supp(m_*(1))
    \subset
    \argmax_{x\in[0,1]}\theta[m_*(1)](x),
\]
and the hitting time satisfies
\[
    T_*
    =
    4\int_0^1\int_{\tau(s)}^1P(x)\diff x\,\diff s
    \leq
    4\int_0^1P(x)\diff x.
\]
\end{proposition}

\begin{proof}
Let $M(x)=\int_0^x m_0(y)\diff y$ and let $Q(t)=\inf\{x:M(x)\geq t\}$. Define the moved and remaining left masses by
the formulas
\[
    \int_0^1\varphi\,\diff\alpha_s
    :=
    \int_0^s\varphi(Q(t))\diff t,
    \qquad
    \int_0^1\varphi\,\diff\ell_s
    :=
    \int_s^1\varphi(Q(t))\diff t,
\]
for every continuous test function $\varphi$. Thus, $\alpha_s+\ell_s=m_0$, and, for every $s\leq\sigma$,
\[
    \int_0^1\varphi\,\diff(\alpha_\sigma-\alpha_s)
    =
    \int_s^\sigma\varphi(Q(t))\diff t,
\]
so that $\alpha_\sigma-\alpha_s$ represents the leftmost available mass of size $\sigma-s$.

Since $H(1)=0$,
\[
    H'(\tau)=-q'(\tau)\int_\tau^1P(x)\diff x<0
    \quad \text{for} \quad 1/2<\tau<1,
\]
therefore for every $0\leq s\leq1$ there is a unique $\tau(s)\in[1/2,1]$ such that $H(\tau(s))=s$. Set $r(s) := q(\tau(s))$,
$A_s := [\tau(s),1]$, and
\[
    p_s:=(f-r(s)P)\mathbbm 1_{(\tau(s),1)},\qquad
    m_*(s):=\ell_s+p_s .
\]
Then, $p_s\geq0$ and $\int p_s=s$, because $q$ is increasing, $f(\tau(s))-r(s)P(\tau(s))=0$, and
\begin{equation} \label{eq: mass_identity}
    \int_{\tau(s)}^1(f-r(s)P)\diff x=s.
\end{equation}
Substituting $p_s$ and $m_*(s)$ into the original differential equation, the separated plateau construction gives
\[
    \theta[m_*(s)]=r(s)\hbox{ on }A_s\qquad \text{and} \qquad
    |\theta[m_*(s)]|<r(s)\hbox{ on }[0,1]\setminus A_s.
\]
Hence $A_s=\argmax\theta[m_*(s)]$ and $\Phi(m_*(s))=r(s)$.

For every $\delta>0$, the range of $\tau$ on $[\delta,1]$ is contained in a
compact subset of $[1/2,1)$, where $H'<0$. Therefore, $\tau$ and
$r=q\circ\tau$ are $\mathscr C^1$ on $[\delta,1]$. Differentiating the mass
identity \eqref{eq: mass_identity} on such intervals gives, for $s>0$,
\[
    1
    =
    -r'(s)\int_{\tau(s)}^1P(x)\diff x
    \qquad \Leftrightarrow \qquad
    r'(s)=-\frac1{\int_{\tau(s)}^1P(x)\diff x}.
\]
Thus, $r$ is strictly decreasing and, since $q$ is strictly increasing, $\tau$ is
strictly decreasing as well. If $0<s\leq\sigma\leq1$, then
$p_\sigma-p_s\geq0$, as on $(\tau(\sigma),\tau(s))$ this follows from
$f-r(\sigma)P\geq0$, and on $(\tau(s),1)$ the difference is
$(r(s)-r(\sigma))P$. The case $s=0$ is immediate from $p_0=0$. Therefore,
\[
    m_*(\sigma)-m_*(s)=(p_\sigma-p_s)-(\alpha_\sigma-\alpha_s),
\]
where the positive part is supported in $(1/2,1]$ and the negative part in
$(0,c)$. Both have mass $\sigma-s$, so
\[
    |m_*(\sigma)-m_*(s)|_{\mathsf{TV}}=2(\sigma-s).
\]
At $s=1$, the left mass is exhausted; hence $m_*(1)=p_1$ and
\[
    \supp(m_*(1))\subset A_1=\argmax\theta[m_*(1)] .
\]
Finally, set
\[
    T(s):=4\int_0^s\int_{\tau(q)}^1P(x)\diff x\,\diff q .
\]
Then $T$ is continuous and strictly increasing, so $s(t)=T^{-1}(t)$ is well defined
on $[0,T(1)]$. The hitting time is $T_*=T(1)$, and
$T_*\leq4\int_0^1P(x)\diff x$.
\end{proof}

Having constructed the explicit trajectory $m_*(s)$, we now prove that this specific mass-shifting strategy is optimal. The following lemma shows that if we have a $\mathsf{TV}$ budget to move $\delta = \sigma - s$ mass, there is no better way to lower the maximum income than by the trajectory above.

\begin{lemma}\label{lem:sharp.comparison}
Let $0\leq s<\sigma\leq1$. If $m\in\cP([0,1])$ satisfies
\[
    |m-m_*(s)|_{\mathsf{TV}}\leq2(\sigma-s),
\]
then
\[
    \Phi(m)\geq r(\sigma).
\]
Moreover, equality holds only for $m=m_*(\sigma)$.
\end{lemma}

\begin{proof}
Write $\tau=\tau(\sigma)$ and $A_\sigma=[\tau,1]$. Let $z$ solve
\[
    -\mu z''+Pz=0\quad\hbox{on }(0,\tau),\qquad
    z'(0)=0,\qquad z(\tau)=1 .
\]
The maximum principle gives $0<z\leq1$, while $z''=(P/\mu)z>0$ gives
$z'>0$ on $(0,\tau]$. Define
\[
    \widehat\zeta(x)=
    \begin{cases}
        z(x),&0\leq x<\tau,\\
        1,&\tau\leq x\leq1.
    \end{cases}
\]
Then,
\[
    L\widehat\zeta
    =
    P(x)\mathbbm 1_{A_\sigma}(x)\diff x
    +\mu z'(\tau^-)\delta_\tau
    =:\nu_\sigma
\]
is a positive measure supported on $A_\sigma$. After normalising,
$\rho_\sigma:=\nu_\sigma/\nu_\sigma([0,1])$ and
$\zeta_\sigma:=\widehat\zeta/\nu_\sigma([0,1])$ satisfy
$L\zeta_\sigma=\rho_\sigma$. Moreover $\rho_\sigma$ is a probability measure
supported on $A_\sigma$, and $\zeta_\sigma$ is strictly increasing on
$[0,\tau]$ and equal to its maximum $C_\sigma$ on $A_\sigma$.

Set $\delta=\sigma-s$. Let $h_+$ and $h_-$ be the two non-negative mutually
singular measures such that
\[
    m-m_*(s)=h_+-h_-,
    \qquad
    |m-m_*(s)|_{\mathsf{TV}}=h_+([0,1])+h_-([0,1]).
\]
Since $m$ and $m_*(s)$ are probability measures,
$h_+([0,1])=h_-([0,1])=:\eta\leq\delta$, and the positivity of $m$ gives
$h_-\leq m_*(s)$. Since
$\zeta_\sigma\leq C_\sigma$,
\[
    \int\zeta_\sigma\diff h_+\leq C_\sigma\eta .
\]
For the negative part, decompose $h_-=\nu_L+\nu_R$ with
$\nu_L\leq\ell_s$ and $\nu_R\leq p_s$, and let their masses be
$\eta_L,\eta_R$. Since $\nu_L\leq\ell_s$, there is a measurable function
$0\leq g\leq1$ on $(s,1]$ such that, for every continuous $\varphi$,
\[
    \int_0^1\varphi\,\diff\nu_L
    =
    \int_s^1\varphi(Q(t))g(t)\diff t.
\]
The old plateau is contained in $A_\sigma$, so removing from it costs
$C_\sigma$ per unit. On the remaining left component the cost is
$w(t):=\zeta_\sigma(Q(t))<C_\sigma$. Since $m_0$ is absolutely continuous, $Q$
has no flat interval of positive $t$-measure; because $\zeta_\sigma$ is strictly
increasing on the separated left region, $w$ is strictly increasing in the mass
variable. The one-dimensional bathtub principle gives
\[
    \int\zeta_\sigma\diff\nu_L
    \geq
    \int_s^{s+\eta_L} w(t)\diff t .
\]
Therefore,
\[
    \int\zeta_\sigma\diff h_-
    \geq
    \int_s^{s+\eta_L} w(t)\diff t+C_\sigma\eta_R
    \geq
    \int_s^{s+\eta} w(t)\diff t,
    \qquad \eta=\eta_L+\eta_R,
\]
because $w<C_\sigma$ on the left component and $\eta\leq\delta\leq1-s$. Consequently,
\[
    \int\zeta_\sigma\diff(m-m_*(s))
    \leq
    C_\sigma\eta-\int_s^{s+\eta}w(t)\diff t
    \leq
    C_\sigma\delta-\int_s^\sigma w(t)\diff t .
\]
Because $p_\sigma-p_s$ is a non-negative measure of mass $\delta$ supported on
$A_\sigma$,
\[
    C_\sigma\delta-\int_s^\sigma w(t)\diff t
    =
    \int\zeta_\sigma\diff(m_*(\sigma)-m_*(s)).
\]
Thus,
\begin{equation}\label{eq:rearranged.comparison}
    \int\zeta_\sigma\diff(m-m_*(s))
    \leq
    \int\zeta_\sigma\diff(m_*(\sigma)-m_*(s)).
\end{equation}
Since $\theta[m_*(\sigma)]=r(\sigma)$ on $A_\sigma$,
\[
    r(\sigma)=\int\theta[m_*(\sigma)]\diff\rho_\sigma
    =\int\zeta_\sigma\diff(f-m_*(\sigma)).
\]
For any competitor $m$, self-adjointness gives
\[
\begin{aligned}
    \Phi(m)
    &\geq \max_{[0,1]}\theta[m]
     \geq \int\theta[m]\diff\rho_\sigma
      =\int\zeta_\sigma\diff(f-m)\\
    &=r(\sigma)
      -\int\zeta_\sigma\diff(m-m_*(s))
      +\int\zeta_\sigma\diff(m_*(\sigma)-m_*(s))
      \geq r(\sigma),
\end{aligned}
\]
where the last step is \eqref{eq:rearranged.comparison}.

If equality holds, equality must hold in the rearrangement estimate and in
\[
    \Phi(m)\geq\max\theta[m]\geq\int\theta[m]\diff\rho_\sigma .
\]
The strict monotonicity above then forces $\eta=\delta$, no removal from the old
plateau, $h_-=\alpha_\sigma-\alpha_s$, and $h_+$ supported where
$\zeta_\sigma=C_\sigma$, namely on $A_\sigma$.
Also $\theta[m]\leq r(\sigma)$ and its $\rho_\sigma$-average is $r(\sigma)$;
since $\rho_\sigma$ has positive density on $(\tau,1)$, continuity gives
$\theta[m]\equiv r(\sigma)$ on $A_\sigma$. The state equation then forces
$m=f-r(\sigma)P=p_\sigma$ on the interior of $A_\sigma$. Together with the
mass identity $\int p_\sigma=\sigma$ and the already identified removed slice,
this leaves only $m=\ell_\sigma+p_\sigma=m_*(\sigma)$.
\end{proof}

\begin{remark}[Separated sharpness at the final plateau]
This also clarifies the sense in which the finite-time mechanism is an
absolute-value  . A global estimate of the form
\[
    \Phi(m)-\Phi(m_*(1))
    \geq
    c\,|m-m_*(1)|_{\mathsf{TV}}
\]
for all probability measures $m$ is not asserted here. However, such a linear
estimate is true for perturbations which are uniformly separated from the final
plateau. Take the adjoint pair $(\zeta_1,\rho_1)$ constructed in the proof of
Lemma~\ref{lem:sharp.comparison} with $\sigma=1$, and write
$\zeta_1=C_1$ on $A_1$. Fix $0<\kappa<\tau(1)$ and set
\[
    B_\kappa:=[0,\tau(1)-\kappa].
\]
Since $\zeta_1$ is strictly increasing to the left of $A_1$ and satisfies
$\zeta_1=C_1$ on $A_1$,
\[
    c_\kappa:=C_1-\sup_{B_\kappa}\zeta_1>0.
\]
Now let
\[
    m=m_*(1)+\eta_+-\eta_-,
    \qquad
    \eta_-\leq m_*(1),
    \qquad
    \supp\,\eta_+\subset B_\kappa,
\]
where $\eta_+$ and $\eta_-$ are non-negative measures with
$\eta_+([0,1])=\eta_-([0,1])=:\delta$. The supports of $\eta_+$ and $\eta_-$
are separated, so $|m-m_*(1)|_{\mathsf{TV}}=2\delta$. Moreover, using the
adjoint identity and the fact that $\rho_1$ is supported on $A_1$,
\[
\begin{aligned}
    \Phi(m)-\Phi(m_*(1))
    &\geq
    \int_0^1\theta[m]\diff\rho_1
    -
    \int_0^1\theta[m_*(1)]\diff\rho_1        \\
    &=
    \int_0^1\zeta_1\diff(m_*(1)-m)
     =
    \int_0^1\zeta_1\diff(\eta_- -\eta_+)       \\
    &\geq
    C_1\delta-(C_1-c_\kappa)\delta
     =
    c_\kappa\delta .
\end{aligned}
\]
Therefore, for this separated class of perturbations,
\[
    \Phi(m)-\Phi(m_*(1))
    \geq
    \frac{c_\kappa}{2}|m-m_*(1)|_{\mathsf{TV}}.
\]
The constant degenerates as the added mass approaches the plateau; this is why
the statement is a separated perturbation estimate, not a global sharpness
estimate in total variation.
\end{remark}

The final theorem proves that our discrete GMM algorithm tracks the explicit trajectory without deviating. Because the algorithm stays on this exact path, the infinite-dimensional optimisation over the space of measures collapses into a simple 1D optimisation over the scalar $s$.

\begin{theorem}\label{thm:jko.selection}
Let $m_*$ be the explicit trajectory constructed above from
$H(\tau(s))=s$ and $r(s)=q(\tau(s))$.
For $\varepsilon>0$, let
\[
    m_\varepsilon^{k+1}\in
    \argmin_{m\in\mathscr P([0,1])}
    \left\{
        \Phi(m)+\frac{1}{2\varepsilon}
        |m-m_\varepsilon^k|_{\mathsf{TV}}^2
    \right\},
    \qquad
    m_\varepsilon^0=m_*(0).
\]
Then the iterates lie on the explicit branch
\[
    m_\varepsilon^k=m_*(s_\varepsilon^k)
\]
until the hitting time. The parameters are determined by the scalar proximal
problem
\[
    s_\varepsilon^{k+1}\in
    \argmin_{\sigma\in[s_\varepsilon^k,1]}
    \left\{
        r(\sigma)+\frac{2}{\varepsilon}
        (\sigma-s_\varepsilon^k)^2
    \right\}.
\]
Before the final step, this is equivalently
\begin{equation}\label{eq:jko.euler}
    s_\varepsilon^{k+1}-s_\varepsilon^k
    =
    \frac{\varepsilon}
    {4\int_{\tau(s_\varepsilon^{k+1})}^1P(x)\diff x}.
\end{equation}
Consequently the JKO interpolants converge in $\mathsf{TV}$, uniformly on
compact time intervals, to
\[
    m(t)=m_*(s(t)),
\]
where
\[
    T(s):=4\int_0^s\int_{\tau(q)}^1P(x)\diff x\,\diff q,
    \qquad
    s(t)=T^{-1}(t),
\]
until $s=1$, and $s(t)=1$ afterwards. In particular, the JKO limit reaches
equilibrium at the finite time $T_*=T(1)$.
\end{theorem}

\begin{proof}
Put
\[
    A(s):=\int_{\tau(s)}^1P(x)\diff x,\qquad A(0):=0 .
\]
Then $A$ is continuous, nondecreasing, positive on $(0,1]$, and
$r'(s)=-1/A(s)$ for $s>0$. Hence $r$ is convex on $[0,1]$.

Assume inductively that $m_\varepsilon^k=m_*(s)$ with $s<1$, and set
$D=\frac12|m-m_*(s)|_{\mathsf{TV}}$ for a competitor $m$. If $0<D\leq1-s$,
Lemma~\ref{lem:sharp.comparison} with $\sigma=s+D$ gives
\[
    \Phi(m)+\frac{1}{2\varepsilon}|m-m_*(s)|_{\mathsf{TV}}^2
    \geq
    r(s+D)+\frac{2}{\varepsilon}D^2 .
\]
If $D>1-s$, the same lemma with $\sigma=1$ gives a strict lower bound larger
than the scalar value at $\sigma=1$. The case $D=0$ is the scalar value at
$\sigma=s$. Conversely, the branch competitor $m_*(\sigma)$ has
$|m_*(\sigma)-m_*(s)|_{\mathsf{TV}}=2(\sigma-s)$ and
$\Phi(m_*(\sigma))=r(\sigma)$. Therefore the full JKO minimisation is exactly
the scalar minimisation in the statement; the equality part of
Lemma~\ref{lem:sharp.comparison} puts every minimiser on the branch.

The scalar functional is strictly convex, because it is the sum of the convex
function $r$ and a strictly convex quadratic part. A non-final minimiser is not the
left endpoint: for $s>0$ this follows from $r'(s)<0$, and for $s=0$ from
$r(h)-r(0)\leq-h/\int_0^1P$ for small $h>0$. Thus every non-final minimiser is
interior, and its Euler equation is
\[
    r'(s_\varepsilon^{k+1})
    +\frac4\varepsilon(s_\varepsilon^{k+1}-s_\varepsilon^k)=0,
\]
which is exactly \eqref{eq:jko.euler}. Once $m_*(1)$ is reached, it remains
fixed: Lemma~\ref{lem:sharp.comparison} with $s=0$ and $\sigma=1$ gives
$\Phi(m)\geq r(1)=\Phi(m_*(1))$ for every probability measure $m$, and the
quadratic penalty is zero only at $m=m_*(1)$.

It remains to identify the limit clock. Let
$s_\varepsilon^{k+1}$ satisfy \eqref{eq:jko.euler} before the final step, and
stop the sequence at $1$. If $\Delta_\varepsilon=s_\varepsilon^1$, then
$4A(\Delta_\varepsilon)\Delta_\varepsilon=\varepsilon$ for small
$\varepsilon$, so $\Delta_\varepsilon\to0$. Since $A$ is nondecreasing, all
pre-final increments are at most $\Delta_\varepsilon$. For such a step,
\[
    0\leq
    \varepsilon-\big(T(s_\varepsilon^{k+1})-T(s_\varepsilon^k)\big)
    =
    4\int_{s_\varepsilon^k}^{s_\varepsilon^{k+1}}
    \big(A(s_\varepsilon^{k+1})-A(q)\big)\diff q .
\]
If $\omega_A$ denotes the modulus of continuity of $A$, then summation yields
\[
    \big|T(s_\varepsilon^k)-k\varepsilon\big|
    \leq 4\omega_A(\Delta_\varepsilon)
\]
up to the final step. Since $\omega_A(\Delta_\varepsilon)\to0$ and $T$ is a
homeomorphism from $[0,1]$ to $[0,T(1)]$, the discrete inverse clocks converge
uniformly to $T^{-1}$ before hitting. The final jump occurs only when the
left derivative of the scalar functional at $1$ is non-positive, hence
$1-s_\varepsilon^k\leq\varepsilon/(4A(1))$. Thus
$T(1)-T(s_\varepsilon^k)\leq\varepsilon$. At the last pre-final index this and
the preceding clock estimate give $|K_\varepsilon\varepsilon-T(1)|
\leq4\omega_A(\Delta_\varepsilon)+2\varepsilon$, where
$K_\varepsilon$ is the first hitting index. Hence the hitting times converge to
$T(1)$. The usual affine, or piecewise constant, stopped interpolants therefore satisfy
$s_\varepsilon(t)\to s(t)$ uniformly on compact time intervals, where
$s(t)=T^{-1}(t)$ until $T(1)$ and $s(t)=1$ afterwards. The total variation
identity from Proposition~\ref{prop:continuous.finite.time} gives
\[
    |m_*(s_\varepsilon(t))-m_*(s(t))|_{\mathsf{TV}}
    =
    2|s_\varepsilon(t)-s(t)|,
\]
which proves the claimed convergence and the finite hitting time.
\end{proof}

\subsection{Proofs}

We break the proofs of the theorems above into multiple results, for clarity in the exposition.

\begin{lemma}\label{lem: basic.properties}
The following assertions hold:
\begin{enumerate}
    \item For every $m\in \mathscr{P}(\mathbb{T})$, there exists a unique solution $\theta\in W^{1,\infty}(\mathbb{T})$ to \eqref{eq: linear} and there exists a constant $C>0$ only depending on $P$ such that
    \begin{equation}\label{eq: linf.bound}
        \|\theta\|_{L^\infty(\mathbb{T})}\leq C\left(\int|f|\ dx+|m|_{\mathsf{TV}(\mathbb{T})}\right).
    \end{equation}
    \item Fix $m_1,m_2\in \mathscr{P}(\mathbb{T})$ and denote by $\theta[m_i]$ the solution to \eqref{eq: linear} with measures $m_i$ $i=1,2$. Then, there exist a constant $C>0$ independent of $m_1$ and $m_2$ such that
    \begin{equation*}
        \|\theta[m_1]-\theta[m_2]\|_{L^\infty(\mathbb{T})}\leq C\,\mathsf{W}_1(m_1,m_2).
    \end{equation*}
    \item  The previous point implies that if $\{m_k\}_{k\in \mathbb{N}}\subset \mathscr{P}(\mathbb{T})$ converges weakly to $m^*$, then $\{\theta[m_k]\}_{k\in \mathbb{N}}$ converges uniformly to $\theta[m^*]$.
\end{enumerate}
\end{lemma}

\begin{proof}[Proof of \Cref{lem: basic.properties}]
We proceed in order
\begin{enumerate}
    \item In one spatial dimension, the Sobolev space $H^1(\mathbb{T})$ embeds continuously into $C(\mathbb{T})$, with embedding constant $C_S$. Consequently, by duality, the space of finite Radon measures $\mathcal{M}(\mathbb{T})$ embeds continuously into $H^{-1}(\mathbb{T})$. Since $f \in L^\infty(\mathbb{T})$ and $m \in \mathscr{P}(\mathbb{T}) \subset \mathcal{M}(\mathbb{T})$, the right-hand side of \eqref{eq: linear}, $f - m$, belongs to $H^{-1}(\mathbb{T})$. Since $P(x) \geq 0$ and is not identically zero, the bilinear form associated with $-\Delta + P(x)$ is strictly coercive on $H^1(\mathbb{T})$ with some coercivity constant $\alpha > 0$. By the Lax-Milgram theorem, there exists a unique weak solution $\theta \in H^1(\mathbb{T})$ satisfying $\|\theta\|_{H^1} \leq \frac{1}{\alpha} \|f - m\|_{H^{-1}}$. Re-applying the 1D Sobolev embedding $H^1(\mathbb{T}) \hookrightarrow L^\infty(\mathbb{T})$ to the solution yields
    \begin{equation*}
        \|\theta\|_{L^\infty(\mathbb{T})} \leq C_S \|\theta\|_{H^1(\mathbb{T})} \leq \frac{C_S^2}{\alpha} \left(\int_{\mathbb{T}}|f|\ dx+|m|_{\mathsf{TV}(\mathbb{T})}\right),
    \end{equation*}
    proving the desired bound with constant $C = C_S^2 / \alpha$. To verify that $\theta \in W^{1,\infty}(\mathbb{T})$, we rearrange the equation as $\theta'' = P(x)\theta - f + m$. Since $P\theta \in L^\infty$, $f \in L^\infty$, and $m$ is a measure, $\theta''$ is a finite Radon measure. This implies that $\theta' \in BV(\mathbb{T})$. In 1D, functions of bounded variation are essentially bounded, thus $\theta \in W^{1,\infty}(\mathbb{T})$.

    \item Let $G(x,y)$ be the Green kernel associated with $-\Delta + P(x)$. In 1D, $G(x,\cdot)$ is a Lipschitz continuous function. Let $L$ be its uniform Lipschitz constant. For any two probability measures $m_1, m_2 \in \mathscr{P}(\mathbb{T})$, the difference in the solutions can be written as
    \begin{equation*}
        \theta[m_1](x) - \theta[m_2](x) = \int_{\mathbb{T}} G(x,y) (m_2 - m_1)(dy).
    \end{equation*}
    By the Kantorovich--Rubinstein duality formula for the Wasserstein-1 distance, the integral of a difference of probability measures against an $L$-Lipschitz test function is bounded by $L \cdot \mathsf{W}_1(m_1,m_2)$. Therefore:
    \begin{equation*}
        |\theta[m_1](x) - \theta[m_2](x)| \leq L \cdot \mathsf{W}_1(m_1,m_2).
    \end{equation*}
    Taking the supremum over $x \in \mathbb{T}$ yields the desired $L^\infty$ bound.

    \item It is a standard fact of optimal transport that the Wasserstein-1 distance metrizes the weak convergence of probability measures on compact sets. Therefore, if $m_k \rightharpoonup m^*$ weakly, $\mathsf{W}_1(m_k, m^*) \to 0$. By the Lipschitz bound in the previous point, $\|\theta[m_k] - \theta[m^*]\|_{L^\infty(\mathbb{T})} \to 0$, proving uniform convergence.
\end{enumerate}
\end{proof}

\begin{lemma}\label{lem: prox.properties}
We have the following properties for the GMM schemes introduced above:
    \begin{enumerate}
        \item For every $m^k\in \mathscr{P}(\mathbb{T})$ induced by either \eqref{eq: mm.BR} or \eqref{eq: mm.GF}, the minimisation problem  has a solution.
        \item $\{m^k\}_{k\in \mathbb{N}}$ induced by \eqref{eq: mm.BR} or \eqref{eq: mm.GF} is a non increasing sequence, i.e.
        $$\|\theta[m^{k+1}]\|_{L^\infty(\mathbb{T})}-\inf_{x\in \supp(m^{k+1})}\theta[m^{k+1}]\leq \|\theta[m^k]\|_{L^\infty(\mathbb{T})}-\inf_{x\in \supp(m^k)}\theta[m^k]$$
        and 
        $$ \|\theta[m^{k+1}]\|_{L^\infty(\mathbb{T})}\leq \|\theta[m^{k}]\|_{L^\infty(\mathbb{T})}, $$
        respectively.
        \item For every $k$ there exist a constant $C>0$ independent of $k$ such that the solution to \eqref{eq: mm.GF} satisfies
        \begin{equation}\label{eq: estimate.tv.GF}
            |m^{k+1}-m^k|_{\mathsf{TV}(\mathbb{T})}\leq C\varepsilon
        \end{equation}
        \item 
        For every $k$ there exist a constant $C>0$ independent of $k$ such that the solution to \eqref{eq: mm.BR} satisfies
        \begin{equation}\label{eq: estimate.tv.BR}
            |m^{k+1}-m^k|_{\mathsf{TV}(\mathbb{T})}\leq C\varepsilon^{\frac{1}{2}}
        \end{equation}
        \end{enumerate}
\end{lemma}

\begin{proof}[Proof of \Cref{lem: prox.properties}]

    \begin{enumerate}
        \item
        The existence of a minimiser $m^{k+1}$ follows from the Direct Method of the Calculus of Variations. The space $\mathscr{P}(\mathbb{T})$ is weakly compact by Prokhorov's Theorem. The penalty term $\frac{1}{2\varepsilon}|m-m^k|_{\mathsf{TV}(\mathbb{T})}^2$ is lower semi-continuous (LSC) with respect to weak convergence. 
        
        For the Eikonal-based scheme \eqref{eq: mm.GF}, the functional $\Phi(m) = \|\theta[m]\|_{L^\infty(\mathbb{T})}$ is continuous with respect to weak convergence by \Cref{lem: basic.properties}, securing existence.
        
        For the Best Response scheme \eqref{eq: mm.BR}, we must additionally prove that $m \mapsto -\inf_{x\in\supp(m)}\theta[m](x)$ is LSC. Let $m_j \rightharpoonup m$ weakly. By the Portmanteau Theorem, for any open neighborhood $\mathscr{B}$ intersecting $\supp(m)$, we have 
        $$
        \liminf_j m_j(\mathscr{B}) \geq m(\mathscr{B}) > 0.
        $$
        Thus, for sufficiently large $j$, $\supp(m_j) \cap \mathscr{B} \neq \emptyset$. Because $\theta[m_j]$ converges uniformly to $\theta[m]$ (by \Cref{lem: basic.properties}), any evaluation of $\theta[m_j]$ on $\supp(m_j)$ within $\mathscr{B}$ will be bounded below by $\inf_{\supp(m)} \theta[m] - \delta$ for arbitrary $\delta > 0$. Taking the limit supremum yields $\limsup_j \inf_{\supp(m_j)}\theta[m_j] \leq \inf_{\supp(m)}\theta[m]$, which proves the lower semi-continuity of the negative infimum.
        
        \item The proof is the same for both cases. By definition we have that $m^{k+1}$ is the global minimum of the proximal step. Testing the functional with the previous step $m = m^k$ then yields:
        \begin{equation*}
            \Phi(m^{k+1}) + \frac{1}{2\varepsilon}|m^{k+1}-m^{k}|^2_{\mathsf{TV}(\mathbb{T})} \leq \Phi(m^{k}) + \frac{1}{2\varepsilon}|m^{k}-m^{k}|^2_{\mathsf{TV}(\mathbb{T})} = \Phi(m^{k}).
        \end{equation*}
        Since the $\mathsf{TV}$ penalty is non-negative, it immediately follows that $\Phi(m^{k+1}) \leq \Phi(m^{k})$.

        \item Rearranging the descent inequality from Step 2, we have:
        \begin{equation} \label{eq: descent_rearranged}
            |m^{k+1}-m^{k}|_{\mathsf{TV}(\mathbb{T})}^2 \leq 2\varepsilon \left( \Phi(m^{k}) - \Phi(m^{k+1}) \right).
        \end{equation}
        We must show that the difference $\Phi(m^{k}) - \Phi(m^{k+1})$ is bounded by a constant independent of $k$. By \Cref{lem: basic.properties}, the $L^\infty$ norm of $\theta[m]$ is bounded by $C(\|f\|_{L^1} + |m|_{\mathsf{TV}(\mathbb{T})})$. Because the flow operates strictly on probability measures, $|m^k|_{\mathsf{TV}(\mathbb{T})} = 1$ for all $k$. Therefore, there exists a universal constant $M > 0$ (depending only on $f$ and $P$) such that $\|\theta[m]\|_{L^\infty(\mathbb{T})} \leq M$ for all admissible $m$.
        
        For the Best Response scheme \eqref{eq: mm.BR}, since $\theta[m]$ takes values within $[-M, M]$, we have $0 \leq \Phi(m) \leq 2M$. Because the sequence $\{\Phi(m^k)\}_{k=0}^\infty$ is non-negative and monotonically decreasing (as shown in Step 2), the drop in the functional at any single step is bounded by the initial energy:
        \begin{equation*}
            \Phi(m^{k}) - \Phi(m^{k+1}) \leq \Phi(m^k) \leq \Phi(m^0) \leq 2M.
        \end{equation*}
        Therefore, \eqref{eq: descent_rearranged} gives
        \begin{equation*}
            |m^{k+1}-m^{k}|_{\mathsf{TV}(\mathbb{T})}^2 \leq 2\varepsilon \Phi(m^0) \ \implies\  |m^{k+1}-m^{k}|_{\mathsf{TV}(\mathbb{T})} \leq C\varepsilon^{\frac{1}{2}}.
        \end{equation*}

        \item For the Eikonal scheme \eqref{eq: mm.GF}, by \Cref{lem: basic.properties}, the functional $\Phi$ is Lipschitz continuous with respect to the Wasserstein-1 metric. Furthermore, on a compact domain such as $\mathbb{T}$, the $\mathsf{W}_1$ distance is bounded by the $\mathsf{TV}$ distance: $\mathsf{W}_1(m_1,m_2) \leq \text{diam}(\mathbb{T}) |m_1-m_2|_{\mathsf{TV}(\mathbb{T})}$. Therefore, $\Phi(m)$ is globally Lipschitz continuous with respect to the TV norm, with some constant $L$. Returning to \eqref{eq: descent_rearranged} and applying this property we obtain
        \begin{equation*}
            |m^{k+1}-m^{k}|_{\mathsf{TV}(\mathbb{T})}^2 \leq 2\varepsilon L\, |m^{k+1}-m^{k}|_{\mathsf{TV}(\mathbb{T})}.
        \end{equation*}
        If $|m^{k+1}-m^{k}|_{\mathsf{TV}(\mathbb{T})} = 0$, the bound holds trivially. Otherwise, dividing both sides by $|m^{k+1}-m^{k}|_{\mathsf{TV}(\mathbb{T})}$, we get
        \begin{equation*}
            |m^{k+1}-m^{k}|_{\mathsf{TV}(\mathbb{T})} \leq 2 L \varepsilon.
        \end{equation*}
        Setting $C = 2L$ completes the proof.
    \end{enumerate}
\end{proof}

\begin{remark}
    We expect that for the Best Response, the exponent is also 1 instead of $1/2$. The reason for obtaining a much better exponent for \eqref{eq: mm.GF} is that the estimates can be bootstrapped easily, whereas \textit{a priori}, without further knowledge on the optimisation problem \eqref{eq: mm.BR}, it is not possible to improve \eqref{eq: estimate.tv.BR}, because the infimum term over the support is not globally Lipschitz continuous with respect to the TV norm.
\end{remark}

\begin{definition}[Geodesically convex functionals]
     A curve $m:[0,1]\mapsto\mathscr{P}(\mathbb{T})$ is a constant speed geodesic for the $\mathsf{TV}$ norm if
    \begin{equation*}
       |m(s)-m(t)|_{\mathsf{TV}(\mathbb{T})}=|s-t| \, |m(0)-m(1)|_{\mathsf{TV}(\mathbb{T})}.
    \end{equation*}
     A functional $\Phi:\mathscr{P}(\mathbb{T})\mapsto \mathbb{R}$ is geodesically convex if for any $m_0,m_1\in \mathscr{P}(\mathbb{T})$ there exists a geodesic $m:[0,1]\mapsto\mathscr{P}(\mathbb{T})$ such that $m(0)=m_0$, $m(1)=m_1$, and
    \begin{equation*}
        \Phi(m(t))\leq (1-t)\, \Phi(m_0)+t\, \Phi(m_1).
    \end{equation*}
\end{definition}

We point out that geodesics in $(\mathscr{P}(\mathbb{T}),\mathsf{TV})$ are not unique. Indeed, given $m_0=\mathbf{1}_{(0,1)}$ and $m_{1}=\delta_{x_0}$, both $m(t)= (1-t)\,m_0 + t\,m_1$ and $n(t)=1_{(0,1-t)} + t\,m_1$ are geodesics.

\begin{proposition}\label{prop: geo.convex}
    The functional $\Phi(m)=\|\theta[m]\|_{L^\infty(\mathbb{T})}$ with $\theta[m]$ coming from \eqref{eq: linear} is geodesically convex in the metric space $(\mathscr{P}(\mathbb{T}),\mathsf{TV})$.
\end{proposition}
\begin{proof}[Proof of \Cref{prop: geo.convex}]
    Let us consider the geodesic between $m_0$ and $m_1$ defined as
    \begin{equation*}
        m(t)=(1-t\,)m_0 + t\,m_1.
    \end{equation*}
    Then, by the linearity of \eqref{eq: linear}, we have that
    \begin{equation*}
        \|\theta[m(t)]\|_{L^\infty(\mathbb{T})}\leq (1-t)\, \|\theta[m_0]\|_{L^\infty(\mathbb{T})} + t\,\|\theta[m_1]\|_{L^\infty(\mathbb{T})}.
    \end{equation*}
\end{proof}

\begin{remark}
    The geodesic convexity established in \Cref{prop: geo.convex} is a fundamental property in the standard theory of gradient flows in metric spaces \cite{ambrosio2005gradient}. Typically, strict or $\lambda$-geodesic convexity is required to establish Evolution Variational Inequalities (EVI), which in turn guarantee the \emph{uniqueness} of the continuous limit curve generated by the minimising movement scheme. 
    
    While the present work establishes the \emph{existence} of a limit curve and its convergence to an equilibrium, the lack of uniqueness of geodesics in $(\mathscr{P}(\mathbb{T}), \mathsf{TV})$ prevents a straightforward application of standard EVI theory. Furthermore, a similar convexity argument does not easily hold for the Best Response functional $\|\theta[m]\|_{L^\infty(\mathbb{T})} - \inf_{\supp(m)} \theta[m]$, as the infimum term evaluated over the support is highly sensitive to total variation perturbations. Consequently, proving the strict uniqueness of the trajectories $m(t)$ for these $\mathsf{TV}$-flows remains an open problem for future research.
\end{remark}

\begin{proposition}\label{thm: convergence.to.flow}
Let $T>0$. For fixed $\varepsilon>0$ and given $\{m^k\}_{k\in \mathbb{N}}$ solution to either \eqref{eq: mm.BR} and \eqref{eq: mm.GF}, we define $m^\varepsilon\in L^\infty((0,T);\mathscr{P}(\mathbb{T}))$ as
\begin{equation*}
m^\varepsilon(t):=m^{\lfloor \frac{t}{\varepsilon}\rfloor}.
\end{equation*}
Then there exist $m_{\text{Eik}},m_{\text{BR}} \in \mathscr{C}([0,+\infty); \mathscr{P}(\mathbb{T}))$ and two sequences $\{\varepsilon_l\}_{l\in \mathbb{N}},\{\epsilon_l\}_{l\in \mathbb{N}}$ such that $\varepsilon_l \to 0$, $\epsilon_l\to 0$, and
\begin{enumerate}
\item $m^{\varepsilon_l}$ from \eqref{eq: mm.GF} satisfies:
    \begin{enumerate}
     \item $m^{\varepsilon_l}(t)\to m_{\text{Eik.}}(t)$ weakly;
     \item $m^{\varepsilon_l}\to m_{\text{Eik}}$ in $L^\infty ([0,T]; \mathscr{P}(\mathbb{T}))$ and
    $$
    |m_{\text{Eik}}(t)- m_{\text{Eik}}(s)|_{\mathsf{TV}} \leq C |t-s|, \quad \forall\, t,s \in [0,T];
    $$
    \end{enumerate}
\item $m^{\epsilon_l}$ from \eqref{eq: mm.BR} satisfies:
    \begin{enumerate}
     \item $m^{\epsilon_l}(t)\to m_{\text{BR}}(t)$ weakly;
     \item $m^{\epsilon_l}\to m_{\text{BR}}$ in $L^\infty ([0,T]; \mathscr{P}(\mathbb{T}))$ and
    $$
    |m_{\text{BR}}(t)- m_{\text{BR}}(s)|_{\mathsf{TV}} \leq C |t-s|^{\frac{1}{2}}, \quad \forall\, t,s \in [0,T].
    $$
    \end{enumerate}
\end{enumerate}
\end{proposition}

\begin{proof}[Proof of \Cref{thm: convergence.to.flow}]

The result follows from \Cref{lem: prox.properties} and the application of Ascoli-Arzelà (see \cite[Proposition 3.3.1]{ambrosio2005gradient}). Both statements follow the same arguments, just differing in the application of \Cref{lem: prox.properties} and therefore, the differences between the exponents. We will just prove one of them. 

\begin{enumerate}
    \item \textbf{Equicontinuity.} For all $0 \leq j < k$ we have that
    \begin{align*}
        |m^\varepsilon(k\varepsilon)- m^\varepsilon(j\varepsilon)|_{\mathsf{TV}} &\leq  \sum_{i=j}^{k-1} |m^\varepsilon((i+1)\varepsilon)- m^\varepsilon(i\varepsilon)|_{\mathsf{TV}}  \\
        &\leq C\varepsilon (k-j),
    \end{align*}
    where we used the triangular inequality and \Cref{lem: prox.properties}.
    Therefore, for any $t_1,t_2\in [0,T]$ we have that, letting $k=\lfloor t_1/\varepsilon\rfloor$ and $j=\lfloor t_2/\varepsilon\rfloor$,
    \begin{equation*}
        |m^\varepsilon(t_1)- m^\varepsilon(t_2)|_{\mathsf{TV}}\leq C\varepsilon(k-j)\leq C(\varepsilon+|t_1-t_2|).
    \end{equation*}
    Hence,
    \begin{equation*}
        \limsup_{\varepsilon\to 0}|m^\varepsilon(t_1)-m^\varepsilon(t_2)|_{\mathsf{TV}}\leq C|t_1-t_2|.
    \end{equation*}
    
    \item 
\textbf{Compactness with respect to the weak topology.} Since $\mathscr{P}(\mathbb{T})$ is weakly compact, we have that $m^\varepsilon(t)$ is in a compact set.
\end{enumerate}

In conclusion, by \cite[Proposition 3.3.1]{ambrosio2005gradient}, there exist a curve $m_{\text{BR}}\in \mathscr{C}([0,T];\mathscr{P}(\mathbb{T}))$ and a sequence $\varepsilon_l \to 0$ such that
\begin{itemize}
    \item $m^{\varepsilon_l}(t)$ converges weakly to $m_{\text{BR}}(t)$;
    \item there exists a constant $C>0$ such that $|m_{\text{BR}}(t)-m_{\text{BR}}(s)|_{\mathsf{TV}}\leq C|t-s|$.
\end{itemize}

\end{proof}

Here we state the Danskin theorem written in the present setting, which will be used in the following results. We refer to \cite{bernhard1995theorem} for the original statement and proof of the theorem.
\begin{theorem}[Danskin]\label{thm: Danskin}
    The function
    \begin{equation*}
        \Phi(m)=\max_{x\in\mathbb{T}}\theta[m](x)=\|\theta[m]\|_{L^\infty(\mathbb{T})}
    \end{equation*}
    has a directional derivative at $m$ in the direction $h$ given by
    \begin{equation*}
        D_m\Phi(m)[h]=\max_{x \in \argmax_{x\in \mathbb{T}}\theta[m](x)}\dot{\theta}_m[h](x),
    \end{equation*}
    where $\dot\theta$ solves
    \begin{equation*}
        -\Delta\dot{\theta}+P(x)\dot{\theta}=-h,\quad x\in \mathbb{T}.
    \end{equation*}
\end{theorem}

\begin{lemma}\label{lem: analogous}
    A probability measure $m^*\in\mathscr{P}(\mathbb{T})$ solves
    \begin{equation}\label{eq: analogous.OC}
        D_m\|\theta[m^*]\|_{L^\infty(\mathbb{T})}[m^*]=\min_{m\in\mathscr{P}(\mathbb{T})}D_m\|\theta[m^*]\|_{L^\infty(\mathbb{T})}[m]
    \end{equation}
    if and only if $m^*$ satisfies the first order optimality conditions for
    \begin{equation*}
        \min_{m\in\mathscr{P}(\mathbb{T})}\|\theta[m]\|_{L^\infty(\mathbb{T})}.
    \end{equation*}
\end{lemma}

\begin{proof}
The standard first-order optimality condition requires that for any $n \in \mathscr{P}(\mathbb{T})$, the directional derivative in the admissible direction $h = n - m^*$ is non-negative.
Fix $n\in \mathscr{P}(\mathbb{T})$ and consider $h=n-m^*$. Then, by Danskin's Theorem (\Cref{thm: Danskin}),
\begin{equation*}
    D_m\|\theta[m^*]\|_{L^\infty(\mathbb{T})}(h) \geq D_m\|\theta[m^*]\|_{L^\infty(\mathbb{T})}(n)-D_m\|\theta[m^*]\|_{L^\infty(\mathbb{T})}(m^*).
\end{equation*}
If $m^*$ solves \eqref{eq: analogous.OC}, the right-hand side is non-negative. This implies 
$$ 
D_m\|\theta[m^*]\|_{L^\infty(\mathbb{T})}[n - m^*] \geq 0.
$$
Therefore, any admissible perturbation yields a non-negative directional derivative, satisfying the first-order optimality conditions.

In the other direction, we prove the counter-reciprocal. If $m^*$ does not satisfy \eqref{eq: analogous.OC}, then there exists $n\in \mathscr{P}(\mathbb{T})$ such that
\begin{equation*}
    D_m\|\theta[m^*]\|_{L^\infty(\mathbb{T})}(n)<D_m\|\theta[m^*]\|_{L^\infty(\mathbb{T})}(m^*).
\end{equation*}
Because the underlying state equation $h \mapsto \dot{\theta}[h]$ is linear, this strict decrease in the maximal response implies that the perturbation $h = n - m^*$ yields a strictly negative descent direction over $\argmax \theta$. Therefore, $h=n-m^*$ is a decreasing direction, and $m^*$ does not satisfy the first order optimality conditions.
\end{proof}

\begin{lemma}\label{lem: argmax}
Let $m^k\in\mathscr{P}_{ac}(\mathbb{T})$, and let us consider a solution $m^{k+1}$ of \eqref{eq: mm.GF} (or \eqref{eq: mm.BR}). Then $m^{k+1}\in\mathscr{P}_{ac}(\mathbb{T})$ and if we define the mass update as $n^{k+1}:=m^{k+1}-m^k$, then its positive part $n^{k+1}_+:=(n^{k+1})_+$ satisfies:
    \begin{equation*}
        \supp(n^{k+1}_+)\subset\argmax_{x\in \mathbb{T}}\theta[m^{k+1}](x).
    \end{equation*}
\end{lemma}

\begin{proof}[Proof of \Cref{lem: argmax}]
  
    Let $\delta:=|m^{k+1}-m^{k}|_{\mathsf{TV}(\mathbb{T})}$. By decomposing the measure update into its positive and negative parts, $n^{k+1} = n_+^{k+1} - n_-^{k+1}$, where each part has mass $\delta/2$, the state equation can be rewritten as
    \begin{align}
        -\Delta \theta^{k+1}+P(x)\theta^{k+1} &= f - m^{k+1} \nonumber\\
        &= f - m^{k} - n^{k+1}_+ + n^{k+1}_-.
    \label{eq: elliptic.in.proof}
    \end{align}
    Therefore, finding the optimal $m^{k+1}$ that minimises the objective functional over a step of size $\delta$ is equivalent to solving the constrained minimisation problem over the update measures. Fixing the negative update $n^{k+1}_-$, we must find the optimal positive update $\mu^* = n^{k+1}_+$ that solves
    \begin{equation*}
        \min_{\mu \in \mathscr{M}(\mathbb{T},\frac{\delta}{2})} \|\theta[\mu, n^{k+1}_-]\|_{L^\infty(\mathbb{T})},
    \end{equation*}
    where $\mathscr{M}(\mathbb{T}, \frac{\delta}{2})$ denotes the convex set of non-negative measures with total mass $\delta/2$. By \Cref{lem: analogous}, finding the optimal update $\mu^*$ that minimizes the global $L^\infty$ norm is equivalent to minimizing the first-order directional derivative. By Danskin's Theorem (\Cref{thm: Danskin}), this directional derivative is determined by the maximal values evaluated on the active set $\mathscr{C} := \argmax \theta[m^{k+1}]$. Let $\dot{\theta}[\mu]$ denote the directional derivative associated with the positive mass addition $\mu$, satisfying the linearised equation
    \begin{equation}
        \label{eq: elliptic_linearised}
        - \Delta \dot\theta + P(x) \dot\theta = -\mu, \quad x \in \mathbb{T}.
    \end{equation}
    The optimality condition therefore requires the optimal update $\mu^*$ to solve the minimax problem
    \begin{equation}
        \label{eq: optimal_minmax}
        \min_{\mu \in \mathscr{M}_+(\mathbb{T}; \delta/2)} \max_{x \in \mathscr{C}} \dot\theta[\mu](x) .
    \end{equation}
    To this end, we introduce an auxiliary $L^p$ regularisation. Let $\rho \in \mathscr{P}(\mathbb{T})$ be a probability measure supported on $\mathscr{C}$. We define the relaxed functional
    \begin{equation}\label{eq: proxi.pb}
        J_p(\mu) := \left(\int |\dot{\theta}[\mu]|^p \rho(dx)\right)^{\frac{1}{p}}.
    \end{equation} 
    We first compute the unconstrained G\^ateaux derivative of $J_p$ at the optimal point $\mu^*$ in an arbitrary direction $\eta \in \mathscr{M}(\mathbb{T})$. Considering the perturbation $\mu^* + t\eta$, the linearity of the state equation yields $\dot{\theta}[\mu^* + t\eta] = \dot{\theta}[\mu^*] + t\Psi[\eta]$, where $\Psi[\eta]$ solves $-\Delta \Psi[\eta] + P(x)\Psi[\eta] = -\eta$. Therefore, we obtain:
    \begin{align}
        D J_p(\mu^*)[\eta] &= \left(\int_{\mathbb{T}} |\dot{\theta}[\mu^*]|^p \rho(dx)\right)^{\frac{1}{p}-1} \int_{\mathbb{T}} |\dot{\theta}[\mu^*]|^{p-2} \dot{\theta}[\mu^*] \Psi[\eta] \rho(dx) \\
        &= \int_{\mathbb{T}} \Psi[\eta](x) \rho_p(dx),
        \label{eq: explicit_derivative}
    \end{align}
    where the normalised density $\rho_p$ depends exclusively on the evaluation point $\mu^*$:
    \begin{equation} \label{eq: normalised_density}
        \rho_p(dx) := \frac{|\dot{\theta}[\mu^*]|^{p-2} \dot{\theta}[\mu^*]}{\|\dot{\theta}[\mu^*]\|_{L^p(\rho)}^{p-1}} \rho(dx).
    \end{equation}
    To avoid the explicit computation of $\Psi[\eta]$, we introduce the adjoint state $\varphi_p$, defined as the solution to $-\Delta\varphi_p+P(x)\varphi_p = \rho_p$. Integrating by parts, we transfer the derivatives:
    \begin{align*}
        D J_p(\mu^*)[\eta] &= \int_{\mathbb{T}} \Psi[\eta] \big( -\Delta\varphi_p + P(x)\varphi_p \big) dx \\
        &= \int_{\mathbb{T}} \varphi_p \big( -\Delta\Psi[\eta] + P(x)\Psi[\eta] \big) dx \\
        &= \int_{\mathbb{T}} \varphi_p (-\eta)(dx) = \int_{\mathbb{T}} (-\varphi_p(x)) \eta(dx).
    \end{align*}
    
    Now we enforce the constraints of our optimization problem. The first-order necessary condition for $\mu^*$ to minimize the convex functional $J_p$ over the convex set $\mathscr{M}(\mathbb{T}, \frac{\delta}{2})$ is that the directional derivative pointing from $\mu^*$ toward any other admissible point $\nu \in \mathscr{M}(\mathbb{T}, \frac{\delta}{2})$ must be non-negative. Testing in the valid constrained direction $\eta = \nu - \mu^*$, we require:
    \begin{equation*}
        D J_p(\mu^*)[\nu - \mu^*] \geq 0 \implies \int_{\mathbb{T}} (-\varphi_p(x)) (\nu - \mu^*)(dx) \geq 0.
    \end{equation*}
    Rearranging this inequality yields:
    \begin{equation*}
        \int_{\mathbb{T}} (-\varphi_p(x)) \mu^*(dx) \leq \int_{\mathbb{T}} (-\varphi_p(x)) \nu(dx) \quad \text{for all } \nu \in \mathscr{M}_+\left(\mathbb{T}, \frac{\delta}{2}\right).
    \end{equation*}
    This demonstrates that $\mu^*$ must be the global minimizer of the linear integral $\int (-\varphi_p) d\nu$ over the constrained set of positive measures with fixed mass $\delta/2$. To minimize this integral, the optimal measure $\mu^*$ must be entirely concentrated on the points where the integrand $(-\varphi_p)$ attains its absolute minimum. Equivalently, $\mu^*$ must be concentrated on the global maximum of $\varphi_p$. Thus, we establish that $\supp(\mu^*) \subset \argmax \varphi_p$.

    Finally, we pass to the limit as $p \to \infty$. By construction, the total variation of $\rho_p$ is bounded by 1. Thus, up to a subsequence, $\rho_p$ converges weakly-$\ast$ in the sense of measures to a limit $\rho_\infty$. For any point $x$ where $|\dot{\theta}[\mu^*](x)| < \|\dot{\theta}[\mu^*]\|_{L^\infty(\rho)}$, the ratio $(|\dot{\theta}[\mu^*]|/\|\dot{\theta}[\mu^*]\|_{L^p(\rho)})^{p-1} \to 0$ as $p \to \infty$. This exponential decay forces the limit measure $\rho_\infty$ to be strictly supported on the maximizers of $\dot{\theta}[\mu^*]$ within the support of the base measure $\rho$. Since $\supp(\rho) = \mathscr{C}$, we have:
    \begin{equation*}
        \supp(\rho_\infty) \subset \mathscr{C} = \argmax_{x \in \mathbb{T}} \theta[m^{k+1}].
    \end{equation*}
    The associated adjoint state converges to $\varphi_\infty$ solving $-\Delta\varphi_\infty+P(x)\varphi_\infty = \rho_\infty$. Because $\varphi_\infty \in H^1(\mathbb{T}) \hookrightarrow C(\mathbb{T})$, it attains a global maximum. Let $U = \mathbb{T} \setminus \supp(\rho_\infty)$ be the open set outside the support of the measure. In $U$, the source term is zero, so the adjoint state weakly solves $-\Delta\varphi_\infty + P(x)\varphi_\infty = 0$. By standard elliptic regularity, $\varphi_\infty$ is sufficiently regular in $U$ to apply the classical strong maximum principle. Because $P(x) \geq 0$ and is not identically zero, $\varphi_\infty$ cannot attain a positive maximum in the interior of $U$. Consequently, the global maximum cannot lie in $U$ and must be attained on its complement:
    \begin{equation}\label{eq: mp.adjoint}
        \argmax_{x\in \mathbb{T}}\varphi_\infty \subset \mathbb{T} \setminus U = \supp(\rho_\infty) \subset \argmax_{x\in \mathbb{T}}\theta[m^{k+1}].
    \end{equation}
    Since the optimal added mass $\mu^*$ must be supported on the maximum of the adjoint state, and $\mu^* = n_+^{k+1}$, we conclude that $\supp(n_+^{k+1}) \subset \argmax \theta[m^{k+1}]$.
 
\end{proof}

\begin{lemma}\label{lem: decrease}
   Assume that, for a certain $m \in \mathscr{P}(\mathbb{T})$,
\begin{equation}\label{eq: non.sup}
    \supp(m)\not\subset\argmax\theta[m].
\end{equation}
Then there exist $\varepsilon>0$ and $n\in \mathscr{P}(\mathbb{T})$ such that
\begin{equation*}
    \|\theta[n]\|_{L^\infty(\mathbb{T})}+\frac{1}{2\varepsilon}|n-m|_{\mathsf{TV}(\mathbb{T})}^2\leq  \|\theta[m]\|_{L^\infty(\mathbb{T})}.
\end{equation*}
\end{lemma}

\begin{proof}[Proof of \Cref{lem: decrease}]
    Let $\Phi(m) = \|\theta[m]\|_{L^\infty(\mathbb{T})}$. Because $\supp(m) \not\subset \argmax \theta[m]$, by \Cref{lem: analogous} and \Cref{lem: argmax}, $m$ is not a stationary point, meaning there exists $n^* \in \mathscr{P}(\mathbb{T})$ that provides a strictly negative descent direction. That is, $D_m\Phi(m)[n^* - m] = -c$ for some $c > 0$.

    Let us define a small perturbation $m^\delta = m + \delta(n^* - m)$ for $\delta \in (0,1)$, so that $|m^\delta - m|_{\mathsf{TV}} = \delta|n^* - m|_{\mathsf{TV}} \leq 2\delta$. 
    Using the first-order Taylor expansion from Danskin's theorem, the objective functional evaluated at $m^\delta$ is
    \begin{equation*}
        \Phi(m^\delta) + \frac{1}{2\varepsilon}|m^\delta - m|_{\mathsf{TV}}^2 \leq \Phi(m) - c\delta + o(\delta) + \frac{4\delta^2}{2\varepsilon}.
    \end{equation*}
    To find an $\varepsilon > 0$ such that this entire expression is strictly less than $\Phi(m)$, we require
    \begin{equation*}
        \frac{2\delta^2}{\varepsilon} < c\delta - o(\delta) \implies \varepsilon > \frac{2\delta}{c - \frac{o(\delta)}{\delta}}.
    \end{equation*}
    For $\delta$ sufficiently small, the term $\frac{o(\delta)}{\delta}$ vanishes. Therefore, by choosing $n = m^\delta$ and setting $\varepsilon(\delta) = \frac{4\delta}{c}$, we guarantee a strict decrease in the functional.
    
    To address the case when $\varepsilon\to 0$ and approaches the continuous flow, we check that when $\delta\to 0$, $\varepsilon(\delta)$ goes to zero as well. We conclude that, for any $\varepsilon>0$ small enough, there exists a perturbation $m^\delta$ that makes the functional decrease, as long as \eqref{eq: non.sup} holds.
\end{proof}

\begin{lemma}\label{lem: ac}
    If $\supp(m)\subset \argmax\theta[m]$,  where $\theta$ is given by \eqref{eq: linear}, then $m\in\mathscr{P}_{ac}(\mathbb{T})$.
\end{lemma}
\begin{proof}[Proof of \Cref{lem: ac}]
    Let us consider the Lebesgue decomposition of the measure $m = m_{\text{ac}} + m_{\text{pp}} + m_{\text{sc}}$ (absolutely continuous, pure point, and singular continuous components).

    Suppose, if possible, that there is a Dirac mass at some point $x_0 \in \argmax\theta[m]$ with weight $\alpha_{\text{pp}} > 0$. Integrating the governing equation $-\Delta\theta + P(x)\theta = f - m$ over $[x_0-\varepsilon, x_0+\varepsilon]$, as $\varepsilon \to 0$ we are left with the distributional derivative jump
    \begin{equation*}
        \lim_{\varepsilon \to 0} \left( -\partial_x\theta(x_0+\varepsilon) + \partial_x\theta(x_0-\varepsilon) \right) = -\alpha_{\text{pp}} \implies \partial_x\theta(x_0^+) - \partial_x\theta(x_0^-) = \alpha_{\text{pp}} > 0.
    \end{equation*} 
    The positive jump means $\theta$ is strictly increasing immediately to the right of $x_0$, contradicting the assumption that $x_0\in \argmax\theta[m]$. Therefore, it must be $m_{\text{pp}} = 0$.
    
    Now, let us assume that $m=m_{\text{ac}}+m_{\text{sc}}$ and set $B=\argmax\theta[m]$. Because $B$ contains the global maxima of $\theta$, the function cannot be strictly convex anywhere on $B$, implying $\Delta \theta[m] \leq 0$ on $B$. Rearranging \eqref{eq: linear} on the set $B$ then yields
    \begin{equation} \label{eq: negative_mass}
        m \leq f(x) - P(x)\theta(x).
    \end{equation}
    By definition, the singular continuous measure $m_{\text{sc}}$ is concentrated on some Borel set $C \subset B$ that has a Lebesgue measure of zero. If we integrate \eqref{eq: negative_mass} over this null set $C$, we obtain
    \begin{equation*}
        m_{\text{sc}}(C) = m(C) \leq \int_C \big( f(x) - P(x)\theta(x) \big)\ dx = 0,
    \end{equation*}
    from where we deduce that $m_{sc}=0$, because $m$ is a probability measure.    
\end{proof}

\begin{remark}
    Note that this proof works  when $\theta$ is given by the non-linear equation \eqref{eq: elliptic_nonlinear}, since it relies exclusively on the sign of the Laplacian at the maximum and the boundedness of the remaining terms.
\end{remark}

With the auxiliary lemmas established, we can now prove the two main theorems.

\begin{proof}[Proof of \Cref{thm: main}]
We proceed in two steps, corresponding to the two statements of the theorem.

\textbf{Step 1: Convergence to a continuous curve.} 
The existence of the limit curve $m \in \mathscr{C}((0,T);\mathscr{P}(\mathbb{T}))$ for the Eikonal scheme \eqref{eq: mm.GF} is a direct consequence of \Cref{thm: convergence.to.flow}. By \Cref{lem: prox.properties}, the discrete scheme satisfies the uniform step-size bound $|m^{k+1}-m^k|_{\mathsf{TV}} \leq C\varepsilon$. This allows us to apply the Ascoli-Arzelà theorem, guaranteeing that the piecewise constant interpolation $m^{\lfloor t/\varepsilon_l \rfloor}$ converges uniformly in the Wasserstein-1 metric to a Lipschitz continuous trajectory $m(t)$ as $\varepsilon_l \downarrow 0$.

\textbf{Step 2: Convergence to the ergodic equilibrium.}
We consider the functional $\Phi(m) = \|\theta[m]\|_{L^\infty(\mathbb{T})}$ as a Lyapunov function for the flow. By \Cref{lem: prox.properties}, the sequence $\Phi(m^k)$ is monotonically decreasing and bounded below by zero, so it must converge to a limit value. By a simple compactness argument, the continuous limit curve $m(t)$ must approach a stationary set as $t \to +\infty$.

Let $m^*$ be a limit point of the trajectory and suppose that it is \emph{not} a mean field Nash equilibrium. By the weak-KAM characterisation \eqref{eq: kam}, this implies
\begin{equation*}
    \supp(m^*) \not\subset \argmax_{x \in \mathbb{T}} \theta[m^*](x).
\end{equation*}
Under this assumption, \Cref{lem: decrease} guarantees the existence of a perturbation that strictly decreases the functional $\Phi(m^*)$. This strict decrease contradicts the fact that $m^*$ is the stationary limit of the minimising movement. Therefore, we must have $\supp(m^*) \subset \argmax \theta[m^*]$. 

Finally, \Cref{lem: ac} ensures that any such measure satisfying this support condition for the elliptic PDE \eqref{eq: linear} cannot contain atoms or singular parts. Thus, $m^*$ satisfies the necessary and sufficient conditions to be a solution of the ergodic mean field game system \eqref{eq: intro.emfg}.
\end{proof}

\begin{proof}[Proof of \Cref{thm: main.2}]
The proof relies on the compactness result established in \Cref{thm: convergence.to.flow}. For the Best Response scheme \eqref{eq: mm.BR}, \Cref{lem: prox.properties} provides the H\"older-type estimate $|m^{k+1}-m^k|_{\mathsf{TV}} \leq C\varepsilon^{1/2}$. While this bound is less regular than the Lipschitz bound of the Eikonal scheme, it is sufficient to establish uniform equicontinuity of the family of approximate solutions. By applying the Ascoli-Arzelà theorem, we extract a subsequence $\varepsilon_l \to 0$ such that the discrete flow $m^{\lfloor t/\varepsilon_l \rfloor}$ converges to a continuous curve $m \in \mathscr{C}((0,T);\mathscr{P}(\mathbb{T}))$, which concludes the proof.
\end{proof}

\subsection{Remarks and extensions}\label{section: remarks}

We outline how the theoretical framework developed above behaves under different modelling assumptions, boundary conditions, and spatial dimensions.

\begin{enumerate}
    \item \textbf{General domains and boundary conditions.}
    While the theory is presented on the flat torus $\mathbb{T}$ for simplicity of notation (as it naturally eliminates boundary terms during integration by parts), all the results extend directly to bounded unidimensional domains endowed with homogeneous Neumann boundary conditions. The Neumann conditions perfectly cancel the boundary terms in Green's Second Identity during the adjoint state method transfer. Furthermore, the Hopf Lemma ensures that the maximum principle arguments hold even if the payoff is maximised exactly at the boundary of the domain.

    \item \textbf{Bilinear and non-linear interactions.} 
    The proof of \Cref{thm: convergence.to.flow} (convergence of the discrete steps to a continuous flow) remains unaltered for a wider class of models, such as those with bilinear interactions
    \begin{equation}\label{eq: bilinear}
        -\Delta\theta+(P(x)+m(x))\theta=f(x),\quad x\in \mathbb{T},
    \end{equation}
    or general non-linear models
    \begin{equation}\label{eq: nonlinear}
        -\Delta\theta+g(x,\theta)+m(x)\theta=f(x),\quad x\in \mathbb{T},
    \end{equation}
    provided that the existence and uniqueness of the solutions hold, and that the map $\mathscr{P}(\mathbb{T})\ni m\mapsto \theta[m]\in \mathscr{C}(\mathbb{T})$ remains continuous with respect to the weak topology. 
    
    Furthermore, the variational characterisation in \Cref{lem: analogous} holds for these non-linear models. However, a delicate issue arises in \Cref{lem: argmax}. The adjoint state method relies heavily on the strong maximum principle to force the mass update into $\argmax \theta$. This requires that the potential resulting from the linearisation of $g$ possess a constant sign. This hypothesis is unfortunately not fulfilled by the standard monostable nonlinearity $g(x,\theta)=-\theta(K(x)-\theta)$ utilised in our numerical tests. While the numerical algorithms clearly succeed in this regime, extending the rigorous proofs to non-monotone nonlinearities requires a more delicate analysis.

    \item \textbf{Extension to higher dimensions ($d \ge 2$).}
    The analytical framework was explicitly restricted to $d=1$ due to the specific topological properties of the state equation \eqref{eq: linear}. In dimensions $d \ge 2$, the Sobolev embedding $H^1 \hookrightarrow C^0$ (and consequently $H^1 \hookrightarrow L^\infty$) fails, and the fundamental solution of the Laplacian exhibits a singularity at the origin $G(x,x)=-\infty$, therefore, since $G$ cannot be Lipschitz, it  breaks the uniform boundedness and the Wasserstein-1 Lipschitz continuity established in \Cref{lem: basic.properties}.
    Without these properties, the uniform finite-difference bounds required for the Ascoli-Arzelà compactness argument in \Cref{thm: convergence.to.flow} cannot be guaranteed. Fully extending the dynamic convergence of the $\mathsf{TV}$-flow to $d \ge 2$ remains an open problem.

    However, the underlying optimisation geometry works in any dimension. The derivation of the optimal placement via the adjoint state method (\Cref{lem: argmax}) and the strict decrease via Taylor expansion (\Cref{lem: decrease}) are dimension-independent. 
    
    Most notably, the regularity result of \Cref{lem: ac} (absence of atoms in the optimal measure) becomes simpler and stronger in $d \ge 2$. Suppose $m$ contains a positive Dirac mass $m_{\text{pp}}\delta_{x_0}$. In $d \ge 2$, this point mass dominates the bounded terms of the PDE, forcing the Laplacian to behave as $\Delta \theta \approx m_{\text{pp}} \delta_{x_0}$. This induces a fundamental singularity where $\theta(x) \to -\infty$ as $x \to x_0$ (e.g., logarithmic in 2D, or $|x|^{2-d}$ in 3D). A point where the payoff collapses to negative infinity trivially contradicts the assumption that $x_0$ belongs to $\argmax \theta$. Consequently, we expect the stationary optimal measures of this class of harvesting mean field games to be absolutely continuous in any spatial dimension.
\end{enumerate}

\section{Numerical methods and tests} \label{sec: numerics}

We describe here the practical implementations of \cref{alg:best_response} and \cref{alg:furthest_mass}, reported, respectively, in \cref{alg: main} and \Cref{alg: eikonal}. Then, we conduct some numerical convergence tests and simulations both in the case where $\theta$ satisfies a linear equation and in the case where it satisfies a non linear one.

\subsection{A refined version of \Cref{alg:best_response} and \Cref{alg:furthest_mass}}
We report the detailed implementations of the Best Response and Eikonal-based algorithms with a dynamic time step size $\varepsilon$ on $\mathbb{T}$. Extensions to higher dimensions and different boundary conditions are straightforward.

\begin{algorithm}
    \caption{Best response algorithm in $\mathbb{T}$}\label{alg: main}
    \begin{algorithmic}[1]
        \Require $m_0\in \mathscr{P}(\mathbb{T}),\ 0 < \underline{\varepsilon} \leq \varepsilon_0,\ M \in \mathbb{N},\ \tau >0,\ \mu > 0,\ f: \R \times \mathbb{T} \to \R$
    
        \State $j \gets 0$
        \State $\theta_0 \gets \text{solution of} \ -\mu\Delta\theta=f(\theta,x)-m_0\theta,\quad x\in \mathbb{T} $ \Comment{\textcolor{blue}{\textit{Compute initial payoff}}}
        \State $R_0 \gets \max_{\mathbb{T}} \theta_0 - \min_{\supp\, m_0} \theta_0$ \Comment{\textcolor{blue}{\textit{Calculate maximum income gap}}}
        
        \While{ $R_j > \tau \And j<M$}
            \State $\ k \gets 0$
            \While{ $\varepsilon_k > \underline{\varepsilon}$} \Comment{\textcolor{blue}{\textit{Adaptive step-size (backtracking line search)}}}
                \State \Comment{\textcolor{blue}{\textit{Find threshold to identify the $\varepsilon_k$ worst-earning players}}}
                \State Find $\ \eta_k \in \mathbb{R}\ $ s.t.
                $$
                     m_k^- = m_{k}\ \chi_{\{\theta_k \leq \eta_k\}}, \quad \int_{\mathbb{T}} m_{k}^- \, \ d x = \varepsilon_k, \quad m_k^+ = m_k - m_k^- 
                $$
                
                \State \Comment{\textcolor{blue}{\textit{Shape new mass $\nu_k$ to match optimal profile on peaks}}}
                \State Find $\ C_k \in \mathbb{R}\ $ s.t.
                $$
                    \nu_k (x) = \left[f(x) - P(x)\theta_k - m^+(x)\right] \cdot \chi_{\{\theta_k \geq C_k\} }(x),
                    \quad
                    \int_{\mathbb{T}} \nu_k \ d x = \varepsilon_k,
                $$
                \State $m_{k+1} \gets m^+_k + \nu_k $ \Comment{\textcolor{blue}{\textit{Move players to optimal configuration}}}
                \State $\theta_{k+1} \gets \text{solution of} \ -\mu\Delta\theta=f(\theta,x)-m_{k+1}\theta,\quad x\in \mathbb{T} $ \Comment{\textcolor{blue}{\textit{Update payoff}}}
                \State $R_{k+1} \gets \max_{\mathbb{T}} \theta_{k+1} - \min_{\supp\, m_{k+1}} \theta_{k+1}$ \Comment{\textcolor{blue}{\textit{Evaluate new gap}}}
            
                \If{$ R_{k+1} \geq R_j$} \Comment{\textcolor{blue}{\textit{If gap increased, reject step and halve step-size}}}
                    \State $\varepsilon_{k+1} \gets \varepsilon_k /2$
                    \State $k \gets k+1$
                \Else \Comment{\textcolor{blue}{\textit{If gap decreased, accept step and proceed}}}
                    \State $m_{j+1} \gets m_{k+1}$
                    \State $R_{j+1} \gets R_{k+1}$
                    \State break
                \EndIf
            \EndWhile
            \State $j \gets j+1$
        \EndWhile
    \end{algorithmic}
\end{algorithm}

\begin{algorithm}
    \caption{Eikonal-based algorithm in $\mathbb{T}$ (highest-income minimisation)} \label{alg: eikonal}
    \begin{algorithmic}[1]
        \Require $m_0\in \mathscr{P}(\mathbb{T}),\ 0 < \underline{\varepsilon} \leq \varepsilon_0, \ M \in \mathbb{N},\ \tau >0,\ \mu > 0,\ f: \R \times \mathbb{T} \to \R$
        \State $j \gets 0$
        \State $\theta_0 \gets \text{solution of} \ -\mu\Delta\theta=f(\theta,x)-m_0\theta,\quad x\in \mathbb{T} $ \Comment{\textcolor{blue}{\textit{Compute initial payoff}}}
        \State $R_0 \gets \max_{\mathbb{T}} \theta_0 - \min_{\supp\, m_0} \theta_0$
        
        \While{ $R_j > \tau \And j<M$}
            \State $\ k \gets 0$
            \While{ $\varepsilon_k > \underline{\varepsilon}$} \Comment{\textcolor{blue}{\textit{Adaptive step-size search}}}
                \State Solve eikonal equation \eqref{eqn:eikonal} to get distance function $v_k$ \Comment{\textcolor{blue}{\textit{Find distance to peaks}}}
                
                \State \Comment{\textcolor{blue}{\textit{Identify the $\varepsilon_k$ players geographically furthest from peaks}}}
                \State Find $\ \eta_k \in \mathbb{R}\ $ s.t.
                $$
                     m_k^- = m_{k}\ \chi_{\{v_k \geq \eta_k\}}, \quad \int_{\mathbb{T}} m_{k}^- \, \ d x = \varepsilon_k, \quad m_k^+ = m_k - m_k^- 
                $$
                
                \State \Comment{\textcolor{blue}{\textit{Shape relocated mass $\nu_k$ to optimal stationary conditions}}}
                \State Find $\ C_k \in \mathbb{R}\ $ s.t.
                $$
                    \nu_k (x) = \left[f(x) - P(x)\theta_k - m^+(x)\right] \cdot \chi_{\{\theta_k \geq C_k\} }(x),
                    \quad
                    \int_{\mathbb{T}} \nu_k \ d x = \varepsilon_k,
                $$
                \State $m_{k+1} \gets m^+_k + \nu_k $ \Comment{\textcolor{blue}{\textit{Relocate the players}}}
                \State $\theta_{k+1} \gets \text{solution of} \ -\mu\Delta\theta=f(\theta,x)-m_{k+1}\theta,\quad x\in \mathbb{T} $
                \State $R_{k+1} \gets \max_{\mathbb{T}} \theta_{k+1} - \min_{\supp\, m_{k+1}} \theta_{k+1}$
            
                \If{$ R_{k+1} \geq R_j$} \Comment{\textcolor{blue}{\textit{Reject step if functional fails to decrease}}}
                    \State $\varepsilon_{k+1} \gets \varepsilon_k /2$
                    \State $k \gets k+1$
                \Else \Comment{\textcolor{blue}{\textit{Accept step otherwise}}}
                    \State $m_{j+1} \gets m_{k+1}$
                    \State $R_{j+1} \gets R_{k+1}$
                    \State break
                \EndIf
            \EndWhile
            \State $j \gets j+1$
        \EndWhile  
    \end{algorithmic}
\end{algorithm}

Observe that the minimising movement induced by \eqref{eq: mm.BR} or \eqref{eq: mm.GF} can be seen as an implicit Euler discretisation of the curves $m\in \mathscr{C}((0,T),\mathscr{P}(\mathbb{T}))$ found in \Cref{thm: main} and \Cref{thm: main.2}, whereas the algorithms have mixed elements of an explicit and implicit Euler discretisations. As a matter of fact, as proved in \cref{lem: argmax}, for a sequence $\{m^k\}_{k\ge 0}$ satisfying \eqref{eq: mm.BR} or \eqref{eq: mm.GF}, the positive part of the difference $m^{k+1}-m^k$ must be supported in $\argmax\theta[m^{k+1}](x)$, for all $k \ge 0$. We exploit this property in \cref{alg: main} and \cref{alg: eikonal} in order to avoid the computation of the implicit argument minimum on the right hand side of \eqref{eq: mm.BR} and \eqref{eq: mm.GF}. In fact, we impose explicitly that the mass update be of size $\varepsilon$ and that its positive part be supported in $\argmax \theta[m^{k+1}](x)$. The implicit element is on determining the shape that such mass will have, that is done thinking on \eqref{eq: implications.NE}.  Furthermore, the two algorithms have an adaptive temporal step size. The requirement of having a small enough $\varepsilon$ is necessary to avoid oscillations on the solution, as shown numerically in \Cref{fig: non.convergence}. Additionally, although a dynamic step size may result in a non-monotonic sequence $\{\varepsilon_j\}_{j\geq0}$, it produces a monotonic decrease in the functional to be minimised (see \Cref{fig:adaptive_epsilon} for a numerical example).
\begin{figure}[h]
    \centering
    \includegraphics[width=0.7\linewidth]{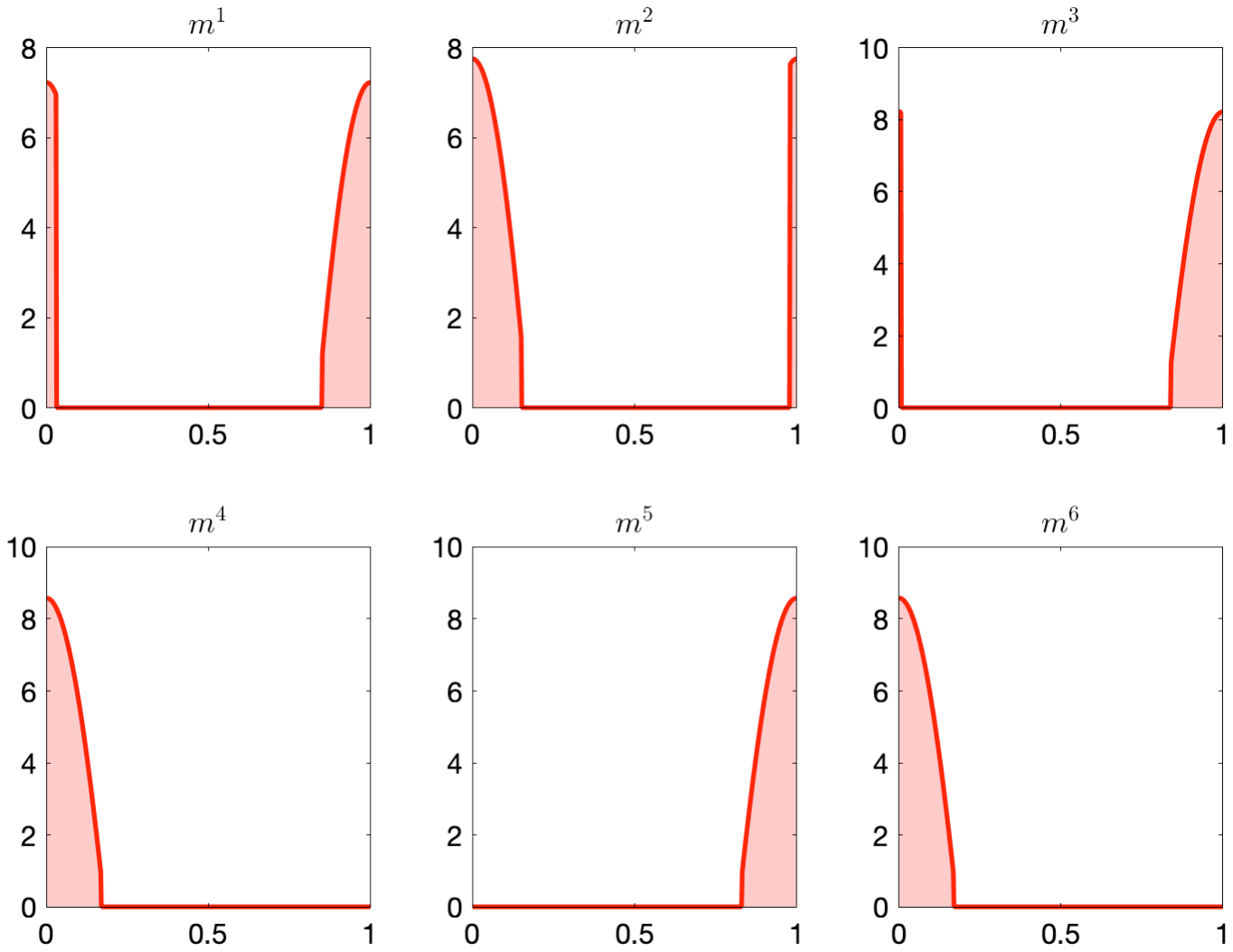}
    \caption{Example of non-convergence of \Cref{alg: main} with fixed step-size $\varepsilon=1$.}
    \label{fig: non.convergence}
\end{figure}

\begin{figure}[h]
    \centering
    \includegraphics[width=0.5\linewidth]{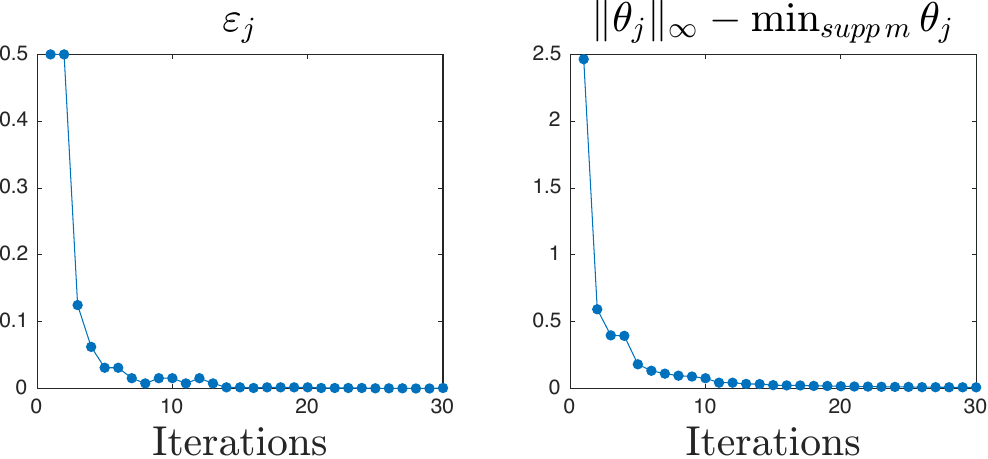}
    \caption{Example of non-monotonic sequence $\{\varepsilon_j\}_{j\geq0}$ producing a strictly decreasing sequence $\{R_j\}_{j\geq0}$ for \Cref{alg: main}.}
    \label{fig:adaptive_epsilon}
\end{figure}

\subsection{Numerical tests}

Before presenting the results, we briefly clarify the scope of the simulations with respect to the theoretical framework developed in \Cref{sec:main_results}. Specifically, we distinguish between the settings that are rigorously supported by our proofs and those that serve as broader empirical explorations.

While the theoretical results are obtained on the torus $\mathbb{T}$ for notational simplicity (eliminating boundary terms during integration by parts), all numerical simulations are performed on bounded domains with homogeneous Neumann boundary conditions. From an ecological modelling perspective, zero-flux Neumann boundaries are the natural choice, representing a closed environment where agents cannot enter or exit the habitat. As detailed in \Cref{section: remarks}, the theoretical framework extends perfectly to this Neumann setting.

The one-dimensional (1D) simulations for the linear case fall entirely within the theoretical framework. In 1D, the properties of \eqref{eq: linear} guarantee the uniform boundedness and convergence of both generalised minimising movement schemes. Conversely, the two-dimensional (2D) simulations presented in this section push the algorithms beyond their proven theoretical limits. The non-linear harvesting model~\eqref{eq: elliptic_nonlinear} utilises a monostable nonlinearity that does not satisfy the strict monotonicity condition required to apply the strong maximum principle in the adjoint state method (\Cref{lem: argmax}). Consequently, the non-linear simulations, even in 1D, fall outside the strict analytical guarantees. Nevertheless, we include these 2D and non-linear tests to demonstrate the practical robustness and broad algorithmic applicability of both \Cref{alg: main} and \Cref{alg: eikonal} in more complex, realistic environments where analytical proofs are currently out of reach.

We now conduct numerical convergence tests and simulations, addressing first the linear case, followed by the non-linear one. In particular, we address both the decrease of the functional to be minimised in \Cref{alg: main} and \Cref{alg: eikonal}, and the convergence of the discrete flow to a gradient flow as $\varepsilon \downarrow 0$. Finally, we show the results obtained with both algorithms in various simulations, both on one- and two-dimensional domains.

For the 1D tests, the computational domain was normalised to $\Omega=[0,1]$, over which we set a uniform grid $x_i,\ i=0,\ldots, 1000$. For the two-dimensional ones, we set $\Omega=[0,1]\times[0,1]$ with the same boundary conditions and a uniform grid $(x_i,y_j)$, for $i,j=0,\ldots, 100$. Unless differently specified, the initial parameters are $\varepsilon_0 = 0.1,\ \underline{\varepsilon} = 10^{-15},\ M=100,\ \tau=\Delta x$. For both algorithms, all the integrals are computed with a trapezoidal quadrature rule for the one-dimensional simulations and with a rectangular quadrature rule for the two-dimensional ones. The levels $\eta_k$ and $C_k$ in Steps 7--8 of \Cref{alg: main} and 8--9 of \Cref{alg: eikonal} are found via a bisection procedure. Finally, we remark that the numerical solution of the eikonal equation in \Cref{alg: eikonal} is rather delicate, as in case of a disconnected target set (i.e. $\argmax \theta$), the solution is non-differentiable and the equation must be regarded only in the viscosity sense. Consequently, numerical schemes that assume regularity may not be suitable. We implemented, both in 1D and 2D, a Fast Marching semi-Lagrangian scheme \cite{sethian1999fast, falcone2013sl}, which has the property to converge to the viscosity solution of \eqref{eqn:eikonal} in case of non-regularity.

\subsubsection{The linear case}\label{sec:numerics_linear}

We begin with the results of various simulations obtained with \Cref{alg: main} and \Cref{alg: eikonal} in the linear case, that is when $\theta[m]$ is given by
\begin{equation*} \label{eqn:linear_explicit}
    -\mu \Delta\theta+P(x)\theta=f(x)-m(x), \quad \text{with } \mu=0.1.
\end{equation*}
This PDE -- which is just \eqref{eq: linear} with an explicit viscosity coefficient $\mu >0$ -- is approximated with a central finite-difference scheme in both one and two dimensions.

In \Cref{fig:funct_conv_linear} we can see how the functional to be minimised by each algorithm decreases in time, first in a simpler case, with $f(x)=4x+1$, then with a more complex $f(x)=\max\{9x \sin(5\pi x), 0\}$. In both cases, the curve we obtain is compared with a linear interpolant of the first and last points and also one that is proportional to the square root of the transported mass.
\begin{figure}[h!]
    \centering
    \includegraphics[width=0.49\linewidth]{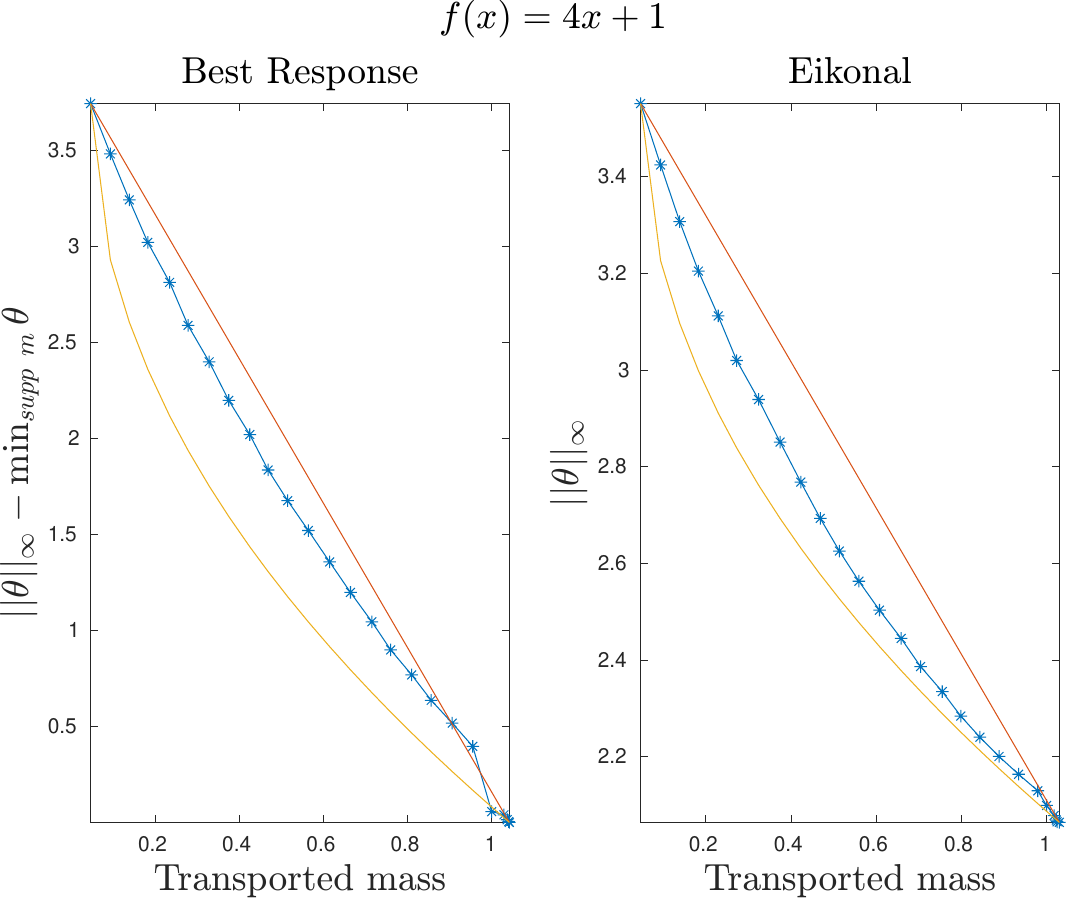} \hfill
    \includegraphics[width=0.49\linewidth]{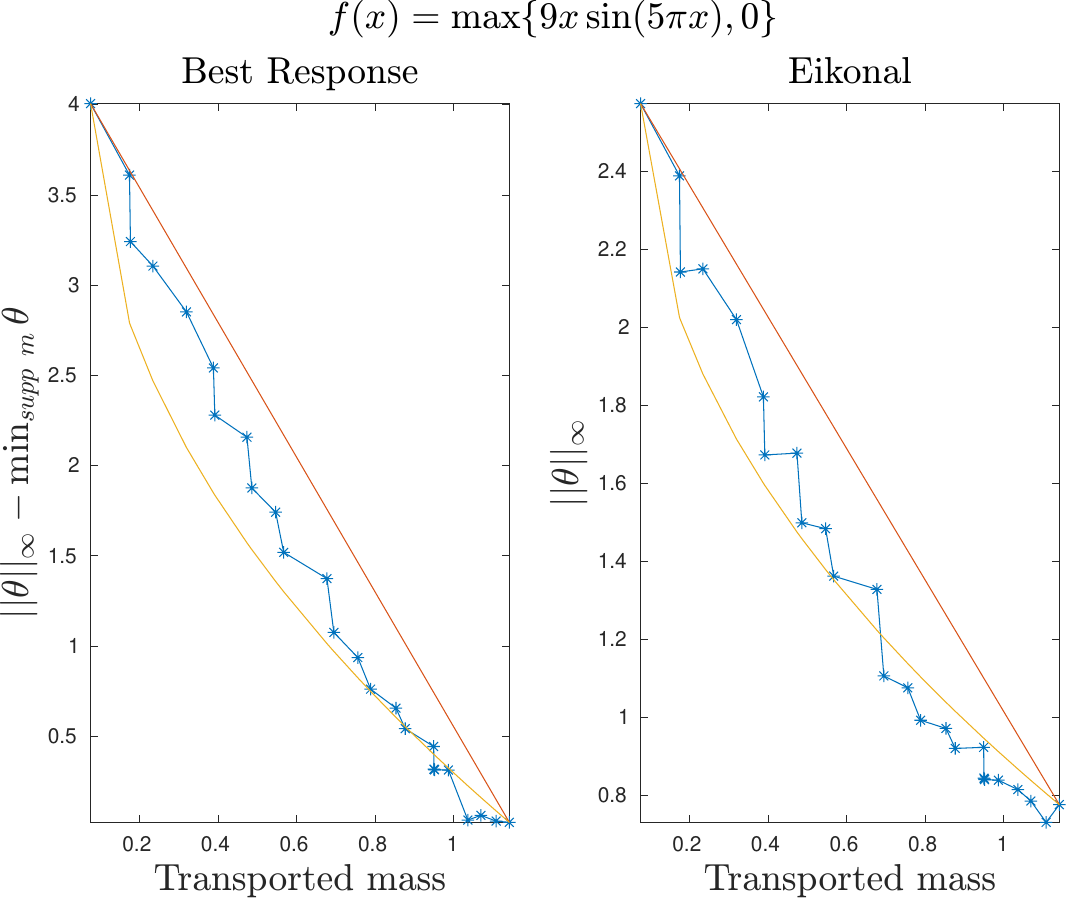}
    \caption{Convergence of the functional to be minimised by \Cref{alg: main} (left) and \Cref{alg: eikonal} (right) in the linear case. The dotted blue curve is the one obtained numerically, the red one is a linear interpolant of the two extrema and the yellow one is an interpolant proportional to the square root of the transported mass.}
    \label{fig:funct_conv_linear}
\end{figure}

Secondly, we test numerically, in \Cref{fig:flow_conv_linear}, the convergence of the sequences generated by the two algorithms to a gradient flow as the step-size $\varepsilon$ decreases. Specifically, starting from $\varepsilon_0 = 0.1$, we compute the $\mathsf{TV}$ distance between the discrete sequences $m^{\big\lfloor\frac{t}{\varepsilon_0/2^k}\big\rfloor}$ and $m^{\big\lfloor\frac{t}{\varepsilon_0/2^{k+1}}\big\rfloor}$, $k=0,\ldots,6$, where time is represented by the cumulative mass that is transported throughout the iterations. Additionally, we plot the maximum of this distance against the step-size epsilon.
\begin{figure}[h!]
    \centering
   \begin{tabular}{c | c}
    \Cref{alg: main} & \Cref{alg: eikonal} \\
    \includegraphics[width=0.48\linewidth, height=6.5cm]{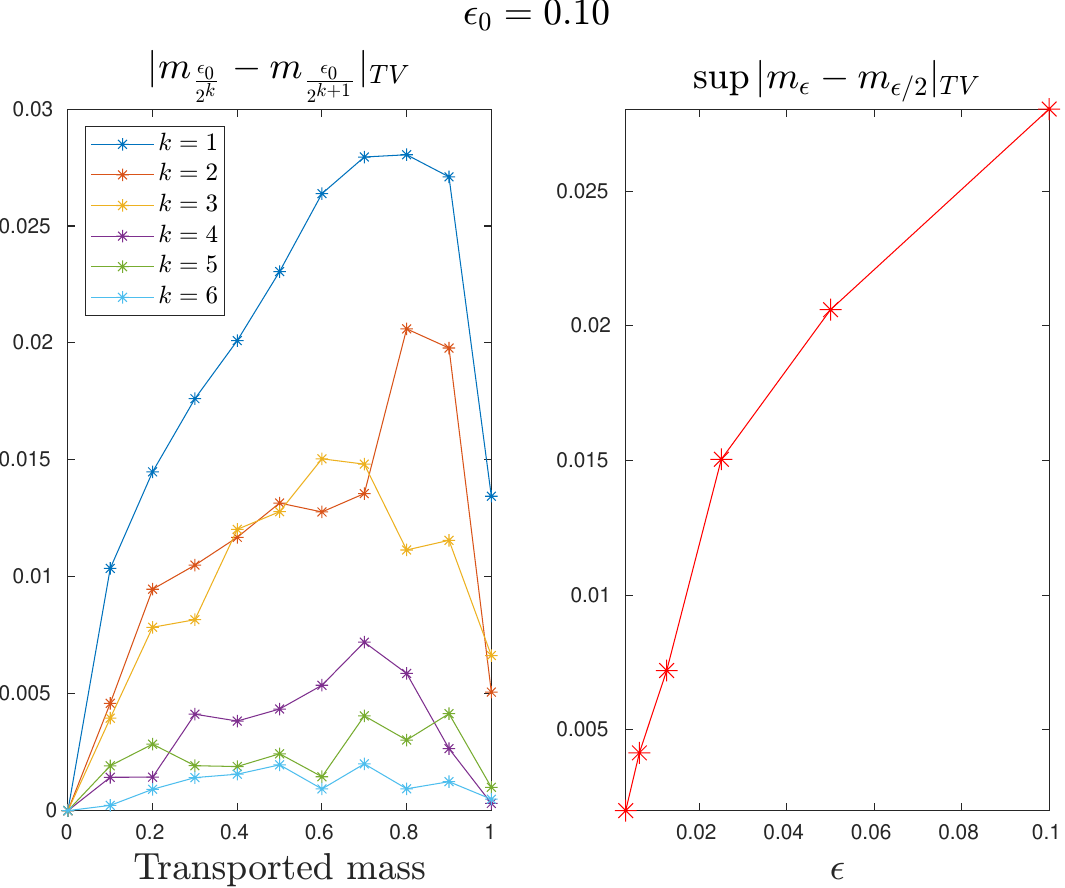} &
    \includegraphics[width=0.48\linewidth, height=6.5cm]{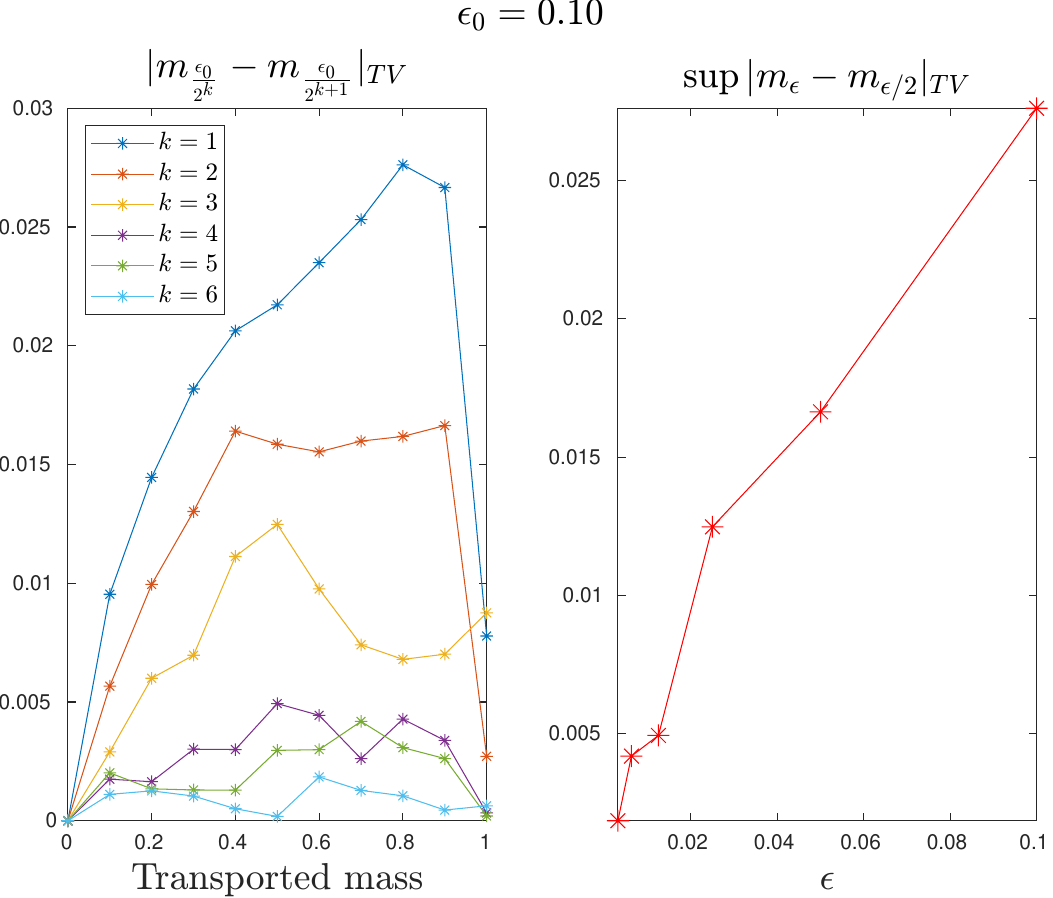} \\
    \end{tabular}
    \caption{Numerical convergence test of the sequence $m^j$ generated by \Cref{alg: main} and \Cref{alg: eikonal} to a gradient flow as the step-size $\varepsilon$ decreases, in the linear case.}
    \label{fig:flow_conv_linear}
\end{figure}

In what follows, we show some numerical simulations. In \Cref{tab:1d_linear_sol} we can see that both algorithms succeed in finding the same configuration -- which, as already mentioned, must be intended as a $\tau$-Nash equilibrium due to the numerical approximation error -- for various choices of the source term $f(x)$ and $P(x)\equiv0.5$. Additionally, \Cref{tab:iterations_1d_linear} shows that the number of iterations needed for convergence is larger for \Cref{alg: eikonal}, although of the same order of magnitude. We point out that the third test, with $f(x)=15(\cos(2\pi x)+1)$, is the same setting that, for fixed $\varepsilon=1$, does not achieve convergence (example in \Cref{fig: non.convergence}).
\begin{figure}[h!]
    \centering
    \begin{tabular}{c c}
    \Cref{alg: main} & \Cref{alg: eikonal} \\
    \toprule
    \multicolumn{2}{c}{\small $f(x)=4x+1$} \\
    \includegraphics[width=0.25\linewidth]{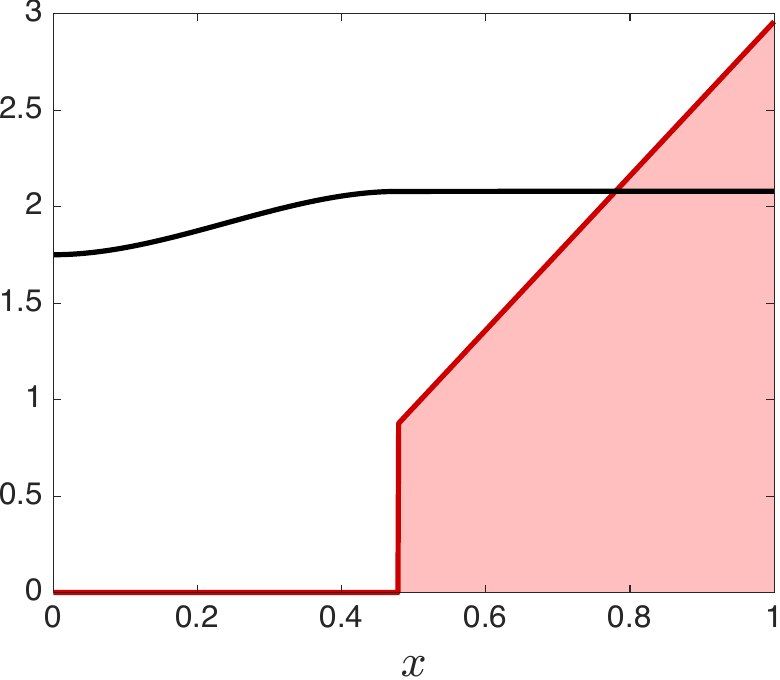} & \includegraphics[width=0.25\linewidth]{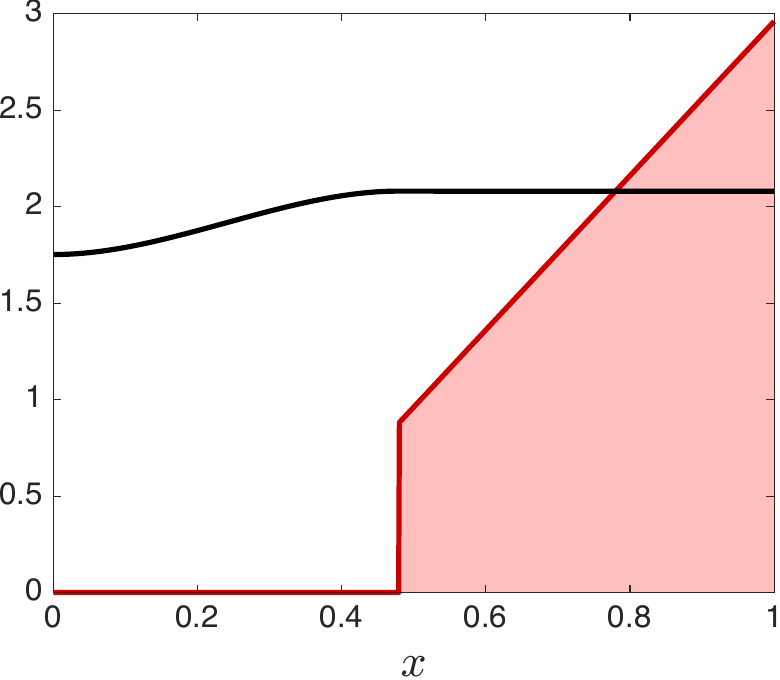} \\
    \multicolumn{2}{c}{\ } \\
    \multicolumn{2}{c}{\small $f(x)=\max(0, 9x\sin(5\pi x))$} \\
    \includegraphics[width=0.25\linewidth]{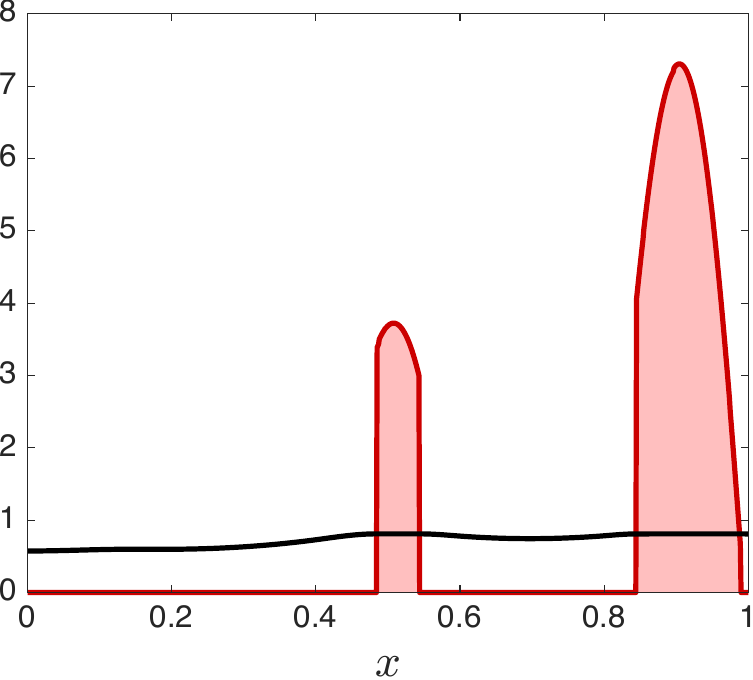} & \includegraphics[width=0.25\linewidth]{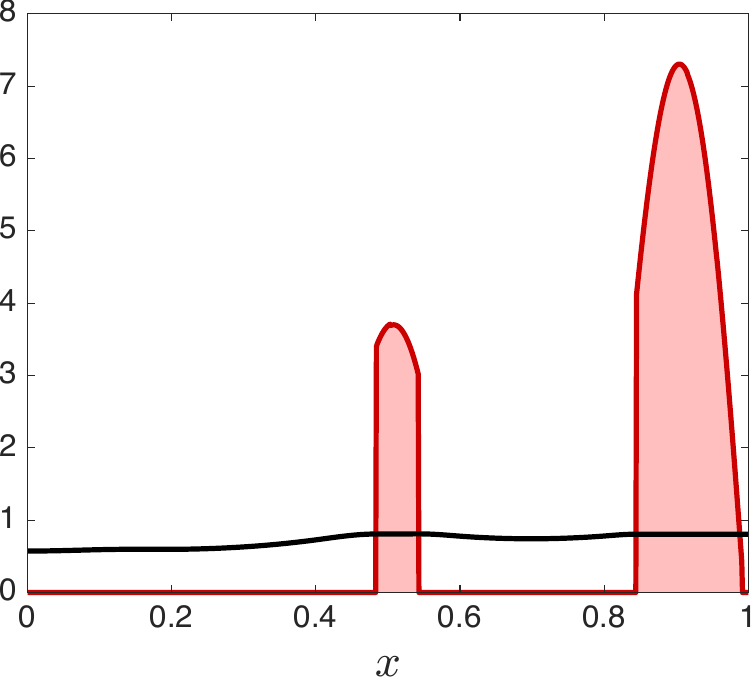} \\
    \multicolumn{2}{c}{\ } \\
    \multicolumn{2}{c}{\small $f(x)=15(\cos(2\pi x)+1)$} \\
    \includegraphics[width=0.25\linewidth]{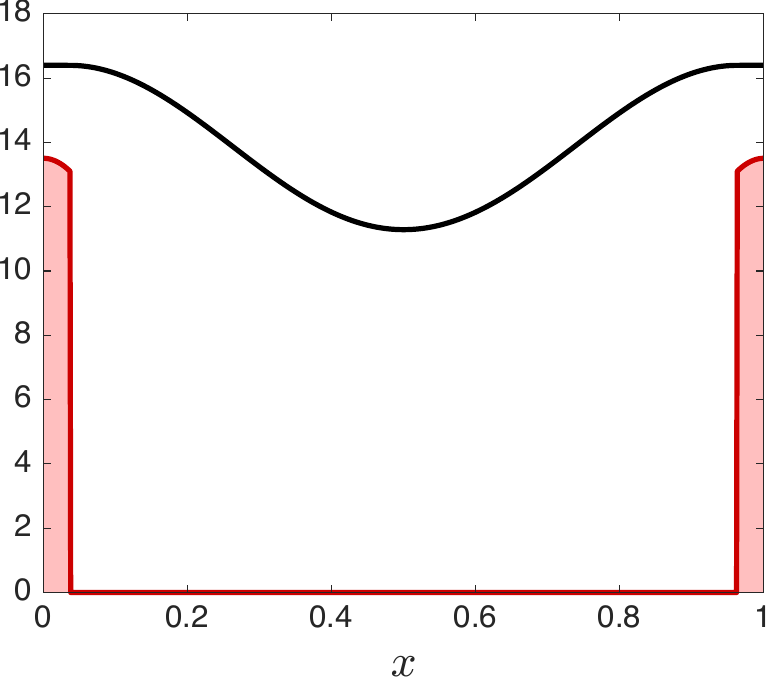} & \includegraphics[width=0.25\linewidth]{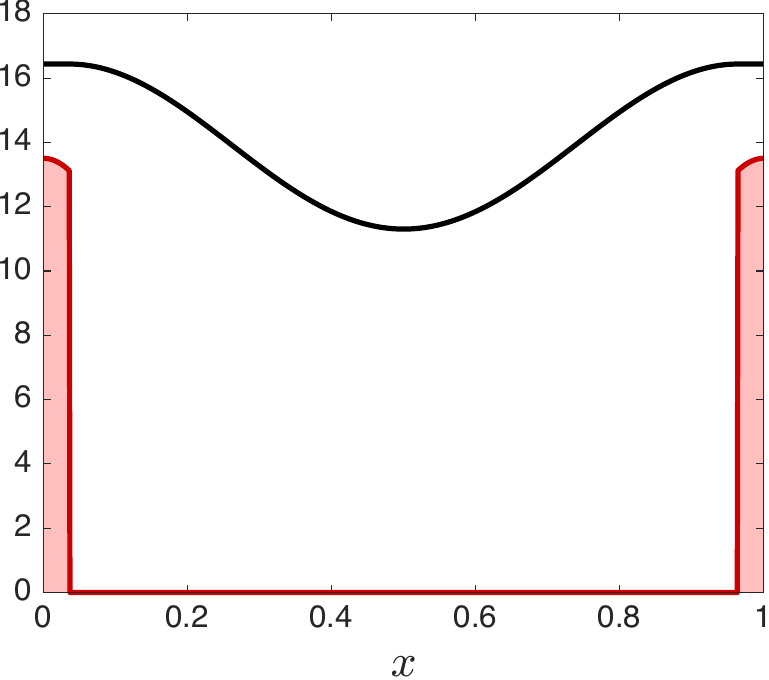} \\
    \end{tabular}
    \caption{Solutions given by \Cref{alg: main} (left column) and \Cref{alg: eikonal} (right column) for various choices of $f$ in one dimension. The red curve represents $m(x)$, the black one represents $\theta[m](x)$.}
    \label{tab:1d_linear_sol}
\end{figure}

\begin{table}[h!]
    \centering
    \begin{tabular}{r c c}
    \toprule
    \ & \Cref{alg: main} & \Cref{alg: eikonal} \\
    \midrule
    $f(x)=4x+1$ & 14 iter. & 15 iter. \\ 
    $f(x)=\max(0, 9x\sin(5\pi x))$ & 30 iter. & 43 iter. \\
    $f(x)=15(\cos(2\pi x)+1)$ & 16 iter. & 20 iter. \\
    \bottomrule
    \end{tabular}
    \caption{Number of iterations until convergence for \Cref{alg: main} and \Cref{alg: eikonal} for the simulations in \Cref{tab:1d_linear_sol}.}
    \label{tab:iterations_1d_linear}
\end{table}

\Cref{tab:2d_linear_sol} reports analogous simulations on a two-dimensional domain. Here, $P(x)\equiv 1$ and the different choices of $f(x)$ are reported on the figures. In particular, in order to test the robustness of both algorithms, $f$ is a sum of cosine functions with random amplitudes and frequencies in the second test. In this case, we observe slight differences in the two solutions, due to numerical accuracy, and \Cref{tab:iterations_2d_linear} shows that in 2D the number of iterations needed by \Cref{alg: eikonal} is about twice as much as those needed by \Cref{alg: main}.
 \clearpage 
\begin{figure}[H]
    \centering
    \begin{tabular}{c c}
    \Cref{alg: main} & \Cref{alg: eikonal} \\
    \toprule
    \multicolumn{2}{c}{\small $f(x,y)=5\exp\left(-\frac{(x-1)^2+(y-1)^2}{0.5}\right)$} \\
    \includegraphics[width=0.3\linewidth]{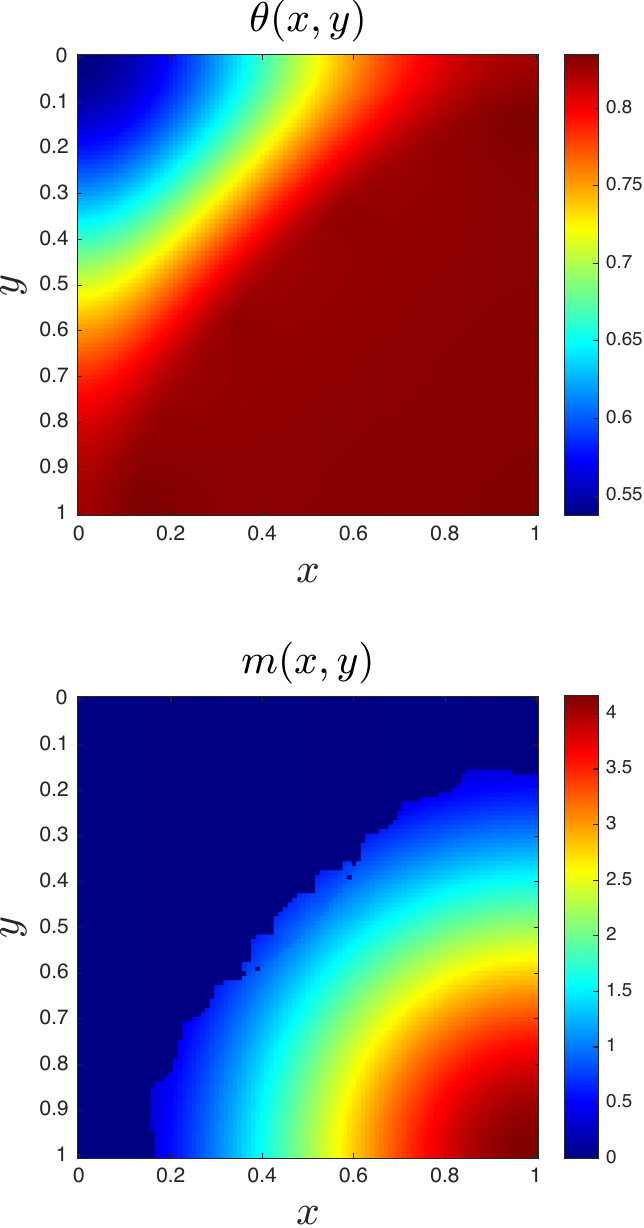} & \includegraphics[width=0.3\linewidth]{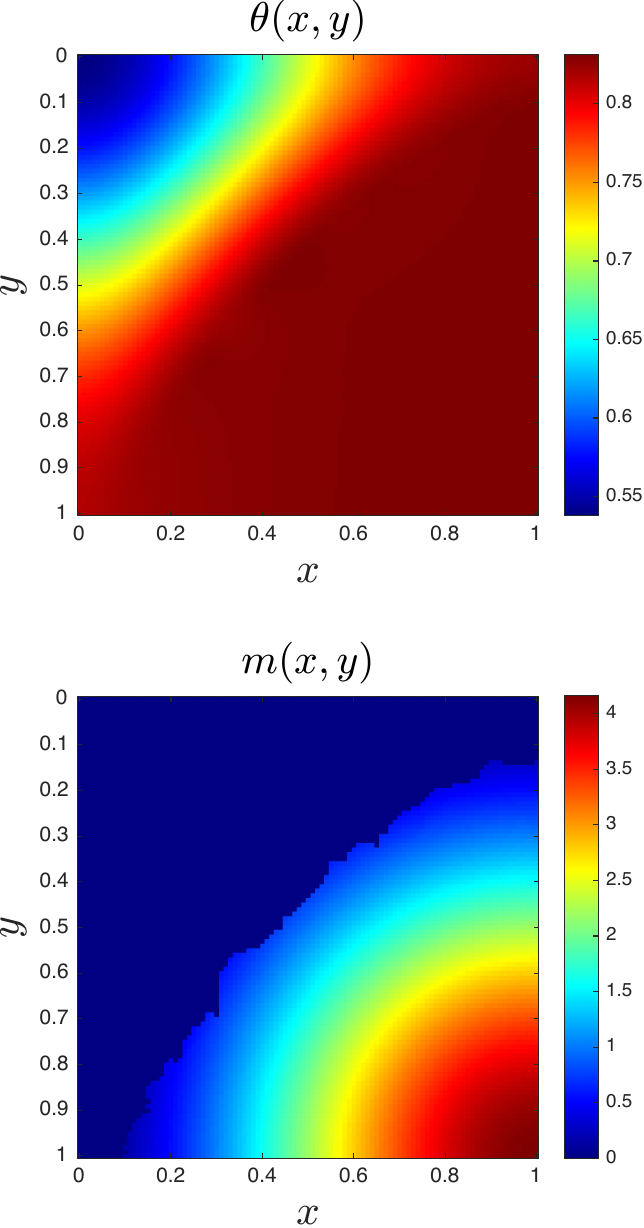} \\
    \multicolumn{2}{c}{\ } \\
    \multicolumn{2}{c}{\small $f(x,y)= \max\left( 0, 4\sum_{i=1}^4 \cos(a_i \pi x)\cos(b_i \pi y) \right), \quad a_i, b_i \sim \mathcal{U}[0,10] $} \\
    \includegraphics[width=0.3\linewidth]{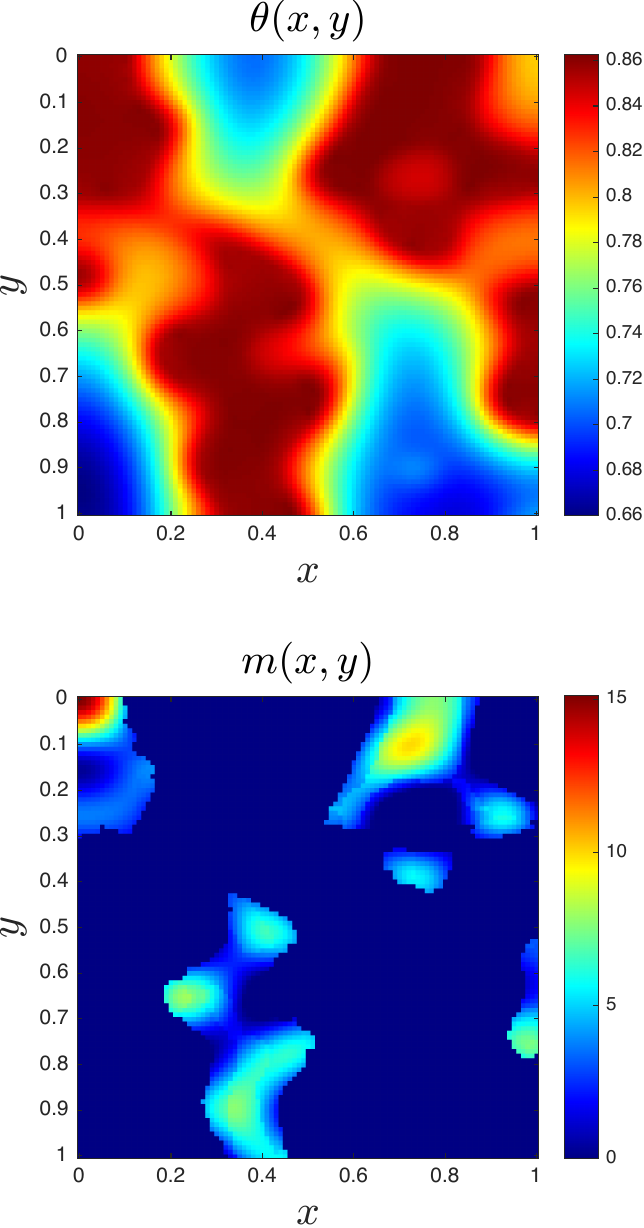} & \includegraphics[width=0.3\linewidth]{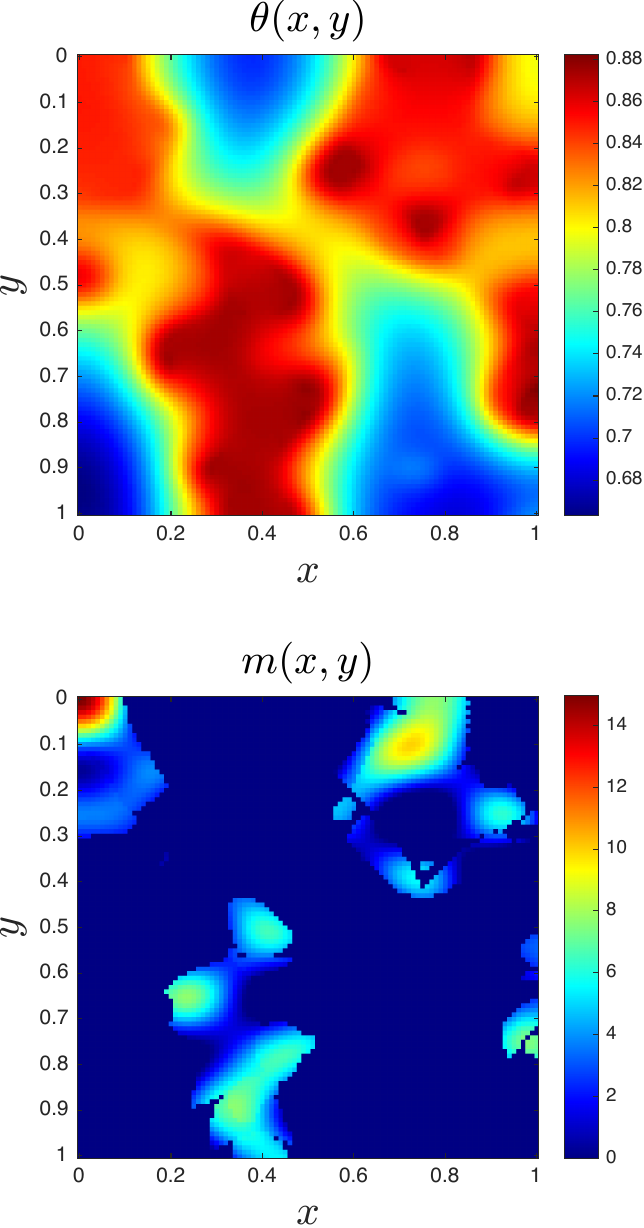} \\
    \bottomrule
    \end{tabular}
    \caption{Solutions given by \Cref{alg: main} (left column) and \Cref{alg: eikonal} (right column) for different choices of $f$ in two dimensions.}
    \label{tab:2d_linear_sol}
\end{figure}

\begin{table}[h!]
    \centering
    \begin{tabular}{r c c}
    \toprule
    \ & \Cref{alg: main} & \Cref{alg: eikonal} \\
    \midrule
    {\small $f(x,y)=5\exp\left(-\frac{(x-1)^2+(y-1)^2}{0.5}\right)$} & 8 iter. & 16 iter. \\ 
    {\small \makecell[r]{ $f(x,y)= \max\left( 0, 4\sum_{i=1}^4 \cos(a_i \pi x) \cos(b_i \pi y) \right),$ \\ $a_i, b_i \sim \mathcal{U}[0,10] $}} & 10 iter. & 24 iter. \\
    \bottomrule
    \end{tabular}
    \caption{Number of iterations until convergence for \Cref{alg: main} and \Cref{alg: eikonal} for the simulations in \Cref{tab:2d_linear_sol}. For these numerical tests, the maximum $\varepsilon$ allowed was $0.5$.}
    \label{tab:iterations_2d_linear}
\end{table}

In conclusion, to further test the robustness of both algorithms and to check whether the number of iterations required for convergence varies depending on the initial mass distribution, we run both \Cref{alg: main} and \Cref{alg: eikonal} with twelve random initialisations $m_i,\ i=1,\ldots,12$ generated as 
$$
m_i(x) = \max\left\{0, \sum_{j=1}^5 a_j \sin(b_j \pi x)\right\}, \quad a_j,b_j \sim \mathcal{U}[1,10],
$$
and normalised so that $\int_\Omega m_i \ d x =1$. A graphical representation of those used in this test is shown in \Cref{fig:random_masses}. The corresponding equilibria found by \Cref{alg: main} and \Cref{alg: eikonal} are reported, respectively, in \Cref{fig:final_masses_BR} and \Cref{fig:final_masses_GF}, whereas the iteration count can be found in \Cref{tab:stress_test_linear}. We observe that the initial distribution does influence the number of iterations, but the order of magnitude stays the same.
\begin{figure}[H]
    \centering
    \includegraphics[width=0.7\linewidth]{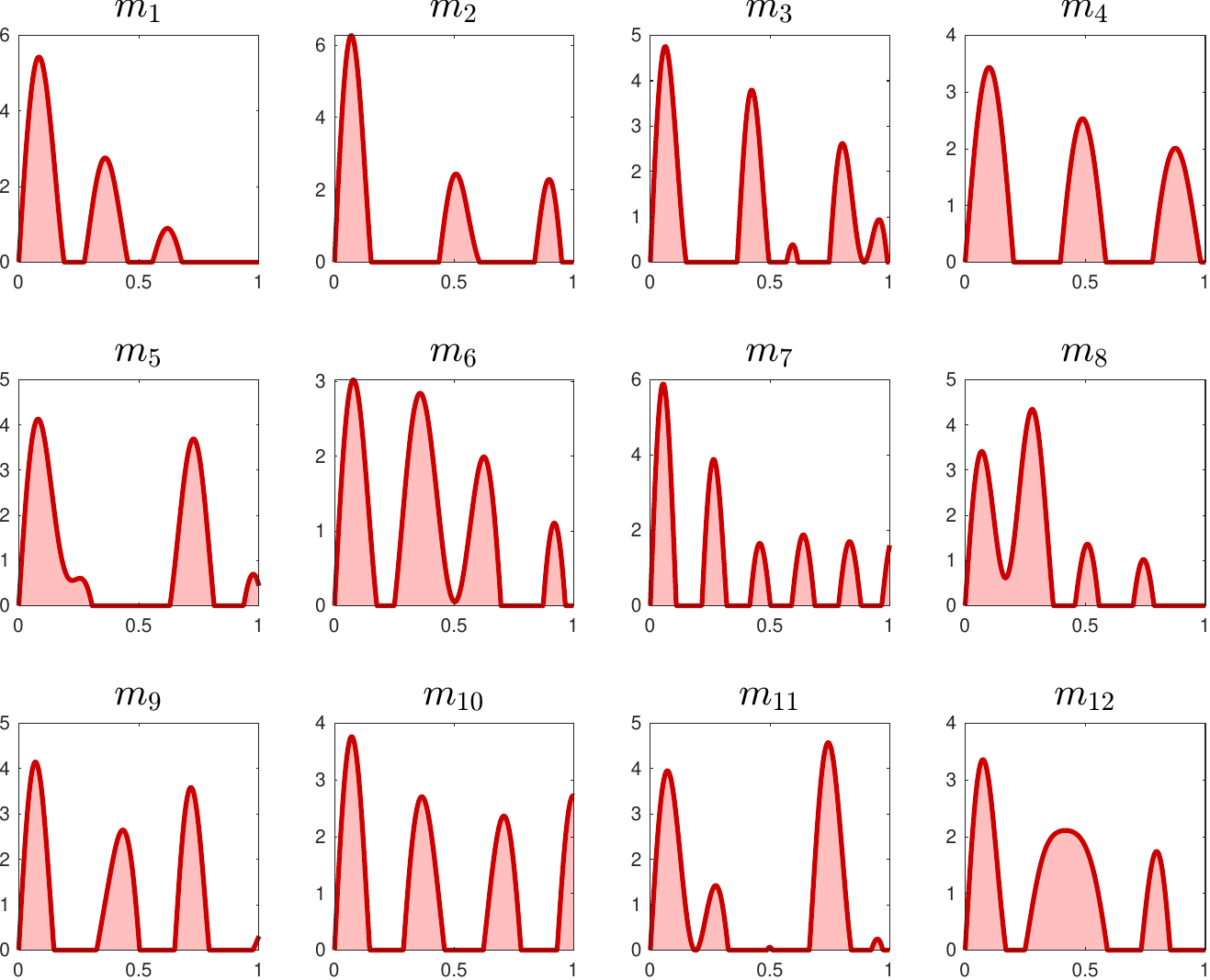}
    \caption{Initialisations $m_i,\ i=1,\ldots,12$ for \Cref{alg: main} and \Cref{alg: eikonal} for the convergence test. }
    \label{fig:random_masses}
\end{figure}

\begin{figure}[H]
    \centering
    \includegraphics[width=0.7\linewidth]{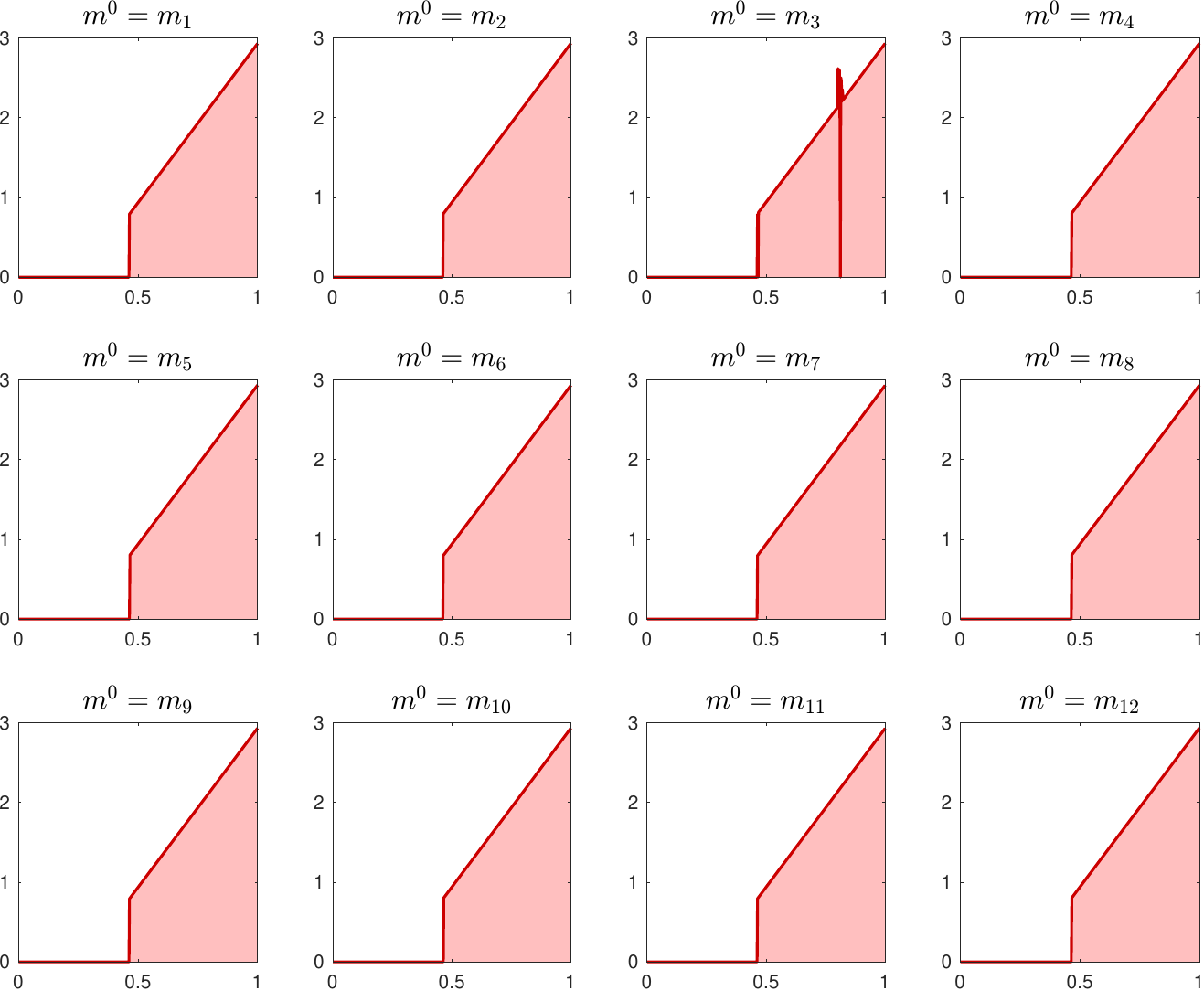}
    \caption{Solutions given by \Cref{alg: main} initialised with the $m_i,\ i=1,\ldots,12$ reported in \Cref{fig:random_masses}.}
    \label{fig:final_masses_BR}
\end{figure}

\begin{figure}[H]
    \centering
    \includegraphics[width=0.7\linewidth]{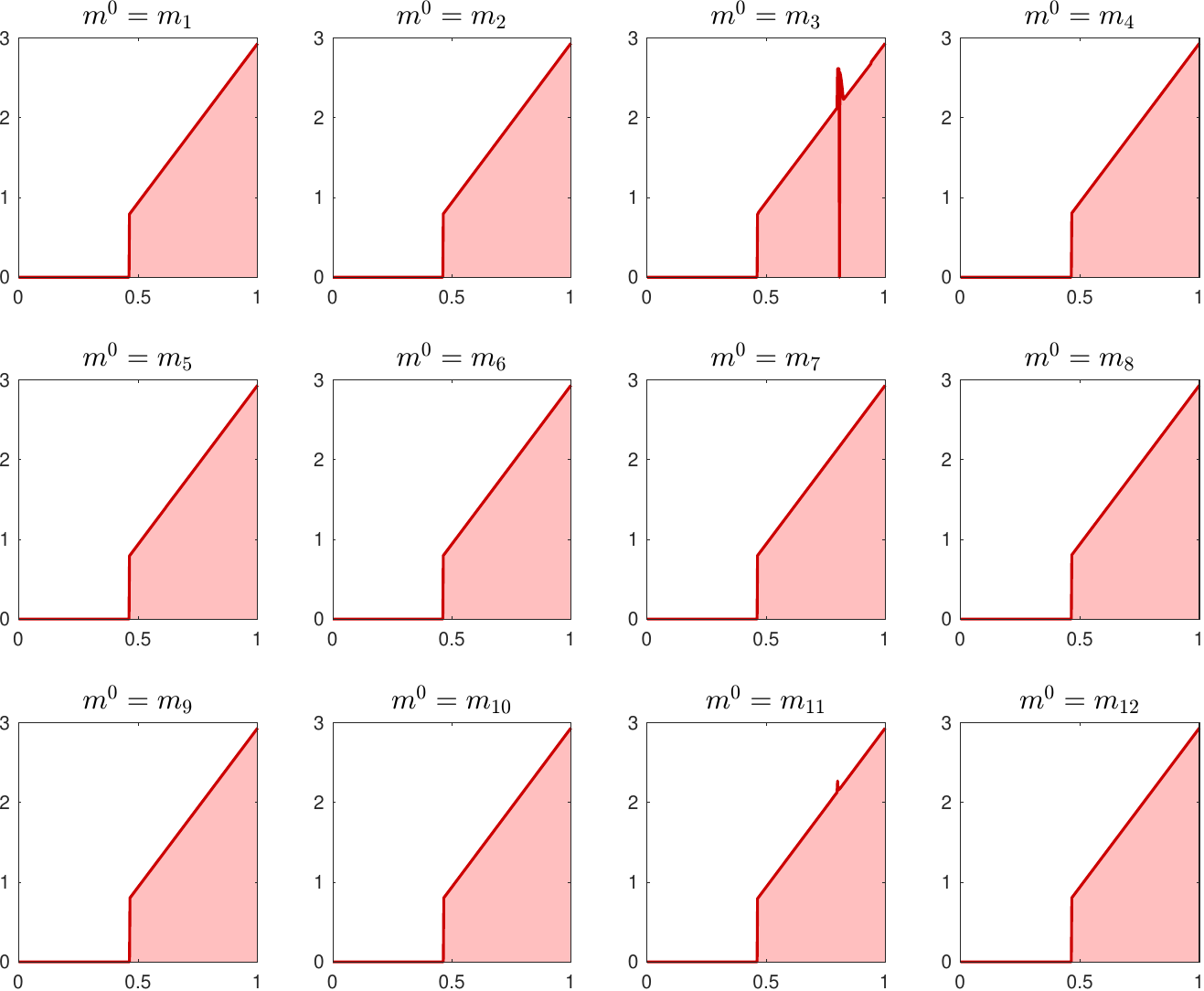}
    \caption{Solutions given by \Cref{alg: eikonal} initialised with the $m_i,\ i=1,\ldots,12$ reported in \Cref{fig:random_masses}.}
    \label{fig:final_masses_GF}
\end{figure}

\begin{table}[H]
    \centering
    \begin{tabular}{r c c}
    \toprule
    \ & \Cref{alg: main} & \Cref{alg: eikonal} \\
    \midrule
    $m_1$ & 8 & 8 \\
    $m_2$ & 8 & 8 \\
    $m_3$ & 12 & 10 \\
    $m_4$ & 7 & 7 \\
    $m_5$ & 7 & 7 \\
    $m_6$ & 7 & 7 \\
    $m_7$ & 7 & 7 \\
    $m_8$ & 8 & 8 \\
    $m_9$ & 7 & 6 \\
    $m_{10}$ & 7 & 7 \\
    $m_{11}$ & 6 & 21 \\
    $m_{12}$ & 8 & 8 \\
    \bottomrule
    \end{tabular}
    \caption{Number of iterations until convergence for \Cref{alg: main} and \Cref{alg: eikonal}  with $f(x)=4x$, initialised with random mass distributions $m_i,\ i=1,\ldots,12$.}
    \label{tab:stress_test_linear}
\end{table}

\subsubsection{The nonlinear case}

Finally, we apply \Cref{alg: main} and \Cref{alg: eikonal} to the bilinear interaction case, i.e. when $\theta[m]$ is the solution of 
\begin{equation} \label{eq:nonlinear_numer}
-\mu \Delta \theta = \theta(K(x)-\theta) - m\theta, \quad \text{with } \mu=0.1,
\end{equation}
which is the second equation in \eqref{eq: ergodic} with an explicit viscosity coefficient $\mu >0$ and $f(x,\theta)=-\theta(K(x)-\theta)$. Regarding its numerical approximation, it is worth noting that, because of our choice of $f$, \eqref{eq:nonlinear_numer} always has at least the trivial solution $\theta \equiv 0$. For this reason, Newton or quasi-Newton iterations are not advisable, as their convergence to one solution or the other is unpredictable \textit{a priori} and highly sensitive to the initialisation. Motivated by \cite{lions1982semilinear}, we opted for the numerical minimisation of the functional
\begin{equation*}
    \int_\Omega \frac{1}{2} |\nabla\theta|^2 - \frac{1}{\mu} F(\theta)\ dx,
\end{equation*}
where $F$ is a primitive of the right-hand side of \eqref{eq:nonlinear_numer} with respect to $\theta$. Furthermore, in this case, we modify Step 8 of \Cref{alg: main} and Step 9 of \Cref{alg: eikonal} to the following:
\begin{algorithmic}
   \State Find $\ C_k \in \mathbb{R}\ $ s.t.
                $$
                    \overline{\theta}_k = \max_{\{\theta_k \geq C_k\} } \theta_k, 
                    \quad
                    \nu_k (x) = \left[\frac{f(x,\overline{\theta}_k)}{\overline{\theta}_k} - m^+(x)\right]\ \chi_{\{\theta_k \geq C_k\} }(x),
                    \quad
                    \int_\Omega \nu_k \ d x = \varepsilon_k
                $$
\end{algorithmic}
in accordance with the different elliptic equation for $\theta$.

First, we repeat the convergence tests also in the non-linear case. In particular, the convergence over time of the functionals to be minimised by \Cref{alg: main} and \Cref{alg: eikonal} is shown in \Cref{fig:funct_conv_NONlinear}, whereas the convergence of the discrete sequence $m^k$ to a gradient flow as the step-size $\varepsilon$ decreases can be seen in \Cref{fig:flow_conv_NONlinear}.
\begin{figure}[H]
    \centering
    \includegraphics[width=0.48\linewidth]{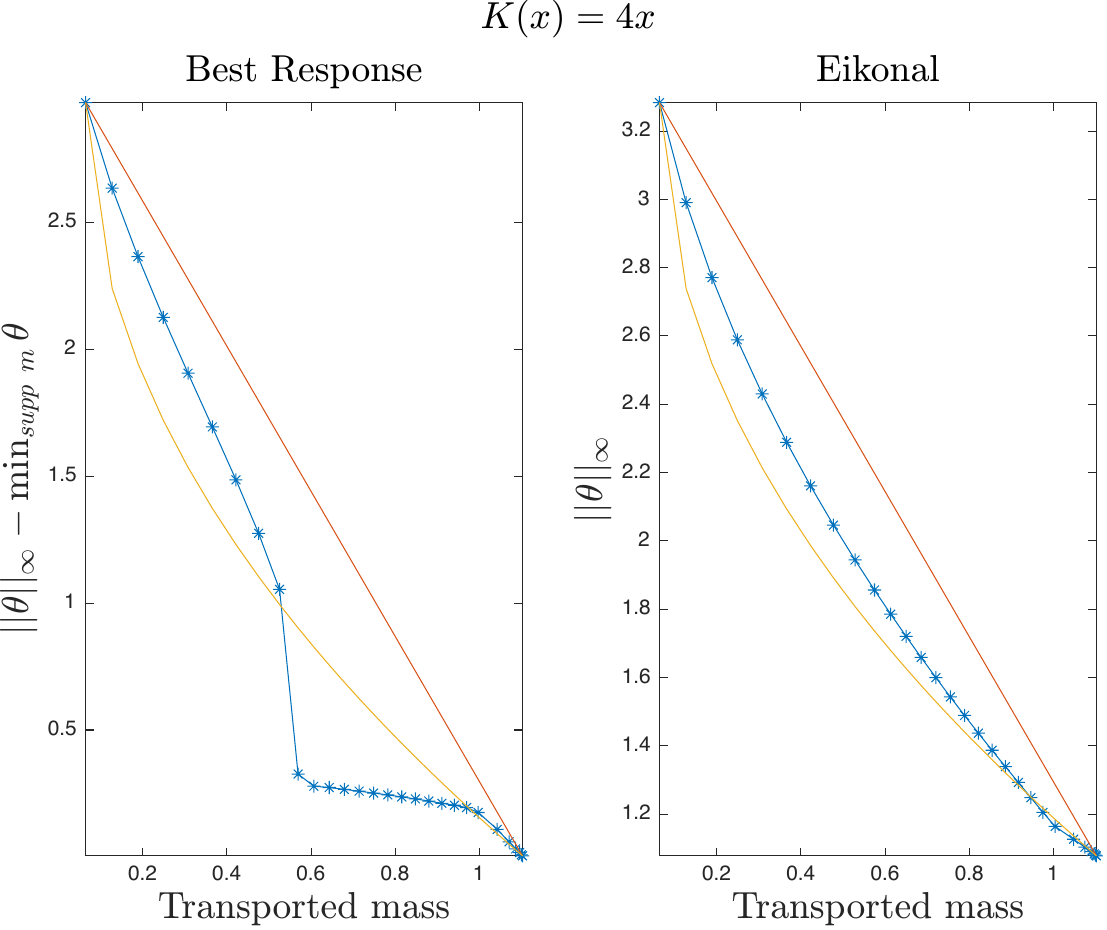} \hfill
    \includegraphics[width=0.48\linewidth]{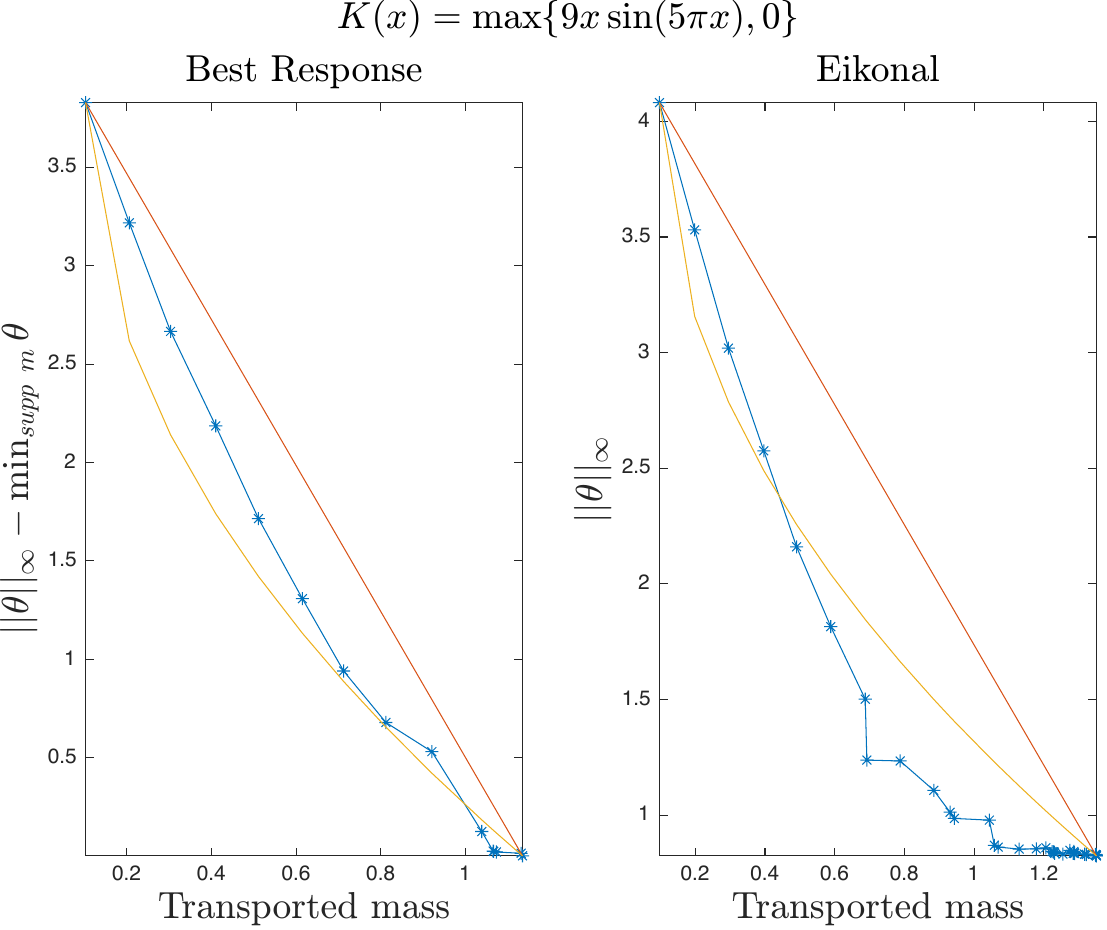}
    \caption{Convergence of the functional to be minimised by \Cref{alg: main} (left) and \Cref{alg: eikonal} (right) in the non-linear case. The dotted blue curve is the one obtained numerically, the red one is a linear interpolant of the two extrema and the yellow one is an interpolant proportional to the square root of the transported mass.}
    \label{fig:funct_conv_NONlinear}
\end{figure}

\begin{figure}[H]
    \centering
   \begin{tabular}{c | c}
    \Cref{alg: main} & \Cref{alg: eikonal} \\
    \includegraphics[width=0.48\linewidth]{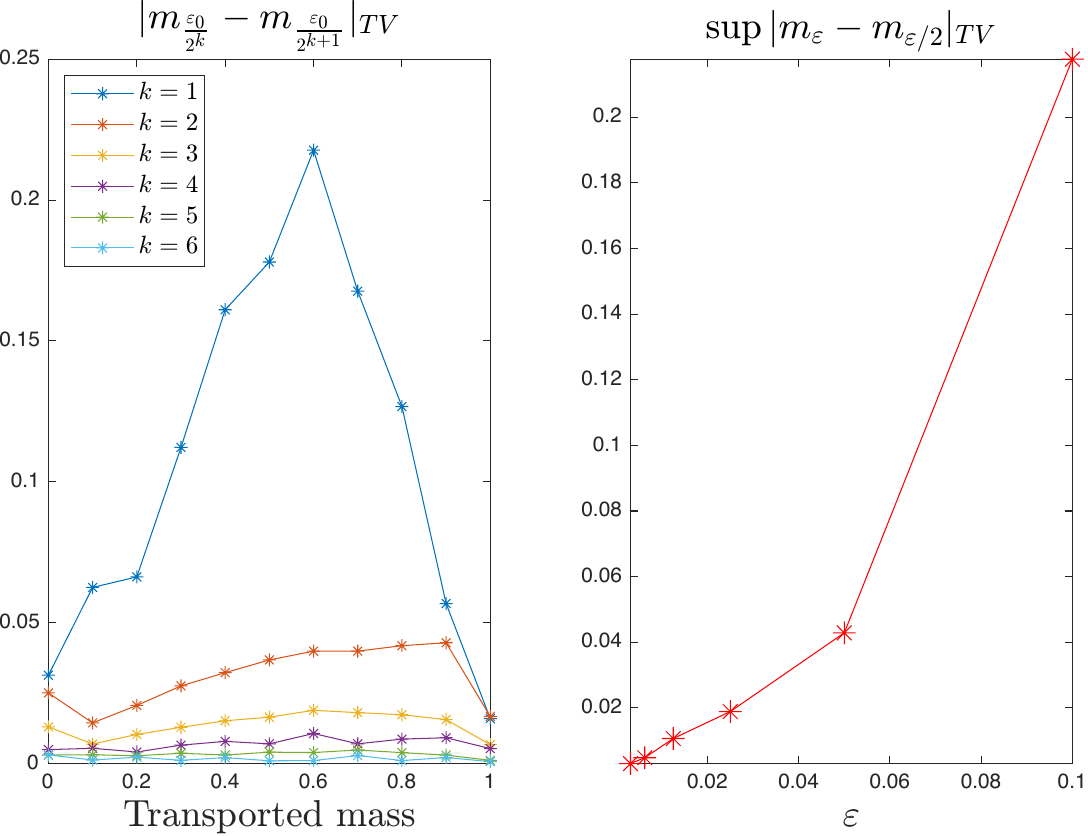} &
    \includegraphics[width=0.48\linewidth]{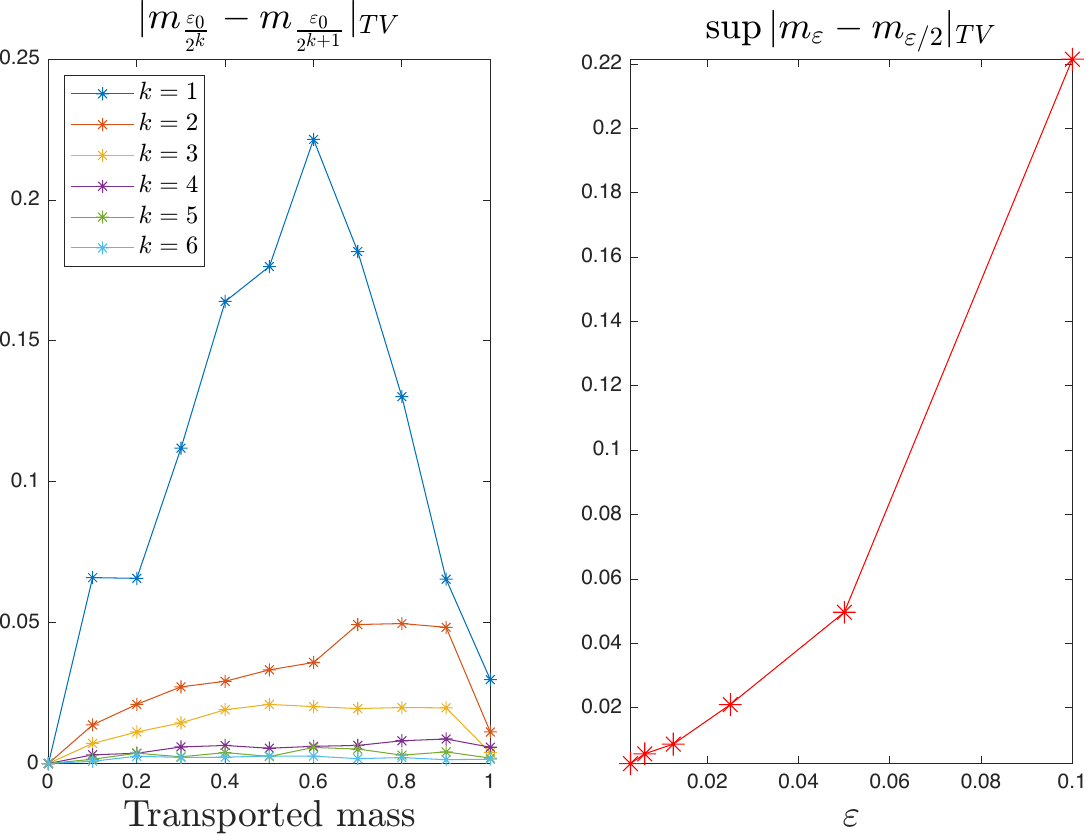} \\
    \end{tabular}
    \caption{Numerical convergence test of the sequence $m^j$ generated by \Cref{alg: main} and \Cref{alg: eikonal} to a gradient flow as the step-size $\varepsilon$ decreases, in the non-linear case.}
    \label{fig:flow_conv_NONlinear}
\end{figure}

Finally, we report the results for the one-dimensional and two-dimensional tests, respectively, in \Cref{tab:1d_NONlinear_sol} and \Cref{tab:2d_NONlinear_sol}, for various choices of $K(x)$. From \Cref{tab:iterations_1d_NONlinear} and \Cref{tab:iterations_2d_NONlinear}, we can see that in this setting the number of iterations needed for convergence is generally higher, compared to the linear case. Nevertheless, both algorithms succeed in finding the $\tau-$Nash equilibrium.

\begin{figure}[H]
    \centering
    \begin{tabular}{c c}
    \Cref{alg: main} & \Cref{alg: eikonal} \\
    \toprule
    \multicolumn{2}{c}{\small $K(x)=4x$} \\
    \includegraphics[width=0.25\linewidth]{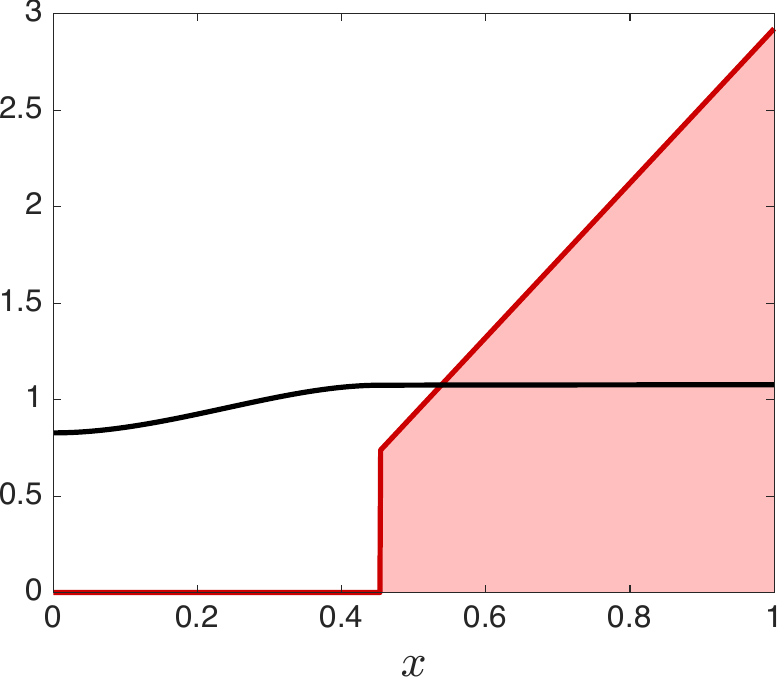} & \includegraphics[width=0.25\linewidth]{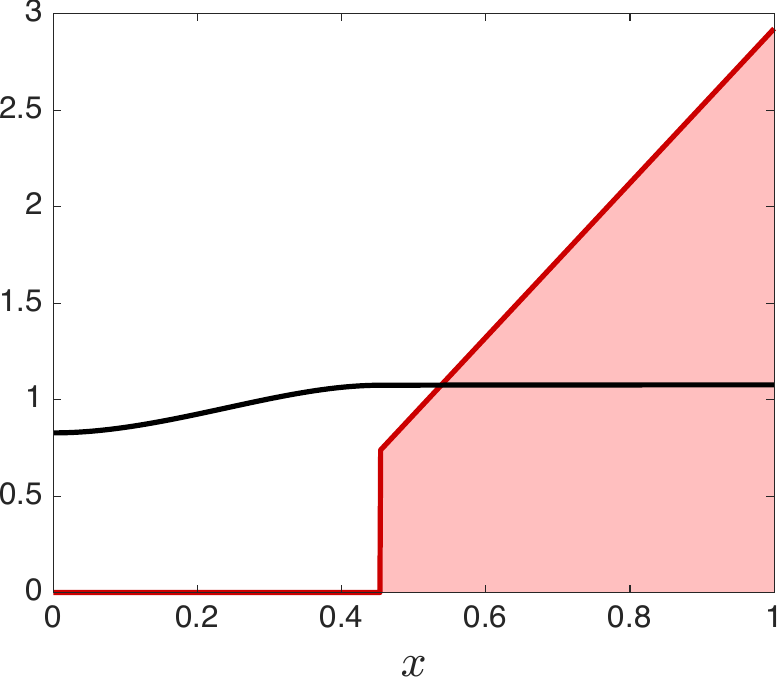} \\
    \multicolumn{2}{c}{\ } \\
    \multicolumn{2}{c}{\small $K(x)=\max(0, 9x\sin(5\pi x))$} \\
    \includegraphics[width=0.25\linewidth]{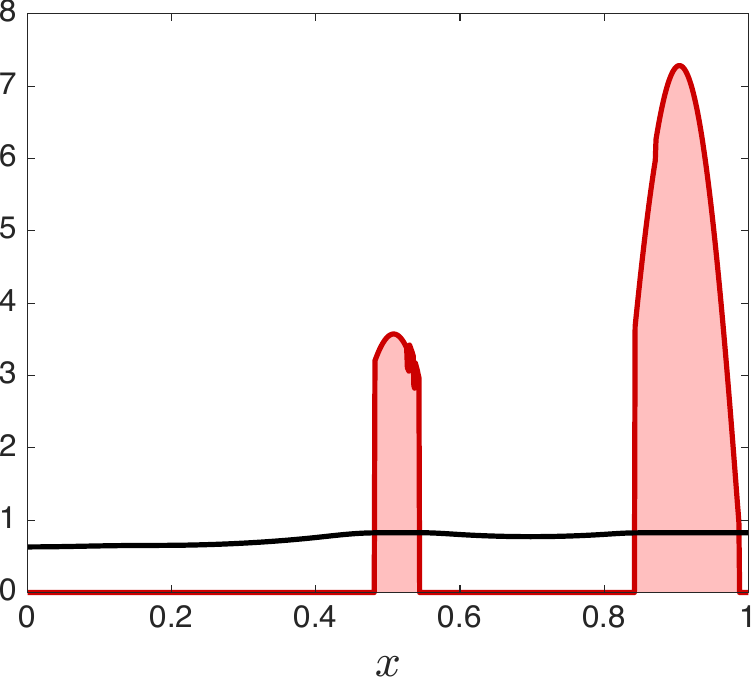} & \includegraphics[width=0.25\linewidth]{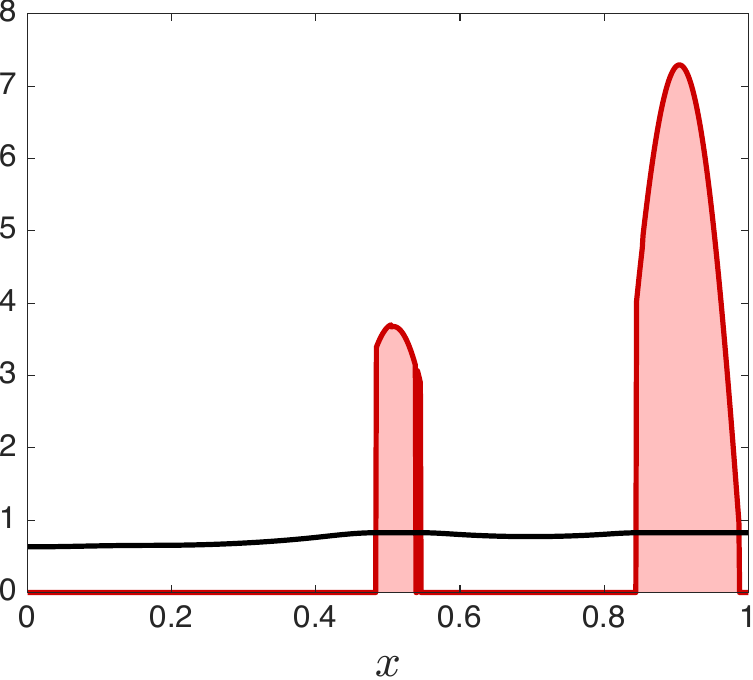} \\
    \multicolumn{2}{c}{\ } \\
    \multicolumn{2}{c}{\small $K(x)=15(\cos(2\pi x)+1)$} \\
    \includegraphics[width=0.25\linewidth]{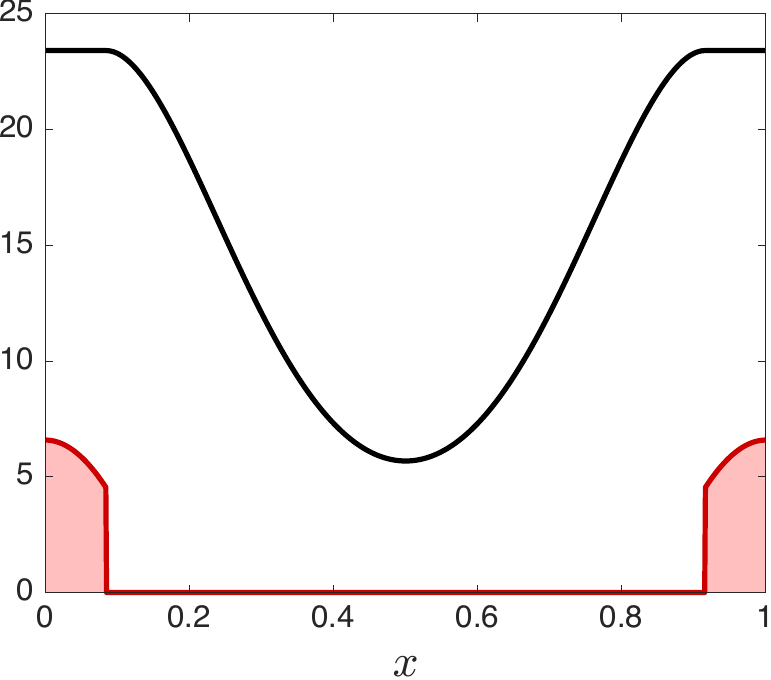} & \includegraphics[width=0.25\linewidth]{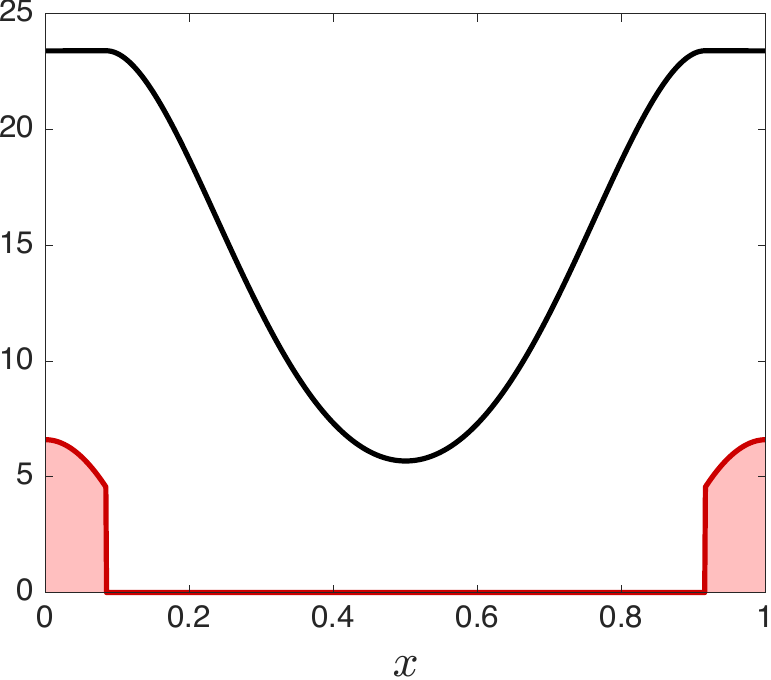} \\
    \bottomrule
    \end{tabular}
    \caption{Solutions in the non linear case given by \Cref{alg: main} (left column) and \Cref{alg: eikonal} (right column) for various choices of $K$ in one dimension. The red curve represents $m(x)$, the black one represents $\theta[m](x)$.}
    \label{tab:1d_NONlinear_sol}
\end{figure}

\begin{table}[H]
    \centering
    \begin{tabular}{r c c}
    \toprule
    \ & \Cref{alg: main} & \Cref{alg: eikonal} \\
    \midrule
    $K(x)=4x$ & 28 iter. & 28 iter. \\ 
    $K(x)=\max(0, 9x\sin(5\pi x))$ & 14 iter. & 24 iter. \\
    $K(x)=15(\cos(2\pi x)+1)$ & 12 iter. & 13 iter. \\
    \bottomrule
    \end{tabular}
    \caption{Number of iterations until convergence for \Cref{alg: main} and \Cref{alg: eikonal} for the simulations in \Cref{tab:1d_NONlinear_sol}.}
    \label{tab:iterations_1d_NONlinear}
\end{table}

\begin{figure}[H]
    \centering
    \begin{tabular}{c c}
    \Cref{alg: main} & \Cref{alg: eikonal} \\
    \toprule
    \multicolumn{2}{c}{\small $K(x,y)=5\exp\left(-\frac{(x-1)^2+(y-1)^2}{0.5}\right)$} \\
    \includegraphics[width=0.3\linewidth]{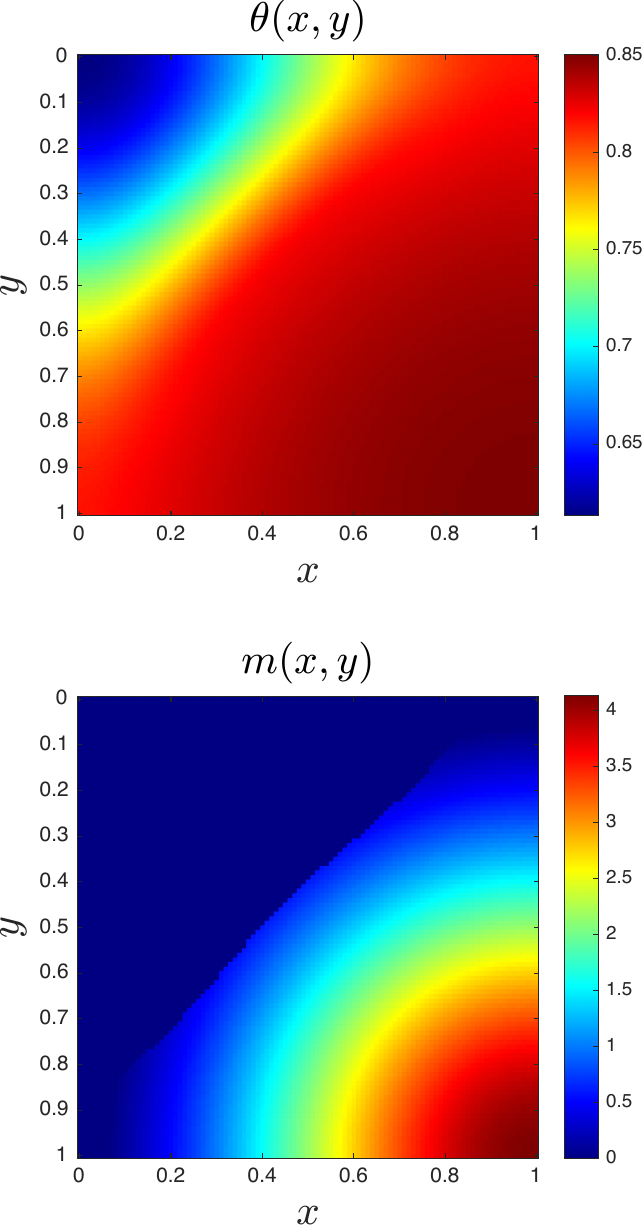} & \includegraphics[width=0.3\linewidth]{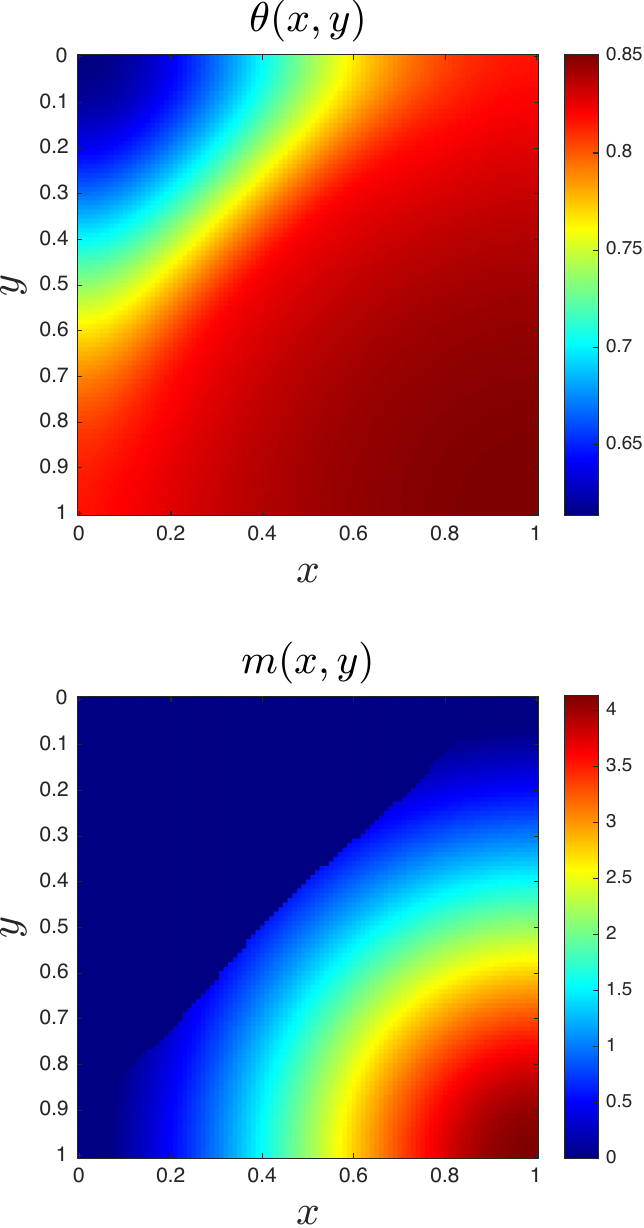} \\
    \multicolumn{2}{c}{\ } \\
    \multicolumn{2}{c}{\small $K(x,y)= \max\left( 0, 4\sum_{i=1}^4 \cos(a_i \pi x)\cos(b_i \pi y) \right), \quad a_i, b_i \sim \mathcal{U}[0,10] $} \\
    \includegraphics[width=0.3\linewidth]{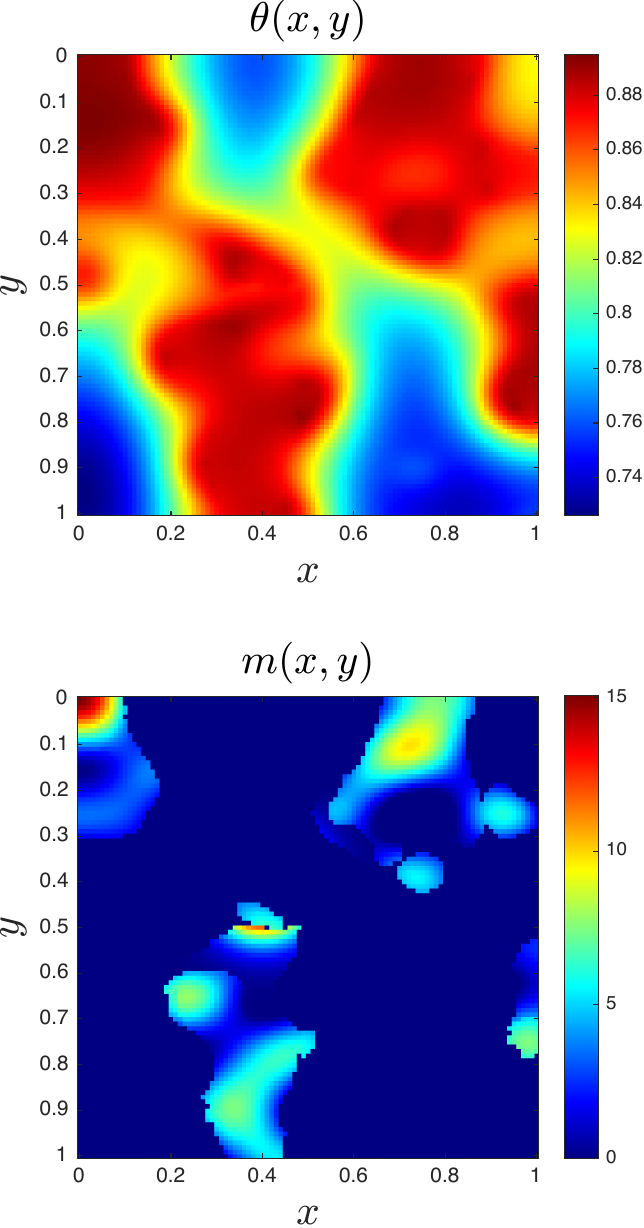} & \includegraphics[width=0.3\linewidth]{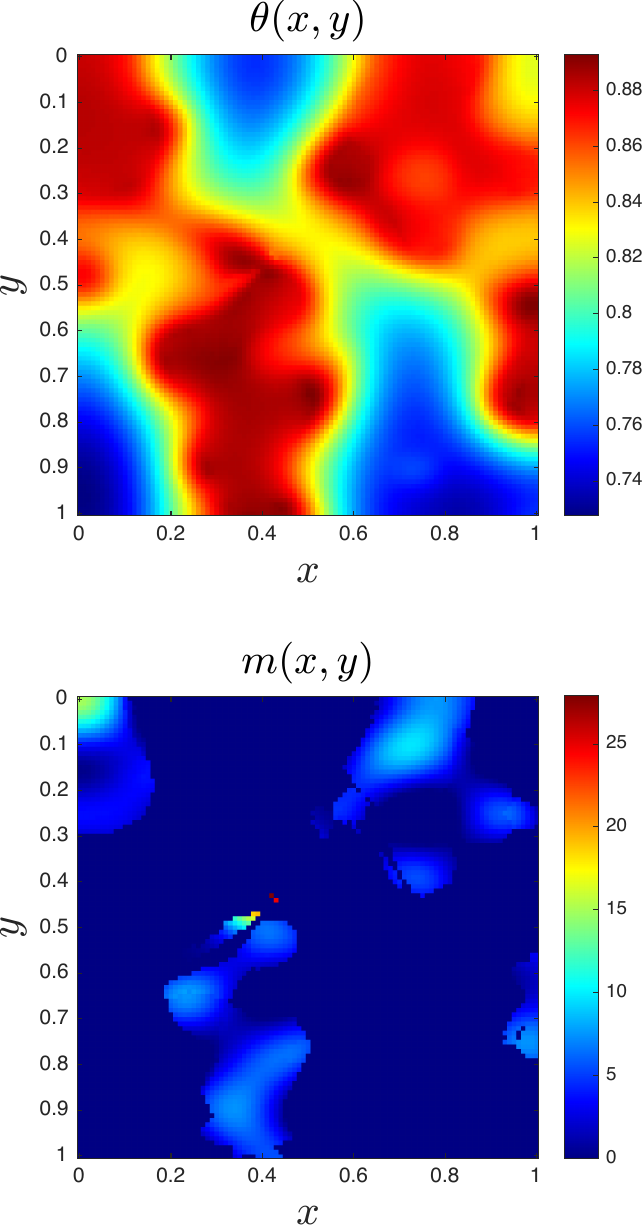} \\
    \bottomrule
    \end{tabular}
    \caption{Solutions for the non linear case given by \Cref{alg: main} (left column) and \Cref{alg: eikonal} (right column) for different choices of $K$ in two dimensions.}
    \label{tab:2d_NONlinear_sol}
\end{figure}

\begin{table}[H]
    \centering
    \begin{tabular}{r c c}
    \toprule
    \ & \Cref{alg: main} & \Cref{alg: eikonal} \\
    \midrule
    {\small $f(x,y)=5\exp\left(-\frac{(x-1)^2+(y-1)^2}{0.5}\right)$} & 31 iter. & 31 iter. \\ 
    {\small \makecell[r]{ $f(x,y)= \max\left( 0, 4\sum_{i=1}^4 \cos(a_i \pi x) \cos(b_i \pi y) \right),$ \\ $a_i, b_i \sim \mathcal{U}[0,10] $}} & 13 iter. & 39 iter. \\
    \bottomrule
    \end{tabular}
    \caption{Number of iterations until convergence for \Cref{alg: main} and \Cref{alg: eikonal} for the simulations in \Cref{tab:2d_NONlinear_sol}. For these numerical tests, the maximum $\varepsilon$ allowed was $0.25$.}
    \label{tab:iterations_2d_NONlinear}
\end{table}

\section{Conclusions and future perspectives}

\bibliographystyle{alpha}
\bibliography{biblio}

@article{bernhard1995theorem,
  title={On a theorem of {Danskin} with an application to a theorem of {Von Neumann}--{Sion}},
  author={Bernhard, Pierre and Rapaport, Alain},
  journal={Nonlinear Analysis: Theory, Methods \& Applications},
  volume={24},
  number={8},
  pages={1163--1181},
  year={1995},
  publisher={Pergamon}
}

@article{jordan1998variational,
  title={The variational formulation of the {Fokker}--{Planck} equation},
  author={Jordan, Richard and Kinderlehrer, David and Otto, Felix},
  journal={SIAM journal on mathematical analysis},
  volume={29},
  number={1},
  pages={1--17},
  year={1998},
  publisher={SIAM}
}

@article{chambolle20231,
  title={$L^1$-Gradient Flow of Convex Functionals},
  author={Chambolle, Antonin and Novaga, Matteo},
  journal={SIAM Journal on Mathematical Analysis},
  volume={56},
  number={5},
  pages={5747--5781},
  year={2024},
  publisher={SIAM}
}

@book{ambrosio2005gradient,
  title={Gradient flows: in metric spaces and in the space of probability measures},
  author={Ambrosio, Luigi and Gigli, Nicola and Savar{\'e}, Giuseppe},
  year={2005},
  publisher={Springer Science \& Business Media}
}

@article{de1993new,
  title={New problems on minimizing movements},
  author={De Giorgi, Ennio},
  journal={Ennio de Giorgi: Selected Papers},
  pages={699--713},
  year={1993},
  publisher={Springer}
}

@article{monderer1996potential,
  title={Potential games},
  author={Monderer, Dov and Shapley, Lloyd S},
  journal={Games and economic behavior},
  volume={14},
  number={1},
  pages={124--143},
  year={1996},
  publisher={Elsevier}
}

@article{rosenthal1973class,
  title={A class of games possessing pure-strategy Nash equilibria},
  author={Rosenthal, Robert W},
  journal={International Journal of Game Theory},
  volume={2},
  pages={65--67},
  year={1973},
  publisher={Physica-Verlag}
}

@article{benamou2017variational,
  title={Variational mean field games},
  author={Benamou, Jean-David and Carlier, Guillaume and Santambrogio, Filippo},
  journal={Active Particles, Volume 1: Advances in Theory, Models, and Applications},
  pages={141--171},
  year={2017},
  publisher={Springer}
}

@article{liero2018optimal,
  title={Optimal entropy-transport problems and a new Hellinger--Kantorovich distance between positive measures},
  author={Liero, Matthias and Mielke, Alexander and Savar{\'e}, Giuseppe},
  journal={Inventiones mathematicae},
  volume={211},
  number={3},
  pages={969--1117},
  year={2018},
  publisher={Springer}
}

@article{gallouet2017jko,
  title={A JKO splitting scheme for Kantorovich--Fisher--Rao gradient flows},
  author={Gallou{\"e}t, Thomas O and Monsaingeon, Leonard},
  journal={SIAM Journal on Mathematical Analysis},
  volume={49},
  number={2},
  pages={1100--1130},
  year={2017},
  publisher={SIAM}
}

@article{kobeissi2023tragedy,
  title={The tragedy of the commons: A Mean-Field Game approach to the reversal of travelling waves},
  author={Kobeissi, Ziad and Mazari-Fouquer, Idriss and Ruiz-Balet, Dom{\`e}nec},
  journal={Nonlinearity},
  volume={37},
  number={11},
  pages={115010},
  year={2024},
  publisher={IOP Publishing}
}

@article{kobeissi2024mean,
  title={Mean-field games for harvesting problems: Uniqueness, long-time behaviour and weak KAM theory},
  author={Kobeissi, Ziad and Mazari-Fouquer, Idriss and Ruiz-Balet, Dom{\`e}nec},
  journal={arXiv preprint arXiv:2406.06057},
  year={2024}
}

@article{briceno2018proximal,
  title={Proximal methods for stationary mean field games with local couplings},
  author={Briceno-Arias, Luis M and Kalise, Dante and Silva, Francisco J},
  journal={SIAM Journal on Control and Optimization},
  volume={56},
  number={2},
  pages={801--836},
  year={2018},
  publisher={SIAM}
}

@article{lasry2007mean,
  title={Mean field games},
  author={Lasry, Jean-Michel and Lions, Pierre-Louis},
  journal={Japanese journal of mathematics},
  volume={2},
  number={1},
  pages={229--260},
  year={2007},
  publisher={Springer}
}

@article{huang2006large,
  title={Large population stochastic dynamic games: closed-loop McKean-Vlasov systems and the Nash certainty equivalence principle},
  author={Huang, Minyi and Malham{\'e}, Roland P and Caines, Peter E},
  year={2006}
}

@article{cardaliaguet2013long,
  title={Long time average of first order mean field games and weak KAM theory},
  author={Cardaliaguet, Pierre},
  journal={Dynamic Games and Applications},
  volume={3},
  pages={473--488},
  year={2013},
  publisher={Springer}
}

@article{lions1982semilinear,
 author = {Lions, Pierre-Louis},
 journal = {SIAM Review},
 number = {4},
 pages = {441--467},
 publisher = {Society for Industrial and Applied Mathematics},
 title = {On the Existence of Positive Solutions of Semilinear Elliptic Equations},
 volume = {24},
 year = {1982}
}

@book{falcone2013sl,
  title={{Semi-Lagrangian} approximation schemes for linear and {Hamilton—Jacobi} equations},
  author={Falcone, Maurizio and Ferretti, Roberto},
  year={2013},
  publisher={SIAM}
}

@book{achdou2021mean,
  title={Mean Field Games: Cetraro, Italy 2019},
  author={Achdou, Yves and Cardaliaguet, Pierre and Delarue, Fran{\c{c}}ois and Porretta, Alessio and Santambrogio, Filippo},
  volume={2281},
  year={2021},
  publisher={Springer Nature}
}

@article{santambrogio2015optimal,
  title={Optimal transport for applied mathematicians},
  author={Santambrogio, Filippo},
  journal={Birk{\"a}user, NY},
  volume={55},
  number={58-63},
  pages={94},
  year={2015},
  publisher={Springer}
}

@article{bardi2024long,
  title={Long-time behavior of deterministic mean field games with nonmonotone interactions},
  author={Bardi, Martino and Kouhkouh, Hicham},
  journal={SIAM Journal on Mathematical Analysis},
  volume={56},
  number={4},
  pages={5079--5098},
  year={2024},
  publisher={SIAM}
}

@article{geshkovski2022turnpike,
  title={Turnpike in optimal control of PDEs, ResNets, and beyond},
  author={Geshkovski, Borjan and Zuazua, Enrique},
  journal={Acta Numerica},
  volume={31},
  pages={135--263},
  year={2022},
  publisher={Cambridge University Press}
}

@book{lam2022introduction,
  title={Introduction to reaction-diffusion equations: Theory and applications to spatial ecology and evolutionary biology},
  author={Lam, King-Yeung and Lou, Yuan},
  year={2022},
  publisher={Springer Nature}
}

@book{cantrell2004spatial,
  title={Spatial ecology via reaction-diffusion equations},
  author={Cantrell, Robert Stephen and Cosner, Chris},
  year={2004},
  publisher={John Wiley \& Sons}
}

@book{fife2013mathematical,
  title={Mathematical aspects of reacting and diffusing systems},
  author={Fife, Paul C},
  volume={28},
  year={2013},
  publisher={Springer Science \& Business Media}
}

@article{mazari2022spatial,
  title={Spatial ecology, optimal control and game theoretical fishing problems},
  author={Mazari, Idriss and Ruiz-Balet, Dom{\`e}nec},
  journal={Journal of Mathematical Biology},
  volume={85},
  number={5},
  pages={55},
  year={2022},
  publisher={Springer}
}

@article{bressan2013multidimensional,
  title={A multidimensional optimal-harvesting problem with measure-valued solutions},
  author={Bressan, Alberto and Coclite, Giuseppe Maria and Shen, Wen},
  journal={SIAM Journal on Control and Optimization},
  volume={51},
  number={2},
  pages={1186--1202},
  year={2013},
  publisher={SIAM}
}

@article{bressan2019competitive,
  title={On the competitive harvesting of marine resources},
  author={Bressan, Alberto and Staicu, Vasile},
  journal={SIAM Journal on Control and Optimization},
  volume={57},
  number={6},
  pages={3961--3984},
  year={2019},
  publisher={SIAM}
}

@article{sethian1999fast,
     author = {J. A. Sethian},
     journal = {SIAM Review},
     number = {2},
     pages = {199--235},
     publisher = {Society for Industrial and Applied Mathematics},
     title = {Fast Marching Methods},
     volume = {41},
     year = {1999}
}

@article{laux2016thresholding,
  title = {Convergence of the thresholding scheme for multi-phase mean-curvature flow},
  author = {Laux, Tim and Otto, Felix},
  journal = {Calculus of Variations and Partial Differential Equations},
  volume = {55},
  number = {5},
  pages = {129},
  year = {2016},
  doi = {10.1007/s00526-016-1053-0}
}

@article{laux2020degiorgi,
  title = {The thresholding scheme for mean curvature flow and {de Giorgi}'s ideas for minimizing movements},
  author = {Laux, Tim and Otto, Felix},
  journal = {Advanced Studies in Pure Mathematics},
  booktitle = {The Role of Metrics in the Theory of Partial Differential Equations},
  volume = {85},
  pages = {63--93},
  year = {2020},
  publisher = {Mathematical Society of Japan},
  address = {Tokyo},
  doi = {10.2969/aspm/08510063}
}

@article{laux2020brakke,
  title = {Brakke's inequality for the thresholding scheme},
  author = {Laux, Tim and Otto, Felix},
  journal = {Calculus of Variations and Partial Differential Equations},
  volume = {59},
  number = {1},
  pages = {39},
  year = {2020},
  doi = {10.1007/s00526-020-1696-8}
}

@article{laux2018gradient,
  title = {Gradient-flow techniques for the analysis of numerical schemes for multi-phase mean-curvature flow},
  author = {Laux, Tim},
  journal = {Geometric Flows},
  volume = {3},
  number = {1},
  pages = {76--89},
  year = {2018},
  doi = {10.1515/geofl-2018-0006}
}

@article{laux2017bulk,
  title = {Convergence of thresholding schemes incorporating bulk effects},
  author = {Laux, Tim and Swartz, Drew},
  journal = {Interfaces and Free Boundaries},
  volume = {19},
  number = {2},
  pages = {273--304},
  year = {2017},
  doi = {10.4171/IFB/383}
}

@article{chambolle2021mullins,
  title = {Mullins--Sekerka as the {Wasserstein} flow of the perimeter},
  author = {Chambolle, Antonin and Laux, Tim},
  journal = {Proceedings of the American Mathematical Society},
  volume = {149},
  number = {7},
  pages = {2943--2956},
  year = {2021},
  doi = {10.1090/proc/15401}
}

@article{cardaliaguet2017learning,
  title = {Learning in mean field games: {The} fictitious play},
  author = {Cardaliaguet, Pierre and Hadikhanloo, Saeed},
  journal = {ESAIM: Control, Optimisation and Calculus of Variations},
  volume = {23},
  number = {2},
  pages = {569--591},
  year = {2017},
  doi = {10.1051/cocv/2016004}
}

@article{graber2025remarks,
  title = {Remarks on potential mean field games},
  author = {Graber, P. Jameson},
  journal = {Research in the Mathematical Sciences},
  volume = {12},
  number = {1},
  pages = {13},
  year = {2025},
  doi = {10.1007/s40687-024-00494-3}
}

\begin{minipage}[t]{.5\textwidth}
{\footnotesize{\bf Dante Kalise}\par
 Department of Mathematics\par
  Imperial College London\par
  Exhibition Rd, South Kensington,\par
  London SW7 2AZ, United Kingdom
 \par
  e-mail: {\scriptsize dante.kalise-balza@imperial.ac.uk}
  }
\end{minipage}
\begin{minipage}[t]{.5\textwidth}
  {\footnotesize{\bf Alessio Oliviero}\par
  MOX, Department of Mathematics\par
   Politecnico di Milano\par
   Via Bonardi, 9\par
   20133 Milan, Italy\par
 \par
  e-mail: {\scriptsize alessio.oliviero@polimi.it}
  }
\end{minipage}%

\begin{center}
\begin{minipage}[t]{.5\textwidth}
  {\footnotesize{\bf Domènec Ruiz-Balet}\par
  Departament de Matemàtiques i Informàtica\par
  Universitat de Barcelona\par
  Gran Via de les Corts Catalanes, 585\par
  08007 Barcelona, Spain\par
 \par
  e-mail: \scriptsize domenec.ruizibalet@ub.edu}
\end{minipage}%
\end{center}

\end{document}